\documentclass[reqno,11pt,a4paper]{amsart}
\usepackage[utf8]{inputenc}
\usepackage{pifont}
\usepackage[T1]{fontenc}
\usepackage{enumerate}
\usepackage{textcomp}

\usepackage{pgffor}

\usepackage{esint}
\usepackage{pdfsync}
\usepackage[colorlinks=true,linkcolor=blue]{hyperref}%
\usepackage[T1]{fontenc}
\usepackage{pifont}
\usepackage[english]{babel}

\usepackage{ulem}
\usepackage{amsmath,esint,color}
\usepackage{amsthm,latexsym,epsfig,graphicx,marvosym, mathrsfs,bigints,stmaryrd,textgreek,upgreek}
\usepackage{amsmath,stmaryrd}
\usepackage{relsize}
\usepackage{amsfonts}
\usepackage{amssymb}
\usepackage[headings]{fullpage}

\allowdisplaybreaks

\usepackage{amsmath}
\usepackage{amsfonts}
\usepackage{amssymb}
\usepackage{amsthm}

\newcommand{\NN}{\mathbb{N}}
\newcommand{\TT}{\mathbb{T}}

\newcommand{\RR}{\mathbb{R}}
\newcommand{\DD}{\mathbb{D}}

\newtheorem{theorem}{Theorem}[section]
\newtheorem{proposition}{Proposition}[section]
\newtheorem{lemma}{Lemma}[section]
\newtheorem{remark}{Remark}[section]

\newcommand{\dd}{\mathrm{d}}

\begin{document}

\title[Time periodic solutions for gSQG equation in the disc]{Emergence of time periodic solutions for the generalized surface quasi-geostrophic equation in the disc}
\author{Taoufik Hmidi}
\address{IRMAR, Universit\'e de Rennes 1, Campus de Beaulieu, 35042 Rennes cedex, France}
\email{thmidi@univ-rennes1.fr}
\author{Liutang Xue}
\address{Laboratory of Mathematics and Complex Systems (MOE), School of Mathematical Sciences, Beijing Normal University, Beijing 100875, P.R. China}
\email{xuelt@bnu.edu.cn}
\author{Zhilong Xue}
\address{Laboratory of Mathematics and Complex Systems (MOE), School of Mathematical Sciences, Beijing Normal University, Beijing 100875, P.R. China}
\email{zhilongxue@mail.bnu.edu.cn}

{\thanks{T. Hmidi has been supported by Tamkeen under the NYU Abu Dhabi Research Institute grant.
L. Xue has been partially supported by National Key Research and Development Program of China (No. 2020YFA0712900) and National Natural Science Foundation of China (Nos. 11771043, 12271045).}}


\maketitle
\begin{abstract}
In this paper we address the existence of time periodic solutions for the generalized inviscid SQG equation in the unit disc  with homogeneous Dirichlet boundary condition when  $\alpha\in (0,1)$. We show the existence of a countable family of bifurcating curves from the radial patches.  In contrast with the preceding studies in active scalar equations,  the Green function  is no longer  explicit  and we circumvent this issue by  a suitable splitting  into a singular explicit part (which coincides with the planar one) and a smooth implicit  one induced by the boundary of the domain. Another problem is connected to the analysis of the linear frequencies which admit a complicated form through a discrete sum involving  Bessel functions and their zeros. We overcome this difficulty by using Sneddon's formula leading to a suitable integral representation of the  frequencies.     
\end{abstract}

\tableofcontents

\section{Introduction}
In this paper we investigate some special structures of the vortical motions for the inviscid generalized surface quasi-geostrophic (abbr. gSQG) equation in the unit disc $\mathbb{D}\subset \mathbb{R}^2$. This model describes the evolution of the potential temperature $\theta$
governed by the transport equation,
\begin{equation}\label{Eq-gSQG}
\left\{ \begin{array}{ll}
  \partial_{t}\theta+u\cdot\nabla\theta=0,\qquad\quad(t,x)\in[0,\infty) \times\mathbb{D}, &\\
  u=-\nabla^\perp(-\Delta)^{-1+\frac{\alpha}{2}}\theta, \\
  \theta_{|t=0}(x) = \theta_0(x).
\end{array} \right.
\end{equation}
Here $u=(u_1,u_2)$ refers to the velocity field, $\nabla^{\perp}=(-\partial_2,\partial_1)$,
$\alpha\in [0,1)$ is a real parameter. The fractional Laplacian operator $(-\Delta)^{-1+\frac{\alpha}{2}}$
is defined via the eigenfunction expansion of the Laplacian in $\DD$ with homogeneous Dirichlet boundary condition
(see \eqref{Kern-For0}-\eqref{Kern-For} below). This model was introduced in \cite{CFMR} for  the flat case $\RR^2$ as an interpolation between 2D Euler equation and the surface quasi-geostrophic (abbr. SQG) model, corresponding to $\alpha=0$ and $\alpha=1$ in \eqref{Eq-gSQG}, respectively.
Notice that the SQG equation was used as a simplified model to track the atmospheric circulation near the the tropopause  \cite{Juck94,HPGS} and the ocean dynamics in the upper layers \cite{LK06}. A strong analogy  with the vorticity formulation of the  3D incompressible Euler equations was discussed in \cite{CMT94}.

These aforementioned active scalar equations have attracted a lot of attention in the past decades and important progress has been settled in various directions. As to
the local well-posedness of classical solutions in the whole space, it was  performed in various function spaces. For instance, we  refer to \cite{CCCGW} where the solutions are constructed in the framework of Sobolev spaces. The global well-posedness issue is still  open except for the Euler case $\alpha=0$. However,
 $L^2$-weak solutions in the whole space  are known to be global, see
\cite{Resnick,Mar08,LX19}. The nonuniqueness of weak solutions to SQG equation  has been explored recently  in \cite{Buck, Isett21}.
Another class of solutions widely discussed in the literature  is described by the patches where the initial data takes the form  of the characteristic function of a smooth bounded domain $D,$ that is, $\theta_0={\bf{1}}_{D}.$ In this case, the patch  structure is  preserved for a short time  and the boundary evolves according to a suitable contour dynamics equation, see \cite{CCCGW,Rod05,Gan08}.
Similar studies have been achieved  for a half plane
\cite{KRYZ,KYZ17,GanP21} and for any   smooth bounded \mbox{domain \cite{Kiselev19}.}
The global in time persistence of the boundary regularity is only known for the case $\alpha=0$
according to Chemin's  result \cite{Chemin93}, see also Bertozzi and Constantin \cite{BerC93} for another proof.
Notice that  some numerical experiments show strong evidence for the singularity formation in finite time,
see for instance \cite{CFMR,SD14,SD19}. 
For the patch problem associated to  gSQG equation, a  finite-time singularity result with multi-signed patches has been established by Kiselev et al \cite{KRYZ} for the case $0<\alpha < \frac{1}{12}$
and Gancedo et al \cite{GanP21} for $0<\alpha <\frac{1}{3}$.

The analysis of SQG type equations  in bounded smooth domains  is much involved than the flat case due in part to the  Green function which is not explicit. This study    was initiated by Constantin and Ignatova \cite{CI16a,CI16b}.
They considered the SQG equation with critical dissipation and obtained the global existence of $L^2$-weak solutions with a global Lipschitz a priori interior estimates. We also refer to the papers  \cite{CI20,Ign19,SV20} for  more results and discussions.
Concerning  the inviscid model \eqref{Eq-gSQG} in smooth bounded domains, the $L^2$ global weak solutions was constructed by Constantin and Nguyen \cite{ConN18b}
for the SQG case $\alpha=1$,
and later generalized in \cite{NHQ18} to the case $\alpha\in (1, 2)$ with more singular constitutive law in the velocity.
The local well-posedness issue  in the framework of classical solutions for the inviscid SQG equation (in bounded smooth domains) was performed in \cite{ConN18a}. We point out that
with  some slight modification, the results of \cite{ConN18a,ConN18b,NHQ18}
can be extended to the gSQG equation with $\alpha\in (0,1)$.

The aim of this work is to construct time periodic  solutions in the patch form  for  the  gSQG model in bounded smooth domains.
We will in particular focus on the class of V-states or rotating patches,
whose dynamics is described by a rigid body transformation. In this setting,
the problem reduces to finding some domains $D$ subject to uniform rotation around their centers of mass.
Observe that during the motion the support $D_t$ of the  patch solution  does not change the shapes and is determined by
$D_t = \mathbf{R}_{x_0,\Omega t} \,D$, where $ \mathbf{R}_{x_0,\Omega t}$ denotes the planar rotation with center $x_0$
and angle $\Omega t$. The parameter $\Omega \in \mathbb{R}$ is called the angular velocity of the rotating domain.


The V-states problem has a long history and it is still  the subject of an intensive research, and many important contributions have been achieved in the  last few decades at the analytical and numerical levels. The first example of rotating patches  for the 2D Euler equation in the plane  was discovered by Kirchhoff \cite{Kirch} who proved that an ellipse of semi-axes $a$ and $b$ rotates perpetually with the uniform angular velocity
$\Omega = \tfrac{ab}{(a+b)^2}$, we also refer to \cite[p. 232]{Lamb} and \cite[p. 304]{MB02}.
One century later, Deem and Zabusky \cite{DZ78} 
gave numerical evidence of the existence of  implicit V-states with  $m$-fold symmetry.
Afterwards, \mbox{Burbea \cite{Burbea82}} provided  an analytical proof of this fact using local bifurcation theory and conformal parametrization.  In particular, he proved that for each symmetry $m\geqslant2$ a curve of non trivial V-states bifurcates from Rankine vorticity (the radial shape) at  the angular velocities $\Omega = \frac{m-1}{2m}$.
See also Hmidi, Mateu and Verdera \cite{HMV13}, where the $C^{\infty}$ boundary regularity and the convexity of these bifurcated V-states are established.
Real analyticity of the boundary was further obtained by Castro, C\'ordoba and G\'omez-Serrano \cite{CCG16b}.
Hassainia, Masmoudi and Wheeler in \cite{HMW20}
 studied through some global bifurcation arguments the analyticity of the V-states along the whole  bifurcating branches. Besides the preceding results,  several  families  of V-states with different topological structures were recently explored. For instance, it was shown in \cite{CCG16b,HM16} that  a second family of countable branches bifurcate  from   Kirchhoff's ellipses. Rotating patches with only one hole exist near the annulus as proved in
 \cite{DHMV16,HM16b}, concentrated multi vortices centered at regular n-gons or distributed according to suitable periodic spatial  patterns are analyzed in \cite{HM17,Hassainia-Mil,Garcia20,Garcia21}. The  study of V-states in radial domains was performed for Euler equation in   \cite{DHHM}.
In particular, De la Hoz, Hassainia, Hmidi and  Mateu  \cite{DHHM} proved the existence of
$m$-fold symmetric V-states for the Euler equation in the unit disc, which bifurcates from the trivial solution
${\bf{1}}_{b \DD}$ at the angular velocity $\Omega = \frac{m-1+b^{2m}}{2m}$ for any  $m\geqslant 1$ and $b\in (0,1)$.

The existence of the V-states for the gSQG equation \eqref{Eq-gSQG} starts with the work of  Hassainia and Hmidi \cite{HH15}  where they showed  similar results to Burbea curves in the whole plane and for  $\alpha\in (0,1)$. Later, Castro, C\'ordoba and G\'omez-Serrano \cite{CCG16} proved the existence of V-states for the range $\alpha\in [1,2)$ and obtained  the $C^{\infty}$-regularity of their boundary for all $\alpha\in (0,2),$ see also \cite{CCG16b} for the real analyticity of the patch boundary.
For other connected topics we refer to the papers \cite{Ao,CQZZ,CCG16,DHH16,DHHM,Gom19,Gravejat,HM17}  and the references therein.
\vskip1mm

In this paper, we shall focus on the existence of the V-states for the gSQG equation \eqref{Eq-gSQG}  with $\alpha\in(0,1)$ in the unit disc $\DD$. More precisely, we want to construct rigid periodic solutions  around radial stationary patches ${\bf{1}}_{b \DD}, b\in(0,1)$  using bifurcation tools.  We remind that the case $\alpha=0$ was discussed in \cite{DHHM}. 
As a by-product  we construct an  infinite family of non-stationary global solutions  for the gSQG equation in the bounded  domain $\DD$, although the global well-posedness/blow up issue is not well understood and remains an open problem. The situation in bounded domains turns out to be  more tricky  due to the non explicit form of Green function associated to the spectral fractional  Laplacian. This has an impact in the study of the regularity of the functional that will describe the V-states. Later,  in Section \ref{sec:Time periodic motion of interfaces}
we shall explore   how to recover the boundary equation  of the  V-states close to the patch ${\bf{1}}_{b\mathbb{D}}$ in terms of  polar coordinates   $\theta\in\RR\mapsto \sqrt{b^2+2r(\theta)}e^{i\theta}$. In this regard,
the deformation radius  $r$ solves a nonlinear integro-differential equation of the type,
\begin{equation}\label{Equa-Nonlin}
  F(\Omega,r(\theta))\triangleq \Omega\, r'(\theta)+ \partial_\theta
  \bigg(\int_0^{2\pi} \int_0^{R(\eta)}K^\alpha\big(R(\theta)e^{i\theta},\rho e^{i\eta} \big) \rho\,
  \dd\rho \dd\eta\bigg)=0 ,\quad R(\eta)\triangleq \sqrt{b^2 + 2r(\eta)},
\end{equation}
where $K^{\alpha}$ stands for  the Green  function of  the spectral fractional Laplacian
$(-\Delta)^{-1+\tfrac{\alpha}{2}}$ on the unit disc $\DD$, defined via the relation
\begin{equation*}
	(-\Delta)^{-1+\frac{\alpha}{2}}f(x)=\int_{\mathbb{D}}K^{\alpha}(x,y)f(y)\dd y.
\end{equation*}
One may easily verify that $F(\Omega,0)=0$ for every $\Omega\in\RR$ and therefore the next step is to check that the bifurcation tools such as the Crandall-Rabinowitz theorem \cite{C-R71} applies in this framework. %
To state our main result, we shall introduce the following angular velocities,
\begin{align}\label{def:Omg-m-alp}
  \Omega_{m,b}^\alpha \triangleq 2\sum_{ k\geqslant 1}x_{0,k}^{\alpha-2}\frac{J_1^2\big(x_{0,k} b \big)}{J_1^2(x_{0,k} )}
  -2\sum_{k\geqslant 1}x_{m,k}^{\alpha-2}\frac{J_m^2(x_{m,k}b)}{J_{m+1}^2(x_{m,k})},
\end{align}
where $J_m$ denotes the Bessel functions of order $m$ and $x_{m,k}$ denotes the $k$-th positive root of $J_m(x)=0$, see Section \ref{Sect-Bess} for more details. Now, we are in a position to give our main result.

\begin{theorem}\label{main-theorem}
 Let  $(\alpha,b,m)$ satisfy one of the following conditions
\begin{align}
  & \alpha\in \,(0,1),\; b\in\, \left(0, (\tfrac{1-\alpha}{2-\alpha/2})^{\frac{1}{2}}\right],\;m\geqslant 1,
  \label{case1} \\
  & \alpha\in \,(0,1),\; b\in \,(0,1), \; m \geqslant m^*, \label{case2} \\
  & \alpha\in \,(0,\alpha^*),\;b\in \,(0,1),\; m\geqslant 1, \label{case3}
\end{align}
with $m^*=m^*(\alpha,b)\in\NN$
$($a rough bound is $m^* \leqslant \frac{1}{\log b} \big(\log \frac{1-\alpha}{1-\alpha/2- (e\log b)^{-1}}\big))$
and $\alpha^* =\alpha^*(b)>0$ a small number. Then there exists a family of m-fold symmetric $V$-states $(V_m)_{m\geqslant  1}$
for the gSQG equation \eqref{Eq-gSQG} bifurcating from the trivial solution $\theta_0={\bf{1}}_{b\mathbb{D}}$
at the angular velocity $\Omega_{m,b}^\alpha$ given by \eqref{def:Omg-m-alp}.
In addition, the boundary of the $V$-states belongs to the H\"older class $C^{2-\alpha}$.

More precisely, there exist a constant $a > 0$ and two continuous functions $\Omega:\,(-a,a)\rightarrow \mathbb{R}$, $r:\,(-a,a)\rightarrow C^{2-\alpha}(\mathbb{T})$  satisfying $\Omega(0)=\Omega_{m,b}^{\alpha}$, $r(0)=0$,
such that $(r_s)_{-a<s< a }$ 
 is a one-parameter non-trivial solution of the equation \eqref{Equa-Nonlin} describing $V$-states.
Moreover, $r_s$ admits the expansion
\begin{align*}
  \forall \theta\in\RR,\quad r_s(\theta)=s\cos(m\theta) + s\sum_{n\geqslant 1} b_{nm}(s) \cos(nm\theta),\quad  b_{nm}=O(s),
\end{align*}
and the mapping $\theta\mapsto \sqrt{b^2+2r_s(\theta)}\,e^{i\theta}$ maps the torus  $\mathbb{T}$ to the boundary of an $m$-fold rotating patch with the angular velocity $\Omega(s)$.
\end{theorem}

Before giving the ideas of the proof, we shall make some comments.
\begin{remark}
Theorem \ref{main-theorem} shows the existence of global solutions near Rankine vortices. This issue is open for general initial data.
\end{remark}
\begin{remark}
When the domain of the fluid is the ball $B(0, R)$, with $R>1$, then by a scaling argument and applying the preceding theorem,
the bifurcation from the unit disc $\DD$ occurs at the angular velocities (see \cite[Proposition 3]{HH15})
\begin{align*}
  \widetilde{\Omega}^\alpha_{m,R} \triangleq R^{-\alpha} \Omega_{m,R^{-1}}^\alpha .
\end{align*}
According to \eqref{eq:Omg-al-m}, \eqref{V-1-transfer} and \eqref{eq:alp-m-exp}, and using \eqref{eq:alpha-m-21} and \eqref{eq:V10-rem}
to control the remainder terms, we obtain
\begin{align}\label{eq:limit1}
  \lim_{R\rightarrow \infty} \widetilde{\Omega}^\alpha_{m,R}
  = \frac{\Gamma(1-\alpha)}{2^{1-\alpha}\Gamma^2(1-\frac{\alpha}{2})}
  \Big( \frac{\Gamma(1+\frac{\alpha}{2})}{\Gamma(2-\frac{\alpha}{2})}
  - \frac{\Gamma(m+\frac{\alpha}{2})}{\Gamma(m+1-\frac{\alpha}{2})} \Big),
\end{align}
which corresponds to Hassainia-Hmidi's result in \cite{HH15}.

On the other hand, when $\alpha\rightarrow 0$, as indicated in Lemma \ref{lem:Omg-m}-(i),
we recover the result of De la Hoz, Hassainia, Hmidi and Mateu \cite{DHHM} in the limit.
\end{remark}

\begin{remark}
  For the SQG equation corresponding to the case  $\alpha=1$  the situation is more delicate.
By reformulating the boundary equation of V-states to relax the violent singularity of the kernel,
we can give a full description of the linearized operator $\partial_r F(\Omega,0)$ and the associated dispersion relation.
It turns out that the spectrum $\Omega^1_{m,b}$ coincides with the limit of $\Omega^\alpha_{m,b}$ when
$\alpha\rightarrow 1$ as \eqref{eq:Omeg-alp->1} shows.
However, the function spaces used here are not well-adapted due to a logarithmic loss in frequency  when  $\alpha=1$. We believe that the $L^2$ weighted spaces introduced in \cite{CCG16} could be used in order to generate the  V-states in the critical case $\alpha=1$.%
\end{remark}

In the proof of Theorem \ref{main-theorem}, the first difficulty that one should  face is related to  the kernel function $K^{\alpha}$ which
has no  explicit  form  as in  the whole space  for  the gSQG equation \cite{HH15}
or Euler equation in the disc \cite{DHHM}. This makes the regularity problem of the functional $F$ introduced in \eqref{Equa-Nonlin} more complicated and to  circumvent this issue we establish in Lemma \ref{lem:ker-exp} the following decomposition in $\DD\times\DD$
\begin{align*}
  K^\alpha (x,y) = \frac{c_{\alpha}}{|x-y|^{\alpha}} + K_1^\alpha(x,y)
\end{align*}
where $K_1^\alpha$ is a smooth function in $\DD\times\DD$. We emphasize that this splitting is valid for any smooth bounded domain and extends a classical  result for Euler equation \cite{DHHM,Evans10}.
Remark that for the specific case of the unit disc one can recover this kernel from the eigenfunctions of the Laplacian which are explicitly described through Bessel functions, see \eqref{lem:disc-exp} below. However, it is not at all clear how to deal with this series and deduce the splitting mentioned above.
%
Then by virtue of this  decomposition together with Lemma \ref{lem:reg-1} dealing with singular kernel integrals on the torus,
we can prove the desired regularity properties of $F$ needed in the Crandall-Rabinowitz theorem (see Theorem \ref{thm:C-R}).
The spectral problem is carefully studied in Section \ref{sec:spect-Stu} and one important delicate point  lies on the analysis of angular velocity sequence \{$\Omega_{m,b}^\alpha, m\geqslant 1\}$. In order to get a one dimensional kernel we need  to check the monotonicity of these frequencies. The formula \eqref{def:Omg-m-alp} seems to be out of use due to the complexity of the sum. Surprisingly, this complicated form admits a nice integral representation by virtue of  Sneddon's formula \cite{Sne66}, from which we conduct  a careful analysis and  manage to show the key monotonicity property of the sequence $\{\Omega^\alpha_{m,b},m\geqslant 1\}$ under some  constraints on $\alpha, b$ and the symmetry. For this discussion, we refer to  Lemma \ref{lem:sned-form} and the proof of Lemma \ref{lem:Omg-m}.
%
Notice that the rest of  the conditions of  Crandall-Rabinowitz theorem are satisfied allowing to get an affirmative answer for the existence of the V-states for the gSQG equation \eqref{Eq-gSQG} in the disc $\DD$
when $\alpha\in (0,1)$.


The reminder of the paper is organized as follows.
In the next section, we shall present some technical results related to the spectral fractional Laplacian and the associated Green function in bounded smooth domain, and then we shall introduce some estimates on singular kernel integrals and also recall Sneddon's formula.
In Section \ref{sec:Time periodic motion of interfaces}, we shall write down the boundary equation of V-states in the unit disc.
The Sections \ref{sec:reg-F} and \ref{sec:spect-Stu} are devoted to the proof of Theorem \ref{main-theorem}.
In \mbox{Section \ref{sec:reg-F},} 
we study the linearization and regularity of the nonlinear functionals in the boundary equation.
In \mbox{Section \ref{sec:spect-Stu},} we conduct the spectral study of the linearized operator around zero and
under suitable assumptions we obtain a Fredholm operator of zero index.
Finally, in the last section we recall the Crandall-Rabinowitz theorem and give the proof of some auxiliary lemmas used in the paper.

\vskip1mm
\noindent \textbf{Notation.} Throughout this space we shall use the following convention and notation.
\begin{enumerate}[$\bullet$]
\item  $C$ denotes a positive constant that may change its value from line to line.
\item The set $\NN=\{0,1,2,\cdots\}$ is the set of nonnegative numbers,
and $\mathbb{N}^{+}=\{1,2,..\}$ denotes the set of all positive integers.
\item Let $X$ and $Y$ be two Banach spaces.
We denote by $\mathcal{L}(X, Y )$ the space of all continuous linear maps $T:X\rightarrow Y$
endowed with its usual strong topology.
\item We denote by  $\mathbf{D}$ a bounded open domain with smooth boundary of the Euclidean space $\mathbb{R}^d$,
while we use the notation $\mathbb{D}$ to denote the unit disc of the Euclidean space $\mathbb{R}^2$.
\end{enumerate}

\section{Tools}\label{sec:tools}
This section is devoted to some technical results related to the structure of the heat semi-flow and the Green function of fractional Laplacian in bounded smooth domains. We shall also discuss some estimates on integrals with singular periodic kernels and recall  Sneddon's  formula which plays a central role on the spectral study.
\subsection{Spectral fractional Laplacian and  Green function}
The main goal of this subsection is to explore the structure of Green function associated to Dirichlet  fractional Laplacian. We shall first recall how to define the spectral fractional Laplacian in bounded domains.
Let $\mathbf{D}\subset \RR^d$ $(d\ge 2)$ be a bounded open domain with smooth boundary. The $L^2(\mathbf{D})$-normalized eigenfunctions of the operator $-\Delta$ supplemented with Dirichlet boundary condition  are denoted by  $\phi_j$, and the associated  eigenvalues
counted with their multiplicities are positive real numbers  $\lambda_j$ such that
\begin{align}\label{def:phi-j}
\hbox{For}\quad j\geqslant1,\,	-\Delta\phi_j=\lambda_j\phi_j,\quad \phi_j |_{\partial \mathbf{D}} =0,\quad  \int_{\mathbf{D}}\phi_j^2(x) \dd x=1.
\end{align}
It is a classical fact that
\begin{align*}
  0<\lambda_1<\lambda_2\leqslant...\leqslant\lambda_j\rightarrow\infty.
\end{align*}
Now, according to the functional calculus, the spectral fractional Laplacian $ (-\Delta)^{-1+\frac{\alpha}{2}}$ with $\alpha\in(0,2)$  is defined through
\begin{equation}\label{Kern-For0}
	\begin{split}
		(-\Delta)^{-1+\frac{\alpha}{2}}f(x)&=\frac{1}{\Gamma(1-\frac{\alpha}{2})}
		\int_0^{\infty} t^{-\frac\alpha2} e^{t\Delta} f(x) \dd t \\
		&=\int_{\mathbf{D}}K^\alpha(x,y) f(y) \dd y,
	\end{split}
\end{equation}
with $\Gamma$ being the classical  Gamma function and $K^\alpha$ the Green function. Let $H_\mathbf{D}(t,\cdot,\cdot)$ denote  the kernel of the heat semigroup $e^{t\Delta}$ on the domain $\mathbf{D}$ with Dirichlet boundary condition, then
\begin{align*}
  e^{t\Delta}f(x) = \int_\mathbf{D} H_\mathbf{D}(t,x,y) f(y)\dd y.
\end{align*}
It is a classical fact that this kernel can be reconstructed from the eigenfunctions \eqref{def:phi-j} as follows
\begin{align}\label{Heat-Str1}
\forall\ t\geqslant0, x,y\in\mathbf{D},\quad  H_\mathbf{D}(t,x,y)=\sum_{j\geqslant 1} e^{-\lambda_j t}\phi_j(x)\phi_j(y).
\end{align}
Consequently, the Green function $K^\alpha$ admits  different representations
\begin{align}
	K^\alpha(x,y) &
	= \frac{1}{\Gamma(1-\frac{\alpha}{2})} \int_0^\infty t^{-\frac{\alpha}{2}} H_{\mathbf{D}}(t,x,y) \dd t\label{eq:K-H-rela} \\
	& = \frac{1}{\Gamma(1-\frac{\alpha}{2})} \int_0^{\infty} t^{-\frac\alpha2}\sum_{j\geqslant 1} e^{-\lambda_j t}\phi_j(x)\phi_j(y) \dd t \nonumber \\
	& = \sum_{j\geqslant1} \lambda_j^{\frac\alpha2-1}\phi_j(x)\phi_j(y).\label{Kern-For}
\end{align}
There is an abundant literature dealing with the analytic properties of the heat kernels, see for instance \cite{CI16a,Davies89,Zhang06}.
Here, we shall restrict the discussion to some of them that will be needed later.
In view of the points  (31)-(32) in \cite{CI16a} or \cite{Davies89,Zhang06},
there exists a time $T_0>0$ and  positive constant $C$ depending only on the domain  $\mathbf{D}$  such that for all $0\leqslant t\leqslant T_0$ and $x,y\in \mathbf{D},$
\begin{equation}\label{eq:H_D-es0}
  0 < H_\mathbf{D}(t,x,y)\leqslant C\min\Big\{ \tfrac{\phi_1(x)}{|x-y|},1\Big\}
  \min\Big\{\tfrac{\phi_1(y)}{|x-y|},1\Big\} t^{-\frac{d}{2}}e^{-\frac{|x-y|^2}{C\,t}}
\end{equation}
and
\begin{equation}\label{eq:H_D-es1}
  \tfrac{|\nabla_x H_\mathbf{D}(t,x,y)|}{H_\mathbf{D}(t,x,y)} \leqslant
  \begin{cases}
    \frac{C}{d(x)}, \quad & \textrm{if}\;\; \sqrt{t}\geqslant d(x), \\
    \frac{C}{\sqrt{t}}\Big(1+\frac{|x-y|}{\sqrt{t}}\Big), \quad & \textrm{if}\;\; \sqrt{t}\leqslant d(x),
  \end{cases}
\end{equation}
with $d(x) \triangleq d(x,\partial \mathbf{D})$ being the Euclidean distance between $x$ and the boundary $\partial \mathbf{D}$. The function  $\phi_1$ is  the first eigenfunction of $-\Delta$ as in \eqref{def:phi-j}.\\
Before stating the main result of this section, see Lemma \ref{lem:ker-exp},  on the Dirichlet Green function $K^\alpha$,
we give two auixilary results on the Dirichlet heat kernel.
The first one is on the higher differentiability  of $H_\mathbf{D}(t)$ inside the domain $\mathbf{D}$  for a short time whose  proof is classical and will be postponed later in the Appendix. Its statement reads as follows.
\begin{lemma}\label{lem:estimate_H}
We have that for every $(x,y)\in \mathbf{D}\times \mathbf{D}$, $n\in\NN$
and for every $0< t \leqslant \min\{ d(x)^2,T_0\}$,
\begin{align*}
  |\nabla^n_x H_\mathbf{D}(t,x,y)| \leqslant C t^{-\frac{n+d}{2}} e^{-\frac{|x-y|^2}{C t}},
\end{align*}
with $C$ depending on $d,n$ and $T_0$.
\end{lemma}
Notice that the latter  estimate in the preceding lemma  degenerates on the diagonal $x=y$ for small time $t$ in a similar way to the Gauss kernel of the heat semigroup $e^{t\Delta}$ on the whole space $\RR^d$ explicitly given by   $(x,y)\mapsto G_t(x-y) = \frac{1}{(4\pi t)^{d/2}} e^{-\frac{|x-y|^2}{4t}}$.
The second auxiliary result deals with the description of this defect and shows that $H_\mathbf{D}(t)$ differs  from $G_t$  by a smooth contribution uniformly for small time. The  proof will be provided  in the Appendix.
\begin{lemma}\label{lem:es-kern}
Let  $(x,y)\in \mathbf{D} \times \mathbf{D}$, $k,l\in\NN$
and fix  $0< t \leqslant \min\{d(x)^2, d(y)^2, T_0\}$. Then, we have
\begin{equation*}
  \left|\nabla_x^k \nabla_y^l \big(H_\mathbf{D}(t,x,y) -G_t(x-y) \big) \right| \leqslant C,
\end{equation*}
where $C>0$ depends on $d,k,l,T_0, d(x)$ and  $d(y)$.
\end{lemma}
Now, we are ready to state the main result of this section dealing with
 a natural decomposition of Dirichlet Green function $K^\alpha$. It can be split into a singular one that coincides with the
 whole-space Green function and a smooth term  with bounded derivatives inside the domain $\mathbf{D}\times \mathbf{D}$.
\begin{lemma}\label{lem:ker-exp}
Let $\mathbf{D}\subset \RR^d$ $(d\ge 2)$ be a bounded open domain with a smooth boundary.
Let $\alpha\in (0,2)$ and $K^\alpha$ be the kernel function given by \eqref{Kern-For}.
Then we have that for every $x\neq y \in \mathbf{D} \times \mathbf{D}$,
\begin{align}\label{eq:Kalp-exp}
  K^\alpha(x,y) = \frac{c_\alpha }{|x-y|^{\alpha+ d -2}} + K^\alpha_1(x,y),
\end{align}
with $c_\alpha =\frac{4^{\alpha/2-1}\Gamma (\frac{\alpha}{2}+\frac{d}{2}-1)}{\pi^{d/2} \Gamma(1-\frac{\alpha}{2})}$ and $K_1^\alpha\in C^\infty (\mathbf{D}\times \mathbf{D})$.
\end{lemma}

\begin{remark}
This result is compatible with the classical one known for $\alpha=0$. For instance  when $\alpha=0$ and  $d = 2$, we have
\begin{equation*}
  K^0(x,y)=-\frac{1}{2\pi}\log |x-y| + K_1^{0}(x,y),
\end{equation*}
where $K_1^{0}$ is a smooth function in the open domain $\mathbf{D} \times \mathbf{D}$ and partially  harmonic in $x$ and $y$ (\cite{Evans10}).
\end{remark}

\begin{proof}[Proof of Lemma $\ref{lem:ker-exp}$]
  For $x,y\in\mathbf{D}$ fixed, there exists an open domain $\mathbf{D}_0 \subset \overline{\mathbf{D}_0}\subset \mathbf{D}$
such that $x,y\in \mathbf{D}_0$.
Denote $d_0 \triangleq \min\{\sqrt{T_0},d(\overline{\mathbf{D}_0}, \partial \mathbf{D})\}>0$.
Then,  we get from \eqref{eq:K-H-rela} the decomposition
\begin{align*}
  K^\alpha(x,y) & = \frac{1}{\Gamma(1-\frac{\alpha}{2})}
  \int_0^{d_0^2} t^{-\frac{\alpha}{2}} H_\mathbf{D}(t,x,y) \dd t
  + \frac{1}{\Gamma(1-\frac{\alpha}{2})} \int_{d_0^2}^\infty t^{-\frac{\alpha}{2}} H_\mathbf{D}(t,x,y) \dd t \nonumber\\
  & =\frac{1}{\Gamma(1-\frac{\alpha}{2})}  \int_0^{d_0^2} t^{-\frac{\alpha}{2}} G_t(x-y) \dd t +
  \frac{1}{\Gamma(1-\frac{\alpha}{2})}
  \int_0^{d_0^2} t^{-\frac{\alpha}{2}} \big(H_\mathbf{D}(t,x,y) - G_t(x-y)\big)\dd t \nonumber \\
  & \quad   +  \frac{1}{\Gamma(1-\frac{\alpha}{2})}
  \int_{d_0^2}^\infty t^{-\frac{\alpha}{2}} H_\mathbf{D}(t,x,y) \dd t \nonumber \\
  & \triangleq \frac{1}{\Gamma(1-\frac{\alpha}{2})} \int_0^{d_0^2} t^{-\frac{\alpha}{2}} G_t(x-y) \dd t +
  K^\alpha_{11}(x,y) + K^\alpha_{12}(x,y).
\end{align*}
To deal with the first term, we
use a  change of variables allowing to get
\begin{align*}
  \frac{1}{\Gamma(1-\frac{\alpha}{2})}\int_0^{d_0^2} t^{-\frac{\alpha}{2}} G_t(x-y)\dd t & =
  \frac{1}{\Gamma(1-\frac{\alpha}{2})}\int_0^{d_0^2} t^{-\frac{\alpha}{2}}
  \frac{1}{(4\pi t)^{d/2}} e^{-\frac{|x-y|^2}{4t}} \dd t \\
  & = \frac{4^{\alpha/2 -1}}{\pi^{d/2}\Gamma(1-\frac{\alpha}{2})} \frac{1}{|x-y|^{\alpha+d -2}}
  \int_{\frac{|x-y|^2}{4d_0^2}}^\infty \tau^{\frac{\alpha}{2}+ \frac{d}{2} -2} e^{-\tau} \dd \tau \\
  & = \frac{4^{\alpha/2-1}\Gamma (\frac{\alpha}{2}+\frac{d}{2}-1)}{\pi^{d/2} \Gamma(1-\frac{\alpha}{2})}
  \frac{1}{|x-y|^{\alpha + d-2}}
  - K_{13}^\alpha(x,y),
\end{align*}
with
\begin{align}\label{eq:K13}
  K^\alpha_{13}(x,y) \triangleq \frac{4^{\alpha/2-1}}{\pi^{d/2}\Gamma(1-\frac{\alpha}{2})}  \frac{1}{|x-y|^{\alpha + d-2}}
  \int_0^{\frac{|x-y|^2}{4d_0^2}} \tau^{\frac{\alpha}{2} + \frac{d}{2}-2} e^{-\tau} \dd \tau.
\end{align}
Therefore, we may decompose $K^\alpha$ as in \eqref {eq:Kalp-exp} with
\begin{align*}
  K^\alpha_1(x,y) = K^\alpha_{11}(x,y) + K^\alpha_{12}(x,y) - K^\alpha_{13}(x,y).
\end{align*}
In what follows we intend to  show that  all these functions  are $C^\infty$-smooth in $\mathbf{D}\times \mathbf{D}$.
For $K^\alpha_{11}$, by virtue of Lemma \ref{lem:es-kern}, we infer that for every $k,l\geq 0$ and $\alpha\in(0,2)$,
\begin{align*}
  |\nabla_x^k \nabla_y^l K^\alpha_{11}(x,y)| &\leqslant C \int_0^{d_0^2}
  t^{-\frac{\alpha}{2}}\dd t\\
  &\leqslant C.
\end{align*}
The smoothness of $K^\alpha_{12}$ is a direct consequence of Lemma \ref{lem:phi}. Indeed, using \eqref{Heat-Str1} we infer
\begin{align*}
  |\nabla_x^k\nabla^l_yK^\alpha_{12}(x,y)| & = \frac{1}{\Gamma(1-\frac{\alpha}{2})} \Big|\int_{d_0^2}^\infty t^{-\frac{\alpha}{2}}
  \Big(\sum_{j\geqslant 1} e^{-\lambda_j t} \nabla^k_x \phi_j(x) \nabla^l_y \phi_j(y) \Big) \dd t \Big| \\
  & \leqslant C \sum_{j\geqslant 1} \int_{d_0^2}^\infty t^{-\frac{\alpha}{2}} e^{-\lambda_j t}
  \|\nabla^k \phi_j\|_{L^\infty} \|\nabla^l \phi_j\|_{L^\infty} \dd t.
  \end{align*}
  Applying the estimates of Lemma \ref{lem:phi} with Sobolev embeddings we get in view of Weyl's law stating that  $\lambda_j \sim j$ as $j\rightarrow \infty$,
\begin{align*}
  |\nabla_x^k\nabla^l_yK^\alpha_{12}(x,y)|
  & \leqslant C \Big(\sum_{j\geqslant 1} e^{-\frac{d_0^2}{2}\lambda_j} (1+\lambda_j)^{\frac{k+l}{2} +d +1} \Big)
  \int_{d_0^2}^\infty t^{-\frac{\alpha}{2}} e^{-\frac{\lambda_1}{2}t}\dd t\\
  & \leqslant C.
\end{align*}
Now we shall move to $K^\alpha_{13}$ defined in \eqref{eq:K13} and show that it is smooth. According to Taylor expansion of $e^{-\tau}$, we find
\begin{align*}
  K^\alpha_{13}(x,y) & = \frac{4^{\alpha/2-1}}{\pi^{d/2}\Gamma(1-\frac{\alpha}{2})}  \frac{1}{|x-y|^{\alpha + d-2}}
  \int_0^{\frac{|x-y|^2}{4d_0^2}} \tau^{\frac{\alpha}{2} + \frac{d}{2}-2} \Big(\sum_{n=0}^\infty \frac{(-\tau)^n}{n!}
  \Big)\dd \tau \\
  & = \frac{1}{2^{2n+d}\pi^{d/2}\Gamma(1-\frac{\alpha}{2})d_0^{\alpha+d-2}}
  \sum_{n=0}^\infty \frac{(-1)^n|x-y|^{2n}}{n!(n+\frac{\alpha}{2}+\frac{d}{2}-1) d_0^{2n}},
\end{align*}
which is an absolutely convergent power series and is $C^\infty$-smooth in $x$ and $y$.
Hence, gathering the above estimates leads to the desired splitting  \eqref{eq:Kalp-exp}
with $K^\alpha_1\in C^\infty(\mathbf{D}\times \mathbf{D})$.
\end{proof}

The next goal is to explore how the symmetry of the  domain can be reflected on the Green function,  in the sense that
if the domain $\mathbf{D}$ is invariant under suitable planar transformations then the kernel functions $H_{\mathbf{D}}$ and $K^\alpha$
will enjoy an  adequate  symmetry property. More precisely, we shall establish the following result.
\begin{lemma}\label{lem:HD-sym}
Let $\mathbf{D}\subset \RR^2$ be a bounded open domain with a smooth boundary.
Let $H_{\mathbf{D}}(t,\cdot,\cdot)$ be  the kernel of the heat semigroup $e^{t\Delta}$ on the domain $\mathbf{D}$
and $K^\alpha(\cdot,\cdot)$ be the Green function defined \mbox{via \eqref{Kern-For0}.} Then the following statements hold true.
\begin{enumerate}
\item Let $\bar{x}=(x_1,-x_2)$ be the reflection point of $x=(x_1,x_2)$.
If $\mathbf{D}$ is invariant by reflection with respect to the real axis, then
$$
\forall\, x,y\in \mathbf{D},\quad
H_{\mathbf{D}}(t,\bar{x},\bar{y}) = H_{\mathbf{D}}(t,x,y)\quad\hbox{and}\quad  K^\alpha(\bar{x},\bar{y}) =K^\alpha(x,y).
$$
\item If $e^{i\theta}\mathbf{D}=\mathbf{D}$ for some $\theta\in \mathbb{R}$, then
$$\forall\, x,y\in\mathbf{D},
\quad H_{\mathbf{D}}(t,e^{i\theta}x,e^{i\theta}y) = H_{\mathbf{D}}(t,x,y)\quad\hbox{and}\quad K^\alpha(e^{i\theta}x,e^{i\theta}y) = K^\alpha(x,y).
$$
\end{enumerate}
\end{lemma}

\begin{proof}[Proof of Lemma $\ref{lem:HD-sym}$]
\textbf{(i)} Observe that for each $ y\in \mathbf{D}$,
\begin{equation}\label{eq:heat-kernal}
\begin{cases}
  \partial_t H_\mathbf{D}(t,x,y) - \Delta_x H_\mathbf{D}(t,x,y) = 0, \quad &  x\in \mathbf{D}, x\neq y,\\
  H_\mathbf{D}(t,x,y) = 0, \quad & x\in \partial \mathbf{D}, \\
  H_\mathbf{D}(0,x,y) = \delta(x-y),
\end{cases}
\end{equation}
where $\delta(\cdot)$ is the Dirac $\delta$-function centered at the origin.
Now we view $H_\mathbf{D}(t,\bar{x},\bar{y})$ as a function of $x$ and  $y$ as a parameter. Then, we can check  that
\begin{align*}
  H_\mathbf{D}(0,\bar{x},\bar{y})=\delta(\bar{x}-\bar{y})=\delta(x-y)=H_\mathbf{D}(0,x,y)
\end{align*}
 and $\Delta_x\big(H_\mathbf{D}(t,\bar{x},\bar{y})\big)
= \big(\Delta_x H_\mathbf{D}\big)(t,\bar{x},\bar{y})$. In addition, using the fact that
$\bar{x}\in \partial\mathbf{D} \Leftrightarrow x\in \partial \mathbf{D}$, we have $H_\mathbf{D}(t,\bar{x},\bar{y})=0$
for any $x\in \partial \mathbf{D}$. Hence, for each $y\in \mathbf{D}$, the mapping $x\in\mathbf{D}\mapsto  H_\mathbf{D}(t,\bar{x},\bar{y})$ solves \eqref{eq:heat-kernal}. Owing to the uniqueness of the initial-boundary value problem \eqref{eq:heat-kernal},
we conclude that $H_\mathbf{D}(t,x,y) = H_\mathbf{D}(t,\bar{x},\bar{y})$.
In view of \eqref{eq:K-H-rela}, the desired equality $K^\alpha(\bar{x},\bar{y})= K^\alpha(x,y)$ directly follows.

\textbf{(ii)} The proof of statement (ii) is quite similar as above, and thus we omit the details.
\end{proof}

Hereafter we shall give a precise description of the Green function $K^\alpha$ when  the domain $\mathbf{D}$ is the planar unit  disc $\mathbb{D}$. Actually, in this radial case    the eigenvalues and eigenfunctions of the spectral Laplacian $-\Delta$ on $\DD$
have precise expression formula through Bessel functions and thus the Dirichlet Green function $K^\alpha$ might be explicitly calculated.
We actually have the following result.
\begin{lemma}\label{lem:disc-exp}
Let $\mathbf{D}=\mathbb{D}$ be the unit disc of $\RR^2$ and let $K^\alpha$ given by \eqref{Kern-For}. Then
the eigenvalues and the eigenfunctions solving the spectral problem \eqref{def:phi-j} are described by  double index families $\{\lambda_{n,k}\}_{n\in \NN,k\geqslant 1}$ and $\{(\phi_{n,k}^{(1)}, \phi_{n,k}^{(2)})\}_{n\in\NN,k\geqslant 1}$ such that
\begin{align*}
  \lambda_{n,k}=x_{n,k}^2, \quad \phi_{n,k}^{(1)}(x)=J_n(x_{n,k}|x|)A_{n,k}\cos(n\theta),
  \quad \phi_{n,k}^{(2)}(x)=J_n(x_{n,k}|x|)A_{n,k}\sin(n\theta),
\end{align*}
where
\begin{align}\label{eq:Ank}
  \pi A_{0,k}^2=\frac{1}{J_{1}^2(x_{0,k})}\quad\hbox{and}\quad
  \pi A_{n,k}^2=\frac{2}{J_{n+1}^2(x_{n,k})},\quad\forall n\geqslant 1,
\end{align}
and $J_n$  denotes the first kind Bessel function of order $n$  and $\big\{x_{n,k}, k\geqslant1\big\}$ are its  zeroes.
Furthermore, we have
\begin{align}\label{eq:K-ss}
  K^\alpha(x,y)=\sum_{n\in\NN, k\geqslant 1}x_{n,k}^{\alpha-2}
  \Big( \phi_{n,k}^{(1)}(x)\phi_{n,k}^{(1)}(y) + \phi_{n,k}^{(2)}(x)\phi_{n,k}^{(2)}(y) \Big).
\end{align}
\end{lemma}

\begin{proof}[Proof of Lemma $\ref{lem:disc-exp}$]
The explicit formula of the eigenvalues and  the normalized eigenfunctions of $-\Delta$ on the disc $\mathbb{D}$
are well-known,
and for instance one can refer to Section 5.5 of Chapter V in \cite{CH09} for the proof.
As a result, the formula \eqref{eq:K-ss} is an immediate consequence of the identity \eqref{Kern-For}.
\end{proof}
\begin{remark}
Let $x=\rho_1 e^{i\theta}, y=\rho_2e^{i\eta}\in\mathbb{D}$, then  by setting  $K^\alpha(x,y)=  G(\rho_1,\theta,\rho_2,\eta)$,
 we get from the expression \eqref{eq:K-ss}
\begin{align}\label{eq:G-exp}
  K^\alpha(x,y) & = G(\rho_1,\theta,\rho_2,\eta) =\sum_{n\in\NN, k\geqslant 1\atop  1\leqslant  j\leqslant  2}x_{n,k}^{\alpha-2}\phi^{(j)}_{n,k}
  \big(\rho_1 e^{i\theta}\big)\phi^{(j)}_{n,k}\big(\rho_2 e^{i\eta}\big) \\
  & = \sum_{n\in\NN\atop  k\geqslant 1} x_{n,k}^{\alpha-2} A_{n,k}^2 J_n(x_{n,k}\rho_1) J_n(x_{n,k}\rho_2)
  \big(\cos(n\theta)\cos(n\eta) + \sin(n\theta)\sin(n\eta) \big) \nonumber \\
  & = \sum_{n\in\NN\atop k\geqslant 1} x_{n,k}^{\alpha-2} A_{n,k}^2 J_n(x_{n,k}\rho_1) J_n(x_{n,k}\rho_2)
  \cos \big(n(\theta-\eta)\big).\nonumber
\end{align}
Straightforward computations yield
\begin{align}\label{eq:K-der}
		\nabla_xK^\alpha(x,y)=\begin{pmatrix}
		\partial_{\rho_1}G\cos \theta - \partial_{\theta}G \frac{\sin \theta}{\rho_1}\\
		\partial_{\rho_1}G\sin \theta + \partial_{\theta}G  \frac{\cos \theta}{\rho_1}
	\end{pmatrix},\quad
     \nabla_y K^\alpha (x,y)=\begin{pmatrix}
		\partial_{\rho_2}G \cos \eta - \partial_{\eta}G \frac{\sin \eta}{\rho_2}\\
		\partial_{\rho_2}G\sin \eta+\partial_{\eta}G  \frac{\cos \eta}{\rho_2}
	\end{pmatrix}.
\end{align}
These identities will be useful later in the explicit computation of the linearized operator at the equilibrium state, see the proof of Proposition  \ref{prop:other-condition}.

\end{remark}

\subsection{Singular kernel integrals on the torus}
In this subsection, we intend  to deal with  integrals with singular kernels  of the following type
\begin{equation}\label{eq:K-f}
  \mathcal{T}(f)(\theta) \triangleq \int_{\mathbb{T}}K(\theta,\eta)\, f(\eta)\dd\eta,
\end{equation}
where $\mathbb{T}$ is the periodic torus (identified with $[0,2\pi)$), $K:\mathbb{T}\times\mathbb{T}\to \mathbb{C}$
is a suitable singular kernel,
and $f: \TT\rightarrow \mathbb{C}$ is a $2\pi$-periodic function.
This structure will appear later when we will  explore  the regularity of the nonlinear operator in the rotating patches formalism, see Section \ref{sec:reg-F}.
The result that we shall present below is more or less classical and is analogous to \cite{HH15,MOV09}, and we shall provide a complete proof for the self-containing of the paper.

\begin{lemma}\label{lem:reg-1}
Let $0\leqslant \alpha<1$ and assume that there exists $C_0>0$ such that $K:\mathbb{T}\times\mathbb{T}\to \mathbb{C}$
satisfies the following properties.
\begin{enumerate}
\item $K$ is measurable on $\mathbb{T}\times\mathbb{T}\backslash\{(\theta,\theta),\, \theta\in \mathbb{T}\}$ and
\begin{equation}\label{eq:lem-k-1}
  |K(\theta,\eta)|\leqslant \frac{C_0}{|\sin \frac{\theta-\eta}{2}|^{\alpha}},\quad \forall \eta \neq \theta\in \TT.
\end{equation}
\item For each $\eta\in \mathbb{T}$, the mapping $\theta\mapsto K(\theta,\eta)$ is differentiable in
$\mathbb{T}\backslash\{\eta\}$ and
\begin{equation}\label{eq:lem-k-2}
  |\partial_{\theta}K(\theta,\eta)|\leqslant \frac{C_0}{|\sin \frac{\theta-\eta}{2}|^{1+\alpha}},
  \quad \forall \theta\neq \eta \in \TT.
\end{equation}
\end{enumerate}
Then the linear operator $\mathcal{T}$ given by \eqref{eq:K-f} is continuous from $L^\infty(\mathbb{T})$
to $C^{1-\alpha}(\mathbb{T})$. More precisely, there exists a constant $C_\alpha>0$ depending only on $\alpha$ such that
\begin{align}\label{eq:T-bdd}
  \Vert \mathcal{T}(f)\Vert_{C^{1-\alpha}}\leqslant C_\alpha C_0\Vert f\Vert_{L^\infty}.
\end{align}
\end{lemma}

\begin{proof}[Proof of Lemma $\ref{lem:reg-1}$]
We first prove that $\mathcal{T}(f)$ is bounded on $\mathbb{T}$.
For every $\theta\in\mathbb{T}$, thanks to \eqref{eq:lem-k-1}, we see that
\begin{align*}
  \vert \mathcal{T}(f)(\theta)\vert \leq  C_0 \Vert f\Vert_{L^\infty}\Big\vert \int_{\mathbb{T}} \frac{\dd\eta}{\vert \sin(\frac{\theta-\eta}{2})\vert^\alpha}\Big\vert
  = C_0 \|f\|_{L^\infty} \int_{-\pi}^{\pi}\frac{\dd\eta}{|\sin \frac{\eta}{2}|^{\alpha}}
  \leqslant  C_\alpha C_0 \Vert f\Vert_{L^\infty}.
\end{align*}
Next, for every $\theta_1,\theta_2\in \mathbb{T}$ such that $0<|\theta_1-\theta_2|\leqslant \pi$,
it is obvious that
\begin{align*}
  \tfrac{2}{\pi}|\theta_1 - \theta_2| \leqslant |e^{i\theta_1}- e^{i\theta_2}| = 2 \big|\sin \tfrac{\theta_1 -\theta_2}{2}\big| \leqslant |\theta_1-\theta_2|.
\end{align*}
Set $r\triangleq |e^{i\theta_1}-e^{i\theta_2}|$ and define $B_r(\theta) \triangleq \big\{ \eta\in \mathbb{T}:\,\vert 2\sin (\frac{\eta-\theta}{2})\vert = |e^{i\theta}-e^{i\eta}|  \leq r\big\}$. Then we have
\begin{align*}
  \Big\vert \mathcal{T}(f)(\theta_1)-\mathcal{T}(f)(\theta_2)\Big\vert
  \leqslant &\Big\vert \int _{B_{3r}(\theta_1)}\vert f(\eta)\vert \vert K(\theta_1,\eta)\vert \dd\eta\Big\vert+
  \Big\vert \int _{B_{3r}(\theta_1)}\vert f(\eta)\vert \vert K(\theta_2,\eta)\vert \dd\eta\Big\vert \\
  & +\Big\vert \int _{B^c_{2r}(\theta_1)}\vert f(\eta)\vert \vert K(\theta_1,\eta)
  -K(\theta_2,\eta)\vert \dd\tau\Big\vert\\
  \triangleq & J_1+J_2+J_3.
\end{align*}
Applying \eqref{eq:lem-k-1} together with a change of  variables allow to get the estimate
\begin{align*}
  J_1+J_2 \leqslant & C_0 \Vert f\Vert_{L^\infty}
  \bigg(  \Big\vert \int _{B_{3r}(\theta_1)}\frac{\dd\eta}{\vert \sin (\frac{\eta -\theta_1}{2})\vert^\alpha}\Big\vert
  +\Big\vert \int _{B_{3r}(\theta_2)}\frac{\dd\eta}{\vert \sin(\frac{\eta -\theta_2}{2})\vert^\alpha}\Big\vert\bigg)\\
  \leqslant &C_{\alpha}C_0 \Vert f\Vert_{L^\infty}
   \Big\vert \int_0^{\frac{3}{2}r} \frac{ \dd w}{\vert w \vert^\alpha \sqrt{1- w^2}} \Big\vert
  \leqslant C_\alpha C_0\Vert f\Vert_{L^\infty}\vert \theta_1-\theta_2\vert^{1-\alpha}.
\end{align*}
To estimate the third term $J_3$, noting that for every $\eta \in B_{2r}^c(\theta_1)$ and for any $\kappa \in [0,1]$,
\begin{align*}
  2 \big|\sin \tfrac{\theta_1-\eta+(1-\kappa)(\theta_2-\theta_1)}{2}\big|
  & = |e^{i(\theta_1-\eta)}-e^{i(1-\kappa)(\theta_1-\theta_2)}| \\
  &\geqslant |e^{i(\theta_1-\eta)}-1|-|e^{i(1-\kappa)(\theta_1-\theta_2)}-1|\\
  &\geqslant |e^{i(\theta_1-\eta)}-1|-|e^{i(\theta_1-\theta_2)}-1| = |e^{i\eta} - e^{i\theta_1}| - |e^{i\theta_1} -e^{i\theta_2}| \\
  &\geqslant \tfrac{1}{2}|e^{i\eta} - e^{i\theta_1}| = \big|\sin \tfrac{\eta -\theta_1}{2} \big|.
\end{align*}
Therefore, applying  the mean value theorem, \eqref{eq:lem-k-2} and the preceding estimate  we  find
\begin{align*}
  |K(\theta_1,\eta)-K(\theta_2,\eta)| \leqslant C_\alpha C_0 \frac{|\theta_1-\theta_2|}{|\sin \frac{\eta -\theta_1}{2}|^{1+\alpha}},
  \quad \forall \,\eta\in B_{2r}^c(\theta_1).
\end{align*}
Consequently, we obtain
\begin{align*}
  J_3 \leqslant &C_{\alpha} C_0\Vert f\Vert_{L^\infty} \int _{B^c_{2r}(\theta_1)}
  \frac{\vert \theta_1-\theta_2\vert}{\vert \sin \frac{\eta -\theta_1}{2} \vert^{1+\alpha}}\dd \eta  \\
  \leqslant & C_\alpha  C_0\Vert f\Vert_{L^\infty} \int_r^1
  \frac{\vert \theta_1-\theta_2\vert}{\vert w \vert^{1+\alpha} \sqrt{1-w^2} } \dd w \\
  \leqslant & C_\alpha C_0\Vert f\Vert_{L^\infty}\vert \theta_1-\theta_2\vert^{1-\alpha}.
\end{align*}
Hence gathering the above estimates concludes the proof of \eqref{eq:T-bdd}.
\end{proof}

\subsection{Special functions}\label{Sect-Bess}
The  main task of this subsection is to recall Sneddon's formula which is very crucial in the spectral problem associated to the linearization of  the vortex patch equations around radial solutions that will be explored in \mbox{Section \ref{sec:spect-Stu}.} It allows in our context to derive a suitable representation of the angular velocities from which periodic solutions  bifurcate from Rankine vortices.
Before stating this formula, we need to remind some special functions and notations. First, the Gamma function $\Gamma:\mathbb{C}\setminus(-\NN)\rightarrow \mathbb{C}$
is the analytic continuation to the negative half plane of the usual Gamma function defined on $\{\mathrm{Re}\,z>0\}$
by the integral formula
\begin{align*}
  \Gamma(z) = \int_0^\infty t^{z-1} e^{-t} \dd t.
\end{align*}
It satisfies the classical   relation
\begin{align}\label{eq:Gam-prop}
  \Gamma(z + 1) = z \Gamma(z),\quad \forall z\in \mathbb{C}\setminus(-\NN).
\end{align}
On the other hand, for every $z\in \mathbb{C}$ we denote $(z)_n$ the Pokhhammer's symbol defined by
\begin{equation*}
(z)_n =
\begin{cases}
  z(z+1)\cdots (z+n-1),\quad &\textrm{if}\;\;n\geqslant 1,\\
  1,\quad & \textrm{if}\;\; n=0.
\end{cases}
\end{equation*}
The following relations are straightforward
\begin{align}\label{eq:(x)n}
  (z)_n = \frac{\Gamma(z+n)}{\Gamma(z)},\quad (z)_n = (-1)^n\frac{\Gamma(1-z)}{\Gamma(1-z-n)},
\end{align}
provided that all the right-hand quantities are well-defined. In what follows we intend to recall Bessel functions and some of their variations.
The Bessel functions of order $\nu\in\mathbb{C}$ is defined by
\begin{align*}
  J_{\nu}(z)=\sum_{m=0}^{\infty}\frac{(-1)^{m}\left(\frac{z}{2}\right)^{\nu +2m}}{m!\Gamma(\nu+m+1)},\quad|\mbox{arg}(z)|<\pi.
\end{align*}
Next, we shall introduce  Bessel functions of imaginary argument also called modified Bessel functions of first and second kind, (e.g. see p.66 of \cite{MOS66})
\begin{equation}\label{I-m-q}
  I_{\nu}(z)=\sum_{m=0}^{\infty}\frac{\left(\frac{z}{2}\right)^{\nu+2m}}{m!\Gamma(\nu+m+1)},\quad|\mbox{arg}(z)|<\pi
\end{equation}
and
\begin{align*}
  K_{\nu}(z)=\frac{\pi}{2}\frac{I_{-\nu}(z)-I_{\nu}(z)}{\sin(\nu\pi)},\quad\nu\in\mathbb{C}\setminus\mathbb{Z},\quad|\mbox{arg}(z)|<\pi.
\end{align*}
However, when $j\in\mathbb{Z},$ $K_{j}$ is defined through the formula  $K_{j}(z)=\displaystyle\lim_{\nu\rightarrow j}K_{\nu}(z).$
\\
It is known that for $\nu>0$ the  zeros of the Bessel function $J_\nu$ are given by a countable family $\{x_{\nu,k},k\in\NN\}$ of positive increasing numbers with the following asymptotics
\begin{align*}
  x_{\nu,k}=(k+\tfrac\nu2-\tfrac14)\pi+O(k^{-1}),\quad k\to \infty.
\end{align*}
We also recall that for any real numbers $c_1,c_2\in \mathbb{R}$, $c_3\in\mathbb{R}\setminus(-\mathbb{N})$
the hypergeometric function
$z\mapsto F(c_1,c_2; c_3;z) $ is defined on the open unit disc $\mathbb{D}$ by the power series
\begin{align*}
  F(c_1,c_2;c_3;z)= \sum_{n=0}^\infty \frac{(c_1)_n (c_2)_n}{(c_3)_n} \frac{z^n}{n!},\quad \forall z\in\mathbb{D},
\end{align*}
where $(x)_n$ is Pokhhammer's symbol. The hypergeometric series converges in the unit disc $\mathbb{D}$.
Assume that $\mathrm{Re}(c_3)> \mathrm{Re}(c_2)>0$, then the hypergometric function
has an integral representation as follows
\begin{align*}
  F(c_1,c_2;c_3;z) = \frac{\Gamma(c_3)}{\Gamma(c_2) \Gamma(c_3-c_2)}\int_0^1 x^{c_2-1} (1-x)^{c_3-c_2-1} (1-zx)^{-c_1 }\dd x,
  \quad |z|<1.
\end{align*}
When $\mathrm{Re}(c_1 + c_2 -c_3)<0$, the hypergeometric series is absolutely convergent on the closed unit disc and one has the following expression
\begin{align*}
  F(c_1,c_2;c_3;1) = \frac{\Gamma(c_3) \Gamma(c_3 -c_1-c_2)}{\Gamma(c_3-c_1) \Gamma(c_3-c_2)}.
\end{align*}

Now, we are ready to state Sneddon's formula that can be found for instance in $(2.2.9)$ \mbox{in \cite{Sne66}},
or $(24)$-$(25)$ in \cite{Mar22}.
\begin{lemma}\label{lem:sned-form}
  Let $0<a,b\leqslant 1$, $n,\beta,\gamma\in \NN$ and  $1<q <\beta + \gamma - 2n +2$. Then we have
\begin{equation}\label{eq:sned-form}
\begin{split}
  \sum_{k=1}^\infty \frac{J_\beta(a x_{n,k}) J_\gamma(b x_{n,k})}{x_{n,k}^q J_{n+1}^2(x_{n,k})}
  = \mathbb{J} + \frac{1}{\pi} \sin \big(\tfrac{\pi}{2}(\beta +\gamma -2n-q )\big)
  \int_0^\infty \rho^{1-q} I_\beta(a\rho) I_\gamma(b \rho) \frac{K_n(\rho)}{I_n(\rho)} \dd \rho,
\end{split}
\end{equation}
where
\begin{align}\label{def:J}
  \mathbb{J} & \triangleq  \frac{1}{\pi} \sin\big(\tfrac{\pi}{2}(\gamma-\beta + q)\big)
  \int_0^\infty \rho^{1-q} I_\beta(a\rho) K_\gamma(b\rho) \dd \rho \nonumber \\
  & = \frac{a^\beta \Gamma\big(1+\tfrac{\beta+\gamma -q}{2}\big)}{2^q b^{2+\beta-q}
  \Gamma(\beta+1)\Gamma(\tfrac{\gamma-\beta +q}{2})}
  F\big(1+\tfrac{\beta+\gamma -q}{2},1+\tfrac{\beta-\gamma -q}{2}; \beta+1;\tfrac{a^2}{b^2}\big).
\end{align}
In particular, if $a=b$, it holds that
\begin{align}\label{def:J2}
  \mathbb{J}|_{a=b} = \frac{\Gamma(1 + \tfrac{\beta+\gamma -q}{2}) \Gamma(q-1)}{2^q a^{2-q} \Gamma(\tfrac{\beta +\gamma + q}{2}) \Gamma (\tfrac{\gamma-\alpha +q}{2}) \Gamma(\tfrac{\beta-\gamma +q}{2})}.
\end{align}
\end{lemma}

\begin{remark}\label{rem:sned-form}
For the explicit formula of  $\mathbb{J}$ stated  in the formula \eqref{def:J}-\eqref{def:J2},
it is enough to apply the following identities.
First, for every
$b > a >0$, $n,\beta,\gamma\in \NN$, $ q <2+\beta-\gamma$, $($see for instance $6.576$ of \cite{GR15}$)$
\begin{align*}
  \int_0^\infty \rho^{1-q} I_\beta(a\rho) K_\gamma(b\rho)\dd \rho
  =\frac{a^\beta \Gamma(1+\tfrac{\beta+\gamma-q}{2}) \Gamma(1-\frac{\gamma-\beta +q}{2})}{2^q b^{2+\beta-q}\Gamma(\beta+1)}
  F\big(1+ \tfrac{\beta+\gamma-q}{2},1+\tfrac{\beta-\gamma -q}{2};\beta+1;\tfrac{a^2}{b^2}\big).
\end{align*}
Second, it gives that in the particular case of
$a=b$, $n,\beta,\gamma\in \NN$, $1<q < 2+\beta-\gamma$, $($see for instance $9.122$ of \cite{GR15}$)$
\begin{align}\label{eq:J1-2}
  \int_0^\infty\rho^{1-q} I_\beta(a\rho) K_\gamma(a\rho) \dd \rho =
  \frac{\Gamma(1-\tfrac{\gamma -\beta + q}{2}) \Gamma(1+\tfrac{\beta +\gamma -q}{2}) \Gamma(q-1)}
  {2^q a^{2-q} \Gamma(\tfrac{\beta +\gamma + q}{2}) \Gamma(\tfrac{\beta -\gamma +q}{2})}\cdot
\end{align}
\end{remark}

\section{Boundary equation of rigid periodic patches}\label{sec:Time periodic motion of interfaces}
This section focuses on the vortex patch motion to the generalized SQG equation \eqref{Eq-gSQG}. In this setting, the solution takes at least for a short time  the form
 $\theta(t)={\bf{1}}_{D_t}$, where  $D_t\subset \DD$ is smooth and will be chosen close to the small disc  $b\DD$ ($0<b<1$).
Then identifying the complex plane $\mathbb{C}$ with $\mathbb{R}^{2}$, one might use the polar coordinates as follows,
see for instance \cite{HHM20},
\begin{equation}\label{def zk}
	z(t):\begin{array}[t]{rcl}
		\mathbb{T} & \mapsto & \partial D_{t}\\
		\theta & \mapsto & R(\theta)e^{i\theta}\triangleq\sqrt{b^{2}+2r(t,\theta)}
        e^{ i \theta}.
\end{array}
\end{equation}
We denote by $\mathbf{n}(t,z(t,\theta))={i}\partial_{\theta}z(t,\theta)$ an {inward normal }vector to the boundary $\partial D_{t}$ of the patch at the point $z(t,\theta)$.
According to \cite[p. 174]{HMV13}, the vortex patch equation writes
\begin{align*}
  \partial_{t}z(t,\theta)\cdot \mathbf{n}&=u(t,z(t,\theta))\cdot\mathbf{n}\\
  &=-\partial_{\theta}[\psi(t,z(t,\theta))],
\end{align*}
where $\psi(t,x) = (-\Delta)^{-1+ \frac{\alpha}{2}} \theta(t,x)$ is the stream function.
Notice that
\begin{align*}
  \psi(t,z(t,\theta))=\int_{{D}_t}K^\alpha (z(t,\theta),y)  \dd y.
\end{align*}
Now we shall write the patch equation in the particular case of  rotating domains $D_t=e^{it \Omega}D$ with some $\Omega\in\RR$,
that is,
\begin{align}\label{def:z-2}
  z(t,\theta)=e^{i t\Omega} z(\theta) = e^{it\Omega} \big(b^2 + 2r(\theta)\big)^{\frac{1}{2}}  e^{i\theta}.
\end{align}
Then making a change of variables we deduce from Lemma \ref{lem:HD-sym} that
\begin{align*}
  \psi(t,z(t,\theta))&=\int_D K^\alpha(e^{i t\Omega}z(\theta),e^{i t\Omega}y)  \dd y\\
  &=\int_D K^\alpha(z(\theta),y)  \dd y.
\end{align*}
In addition
\begin{align*}
  \partial_t z(t,\theta)=i\Omega z(t,\theta)=i\,\Omega\,e^{i t\Omega} \sqrt{b^{2}+2r(\theta)} e^{ i \theta},
\end{align*}
and
\begin{align*}
  \partial_{t}z(t,\theta)\cdot \mathbf{n}(t,z(t,\theta))&=\mbox{Im}\left(\partial_{t}z(t,\theta)\overline{\partial_{\theta}z(t,\theta)}\right)\\ & \stackrel{\eqref{def:z-2}}=\Omega r^\prime(\theta).
\end{align*}
Therefore we find the equation
\begin{align}\label{eq:rot-pat}
  \Omega  r^\prime(\theta)& =- \partial_\theta[\Psi(z(t,\theta))]\\
  &=- \partial_\theta\int_{{D}}K^\alpha(z(\theta),y) \dd y. \nonumber
\end{align}
Using the polar   coordinates yields
\begin{align}\label{F-psi}
\int_{{D}}K^\alpha(z(\theta),y) \dd y=\int_0^{2\pi} \int_0^{R(\eta)}K^\alpha\left(R(\theta)e^{i\theta},\rho e^{i\eta}\right) \rho
  \dd\rho \dd\eta ,\quad R(\eta)= \sqrt{b^{2}+2r(\eta)}.
\end{align}
According to Lemma \ref{lem:ker-exp}, we have the splitting
\begin{align}\label{eq:Kalp}
  K^\alpha(x,y) = K^\alpha_0(x,y) + K^\alpha_1(x,y) =  K^\alpha_0(x-y) + K^\alpha_1(x,y),
\end{align}
where $(x,y)\in \mathbb{D}^2\mapsto K^\alpha_1(x,y)$ is smooth and we make an abuse of notation
\begin{equation}\label{def:K0alp}
  K^\alpha_0(x,y) = K^\alpha_0(x-y) \triangleq \frac{4^{\alpha/2-1}\Gamma (\frac{\alpha}{2})}{\pi \Gamma(1-\frac{\alpha}{2})}\frac{1}{|x-y|^{\alpha}}
  =\frac{c_{\alpha}}{|x-y|^{\alpha}}\cdot
\end{equation}
Therefore, we get due to the fact $\nabla_xK_0^\alpha(x,y)=-\nabla_yK_0^\alpha(x,y)$,
\begin{align}\label{eq:P-F-psi} 	
  \partial_{\theta}\Psi(r(\theta))=
  &\int_0^{2\pi}\int_{0}^{R(\eta)}\nabla_x K^\alpha
  \left(R(\theta)e^{i\theta},\rho e^{i\eta}\right)\cdot
  \partial_\theta( R(\theta)e^{i\theta})\rho \dd\rho \dd\eta  \\
  =&\int_0^{2\pi} \int_0^{R(\eta)}\nabla_xK_1^\alpha
  \left(R(\theta)e^{i\theta},\rho e^{i\eta}\right)\cdot
  \partial_\theta(R(\theta)e^{i\theta}) \rho \dd\rho \dd\eta \nonumber \\
  &-\int_0^{2\pi} \int_{0}^{ R(\eta)}\nabla_yK_0^\alpha
  \left(R(\theta)e^{i\theta},\rho e^{i\eta}\right)\cdot
  \partial_\theta(R(\theta)e^{i\theta})\rho \dd\rho \dd\eta. \nonumber
\end{align}
To deal with the last integral term we apply Gauss-Green theorem,
\begin{align*}
  & \int_0^{2\pi} \int_0^{R(\eta)}\nabla_y K_0^\alpha
  \left(R(\theta)e^{i\theta},\rho e^{i\eta}\right) \cdot
  \partial_\theta\big(R(\theta)e^{i\theta}\big)\rho\, \dd\rho \dd\eta \\
  = & \iint_D \nabla_y K^\alpha_0\left(R(\theta) e^{i\theta}, y \right)\cdot
  \partial_\theta\big(R(\theta)e^{i\theta}\big)\, \dd y \\
  =&\int_0^{2\pi}K_0^\alpha \left( R(\theta)e^{i\theta},R(\eta) e^{i\eta}\right)
  \big(-i \partial_\eta( R(\eta)e^{i\eta})\big)\cdot \partial_\theta\big(R(\theta)e^{i\theta}\big) \dd\eta.
\end{align*}
Consequently, we find
\begin{equation}\label{eq:P-F-psi2}	
\begin{split}
  \partial_{\theta}\Psi(r(\theta))
  =& \int_0^{2\pi}\int_0^{R(\eta)}\nabla_xK_1^\alpha
  \left(R(\theta)e^{i\theta},\rho e^{i\eta}\right)\cdot
  \partial_\theta(R(\theta)e^{i\theta}) \rho \dd\rho \dd\eta \\
  &-  \int_0^{2\pi} K_0^\alpha\big(R(\theta)e^{i\theta} - R(\eta)e^{i\eta}\big)
  \big(-i \partial_\eta( R(\eta)e^{i\eta})\big)\cdot \partial_\theta\big(R(\theta)e^{i\theta}\big)\, \dd\eta.
\end{split}
\end{equation}
Hence the V-states equation reads as follows
\begin{equation}\label{eq-V-s}
\begin{split}
  F(\Omega,r(\theta))&\triangleq \Omega  r'(\theta)+\partial_{\theta}\Psi(r(\theta)) \\
  & \triangleq F_1(\Omega,r(\theta))+F_2(r(\theta)) =0,
\end{split}
\end{equation}
with
\begin{equation}\label{def:G1}
\begin{split}
  F_1(\Omega,r(\theta))
  & \triangleq \Omega r'(\theta) -  \int_0^{2\pi}
  K_0^\alpha\big(R(\theta)e^{i\theta} - R(\eta)e^{i\eta}\big)\,
  \big(-i \partial_\eta( R(\eta)e^{i\eta})\big)\cdot \partial_\theta\big(R(\theta)e^{i\theta}\big)\, \dd\eta \\
  & = \Omega r'(\theta) -  \int_0^{2\pi}
  K_0^\alpha\big(R(\theta)e^{i\theta} - R(\eta)e^{i\eta}\big)\,
  \mathrm{Im}\big(\partial_\eta(R(\eta)e^{i\eta}) \overline{\partial_\theta(R(\theta) e^{i\theta})}\big)\, \dd\eta
\end{split}
\end{equation}
and
\begin{align}\label{def:G2}
  F_2(r(\theta)) \triangleq \int_0^{2\pi}\int_0^{R(\eta)}\nabla_x K_1^\alpha
  \big(R(\theta)e^{i\theta},\rho e^{i\eta}\big)\cdot \partial_\theta(R(\theta)e^{i\theta}) \,\rho \dd\rho \dd\eta.
\end{align}
In the previous decomposition we recognize two terms:
the first one  $F_1$ is the same functional as in  the flat space $\mathbb{R}^2$ describing the induced patch  effect, see  (15) in \cite{HH15},
and the second one $F_2$ describes the rigid boundary effect on the patch.

\section{Linearization and regularity of the functional $F$}\label{sec:reg-F}

In order to apply the Crandall-Rabinowitz theorem stated in Theorem \ref{thm:C-R}, we need first to fix the function spaces and check the regularity of the functional $F$ introduced in \eqref{eq-V-s} with respect to these spaces. We should look for Banach spaces $X$ and $Y$ such that
$F: \RR\times X \rightarrow Y$ is well-defined and satisfies the assumptions of Theorem \ref{thm:C-R}.
Let $\alpha\in(0,1)$, $m\in\NN^+$ and consider the $m$-fold Banach  spaces
\begin{align*}
  X = X_m \triangleq \Big\{f\in C^{2-\alpha}(\mathbb{T}):\, f(\theta)=\sum_{n\geqslant 1}b_n\cos(nm\theta),\, b_n\in \RR,
  \, \theta\in\mathbb{T}\Big\}
\end{align*}
and
\begin{align*}
  Y = Y_m \triangleq \Big\{f\in C^{1-\alpha}(\mathbb{T}):\, f(\theta)=\sum_{n\geqslant 1}b_n\sin(nm\theta),\, b_n\in \RR,
  \, \theta\in\mathbb{T}\Big\}
\end{align*}
equipped  with their usual norms.
For $\epsilon_0\ll 1$, we denote by $B_{\epsilon_0}$ the open ball of $X_m$ with center $0$ and radius $\epsilon_0$,
\begin{align*}
  B_{\epsilon_0} \triangleq \big\{r\in X_m: \lVert r\rVert_{X_m} < \epsilon_0 \big\}.
\end{align*}
Recall from \eqref{eq:P-F-psi} and \eqref{eq-V-s} that
\begin{equation}\label{eq:F}
\begin{split}
  F(\Omega,r)
  & =\Omega \,r'(\theta)+\int_0^{2\pi}
  \int_0^{R(\eta)}\nabla_x K^\alpha\big(R(\theta) e^{i\theta},
  \rho e^{i\eta}\big) \cdot \Big(\frac{r'(\theta)}{R(\theta)} e^{i\theta}+ R(\theta) i e^{i\theta}\Big)
  \rho \dd\rho \dd\eta\\
  &= F_1(\Omega,r) + F_2(r)=0,
\end{split}
\end{equation}
where $\Omega \in \mathbb{R}$, $r\in X_m$ and the functionals $F_1,F_2$ are defined by \eqref{def:G1}-\eqref{def:G2}.\\
Notice that Rankine vortices are stationary solutions and therefore
\begin{equation}\label{eq:F(0)}
  F(\Omega,0)\equiv 0,\quad \forall\, \Omega \in \mathbb{R}.
\end{equation}
This can be analytically checked using the fact that the stream function is radial, which follows from the rotation invariance  of the Green function stated in Lemma \ref{lem:HD-sym} via the following identity
\begin{align}\label{eq:K-itheta}
  \nabla_x K^\alpha\big(be^{i\theta},\rho e^{i\eta}\big)\cdot (ie^{i\theta})
  = b^{-1}\partial_{\theta}G(b,\theta,\rho,\eta).
\end{align}
%
%
\subsection{Linearization}
The next goal is to linearize the nonlinear equation \eqref{eq:F} around an arbitrary small state $r.$ The computations below are formal that can be implemented from the Gateaux derivative
\begin{align*}
  \partial_r F(\Omega,r)h(\theta)=\left.\frac{\dd}{\dd s}F(\Omega,r+sh)\right|_{s=0}.
\end{align*}
They can  be rigorously checked in the Fr\'echet sense. Denoting by $B(r)(\theta)\triangleq \partial_{\theta}(R(\theta)e^{i\theta})$,
and via straightforward computations based on \eqref{eq:P-F-psi} and \eqref{eq-V-s}-\eqref{def:G2}  we get that for any $h\in X_m$,

\begin{align*}
  \partial_r F(\Omega,r)h(\theta) & = \Omega h'(\theta) +
  \int_0^{2\pi} \int_0^{R(\eta)} \nabla_x K^\alpha \big(R(\theta)e^{i\theta},
  \rho e^{i\eta}\big)\cdot \left.\frac{\dd}{\dd s}(B(r+sh)(\theta))\right|_{s=0} \rho\, \dd\rho \dd\eta \\
  & \quad + \frac{h(\theta)}{R(\theta)}\int_0^{2\pi} \int_0^{R(\eta)} \left(\nabla^2_x K_1^\alpha \big(R(\theta)e^{i\theta},
  \rho e^{i\eta}\big)\cdot e^{i\theta} \right)\cdot B(r)(\theta) \rho\, \dd\rho \dd\eta \\
  & \quad + \int_0^{2\pi}\nabla_x K_1^\alpha \big(R(\theta)e^{i\theta},R(\eta) e^{i\eta} \big)
  \cdot \partial_\theta\big( R(\theta)e^{i\theta}\big) h(\eta) \dd\eta\\
  &\quad -\int_0^{2\pi} K_0^\alpha \big(R(\theta)e^{i\theta} - R(\eta)e^{i\eta}\big)\bigg(-i \frac{\dd }{\dd s}B(r+sh)(\eta)\Big|_{s=0}\bigg)
  \cdot B(r)(\theta)\dd \eta \\
  & \quad - \int_0^{2\pi} \nabla_x K_0^\alpha\left(R(\theta)e^{i\theta} - R(\eta)e^{i\eta} \right)
  \cdot\Big(e^{i\theta} \frac{h(\theta)}{R(\theta)} - e^{i\eta} \frac{h(\eta)}{R(\eta)} \Big)
  \big(-iB(r)(\eta)\big)\cdot B(r)(\theta) \, \dd \eta.
\end{align*}
Noting that
\begin{equation*}
  \left.\frac{\dd}{\dd s} B(r+sh)(\theta)\right|_{s=0}=\partial_\theta \left(\frac{h(\theta)e^{i\theta}}{R(\theta)}\right)
  =\frac{h'(\theta)}{R(\theta)}e^{i\theta}
  + \partial_{\theta}\left(\frac{e^{i\theta}}{R(\theta)}\right)h(\theta),
\end{equation*}
and
\begin{equation}\label{eq:ted-identity}
  \partial_\eta \big(-if(\eta)e^{i\eta}\big) \cdot \partial_\theta \big(g(\theta)e^{i\theta}\big)
  = \partial_\eta \partial_\theta \big(f(\eta)g(\theta)\sin(\eta-\theta) \big),\quad \forall  f,g\in C^1,
\end{equation}
we can rewrite the above equation as
\begin{align}\label{F-Lin-r}
  \partial_r F(\Omega,r) h(\theta)& = \big[\Omega+V_1(r)(\theta)\big]
  h'(\theta)+V_2(r)(\theta)h(\theta)+ V_3(r,h)(\theta) + V_4(r,h)(\theta)
\end{align}
with
\begin{align}\label{Veq-F-1}
  V_1(r)(\theta) &
  \triangleq R^{-1}(\theta)\int_0^{2\pi} \int_{0}^{R(\eta)}
  \nabla_x K^\alpha \big(R(\theta)e^{i\theta},\rho e^{i\eta}\big)\cdot e^{i\theta} \rho \,\dd\rho \dd \eta,
\end{align}
and
\begin{equation}\label{Veq-F-2}
\begin{aligned}	
  V_2(r)(\theta)\triangleq &R^{-1}(\theta)\int_0^{2\pi}\int_0^{R(\eta)}
  \left(\nabla^2_x K_1^\alpha (R(\theta)e^{i\theta},\rho e^{i\eta})\cdot e^{i\theta}\right)\cdot
  \partial_\theta\big(R(\theta)e^{i\theta}\big)\,\rho\, \dd\rho \dd\eta \\		
 &+\int_0^{2\pi} \int_0^{R(\eta)} \nabla_x K^\alpha \big(R(\theta)e^{i\theta},
  \rho e^{i\eta}\big)\cdot \partial_\theta\left(\frac{e^{i\theta}}{R(\theta)}\right) \rho\, \dd\rho \dd\eta,
\end{aligned}
\end{equation}
and
\begin{equation}\label{Veq-F-3}
\begin{aligned}
  V_3(r,h)(\theta) \triangleq &\int_0^{2\pi}\nabla_x K_1^\alpha \big(R(\theta)e^{i\theta},R(\eta) e^{i\eta} \big)
  \cdot \partial_\theta( R(\theta)e^{i\theta}) h(\eta) \dd\eta\\
  &-\int_0^{2\pi} K_0^\alpha \big(R(\theta)e^{i\theta} - R(\eta)e^{i\eta}\big)L_A(h)\dd \eta,
\end{aligned}
\end{equation}
\begin{equation}\label{Veq-F-4}
\begin{split}
  V_4(r,h)(\theta)
  \triangleq & - \int_0^{2\pi} \nabla_x K_0^\alpha\left(R(\theta)e^{i\theta} - R(\eta)e^{i\eta} \right)
  \cdot\Big(e^{i\theta} \frac{h(\theta)}{R(\theta)} - e^{i\eta} \frac{h(\eta)}{R(\eta)} \Big) \\
  &\qquad\times
  \partial_\eta\partial_\theta\Big(R(\eta)R(\theta)\sin(\eta-\theta)\Big)\, \dd \eta,
\end{split}
\end{equation}
where
\begin{equation}\label{eq:LA}
\begin{aligned}
  L_A(h)  \triangleq \partial_\eta\partial_\theta\Big(\frac{h(\eta)R(\theta)\sin(\eta-\theta)}{R(\eta)}\Big).
\end{aligned}
\end{equation}

\subsection{Strong regularity}
This subsection is devoted to the regularity of the functional $F$ described by the formula \eqref{eq:F}. One gets the following result.
\begin{proposition}\label{prop:regularity-F1}
Let  $\alpha\in (0,1)$, there exists $\epsilon_0>0$ sufficiently small such that
the following statements hold true for any $m\in \mathbb{N}^+$.
\begin{enumerate}
\item $F:\mathbb{R}\times B_{\epsilon_0}\to Y_m$ is well-defined.
\item $F:\mathbb{R}\times B_{\epsilon_0}\to Y_m$ is of class $C^1$.
\item The partial derivative $\partial_{\Omega}\partial_r F: \mathbb{R}\times B_{\epsilon_0}\to \mathcal{L}(X_m,Y_m)$ exists and is continuous.
\end{enumerate}
\end{proposition}

\begin{proof}[Proof of Proposition $\ref{prop:regularity-F1}$]
{\bf{(i)}} We shall use the expression of $F$ detailed in \eqref{eq-V-s}, \eqref{def:G1} and \eqref{def:G2}.
First, notice that the  regularity of $\theta\in\mathbb{T}\mapsto r'(\theta) $ is obvious since  $r^\prime\in Y_m$ whenever $r\in X_m$. Second,
 since $\rVert r\lVert_{C^{2-\alpha}}\le \epsilon_0\ll 1$, then
$x= R(\theta)e^{i\theta} = \sqrt{b^2 + 2r(\theta)} e^{i\theta}$
is in the compact set $B(0,\sqrt{b^2 +2\epsilon_0}) \subset \DD$.
Next we shall  check the regularity of  the second part of  $F_1$ given by \eqref{def:G1}. Set
\begin{align}\label{eq:K0-Hold}
  F_{1,1}(r(\theta)) & \triangleq \int_0^{2\pi}
  K_0^\alpha\big(R(\theta)e^{i\theta} - R(\eta)e^{i\eta}\big)\,
  \big(-i \partial_\eta( R(\eta)e^{i\eta})\big)\cdot \partial_\theta\big(R(\theta)e^{i\theta}\big)\, \dd\eta \nonumber \\
  & = \int_0^{2\pi} K^\alpha_0\big(R(\theta)e^{i\theta}- R(\eta)e^{i\eta} \big)
  \big( -i \partial_\eta (R(\eta)e^{i\eta})\big) \dd \eta \cdot \partial_\theta \big(R(\theta) e^{i\theta}\big).
\end{align}
Since $r\in C^{1-\alpha}(\TT)$ then using the law products we deduce that
\begin{equation}\label{eq:par-th-R}
 \theta\in \mathbb{T}\mapsto  \partial_\theta(R(\theta)e^{i\theta}) = \left( \tfrac{r'(\theta)}{R(\theta)}e^{i\theta}+i\, R(\theta)e^{i\theta}\right) \in C^{1-\alpha}(\TT).
\end{equation} Now, to analyze the regularity of the first term in the right-hand side of \eqref{eq:K0-Hold} it suffices to combine   Lemma \ref{lem:reg-1} with the estimates below
\begin{equation}\label{eq:K-d}
  \forall \theta\neq \eta \in\TT,\quad   |K^{\alpha}_0(R(\theta)e^{i\theta},R(\eta)e^{i\eta})|
  \leqslant C\left|\sin \tfrac{\theta-\eta}{2}\right|^{-\alpha},
\end{equation}
and
\begin{align}\label{eq:K0-d2}
  |\partial_\theta K^{\alpha}_0(R(\theta)e^{i\theta},R(\eta)e^{i\eta})|   & \leqslant
  C\left|\sin \tfrac{\theta-\eta}{2}\right|^{-(1+\alpha)},
\end{align}
in order to get
 \begin{align}\label{eq:K0-Hold11}
  \theta\in\mathbb{T}\mapsto  \int_0^{2\pi} K^\alpha_0\big(R(\theta)e^{i\theta}, R(\eta)e^{i\eta} \big)
  \big( -i \partial_\eta (R(\eta)e^{i\eta})\big) \dd \eta \in C^{1-\alpha}(\mathbb{T}).
\end{align}
Thus the classical law products give in view of \eqref{eq:K0-Hold}, $ F_{1,1}(r)\in C^{1-\alpha}(\mathbb{T})$ and then  $ F_{1}(r)\in C^{1-\alpha}(\mathbb{T})$. Now, let us check how to get \eqref{eq:K-d} and \eqref{eq:K0-d2}.
 The first  step is to show the following
\begin{equation}\label{eq:R-lip}
  \forall\, \theta,\eta\in\mathbb{R},\quad C_1\left|\sin \tfrac{\theta-\eta}{2}\right|\leqslant |R(\theta)e^{i\theta}-R(\eta)e^{i\eta}|\leqslant C_2\left|\sin \tfrac{\theta-\eta}{2}\right|,
\end{equation}
with some constants $0<C_1 \leqslant C_2 .$ To do so, we write
\begin{align*}
  R(\theta)e^{i\theta} - R(\eta)e^{i\eta} = b e^{i\eta} \big(e^{i(\theta -\eta)} - 1\big)
  + \Big(\big(R(\theta)e^{i\theta}
  -b e^{i\theta} \big) - \big(R(\eta)e^{i\eta} -b e^{i\eta} \big) \Big),
\end{align*}
\begin{align*}
  |\partial_\theta\big(R(\theta)e^{i\theta}-b e^{i\theta}\big)| \leqslant
  \frac{|r| + |r'(\theta)|}{\sqrt{b^2 - 2\epsilon_0}}\leqslant \frac{2\epsilon_0}{\sqrt{b^2 - 2\epsilon_0}},
\end{align*}
and
\begin{equation}\label{eq:sin-theta}
  \tfrac{2|\theta|}{\pi }\leqslant |e^{i\theta}-1|=|2\sin \tfrac{\theta}{2}|\leqslant |\theta|, \quad \text{ for }\; |\theta|\leqslant \pi,
\end{equation}
that we combine with Taylor's formula in order to get
\begin{equation}\label{es:Rthe}
  C_1\Big|\sin\tfrac{\theta -\eta}{2}\Big|\leqslant  |R(\theta)e^{i\theta}-R(\eta)e^{i\eta}|\leqslant C_2\Big|\sin\tfrac{\theta -\eta}{2}\Big|,
  \quad \forall\,|\theta-\eta|\leqslant \pi.
\end{equation}
For the general case  $\theta,\eta \in \mathbb{R}$ there exists a $k_0$
such that $|\theta+2 k_0\pi-\eta|\leqslant \pi$ and by periodicity we infer $R(\theta+2k_0\pi)e^{i(\theta+2k_0\pi)}=R(\theta)e^{i\theta}$, and then  \eqref{es:Rthe} applies, leading to \eqref{eq:R-lip}.
%
Therefore using  \eqref{eq:R-lip} and the explicit formula of $K^\alpha_0$ given in \eqref{def:K0alp}, we  find \eqref{eq:K-d}
and by differentiation
\begin{align*}
  |\partial_\theta K^{\alpha}_0(R(\theta)e^{i\theta},R(\eta)e^{i\eta})| &\leqslant |\nabla_x K_0^\alpha (R(\theta)e^{i\theta},
  R(\eta) e^{i\eta})| |\partial_\theta(R(\theta)e^{i\theta})|\nonumber\\
  & \leqslant   C\left|\sin \tfrac{\theta-\eta}{2}\right|^{-(1+\alpha)},
\end{align*}
achieving \eqref{eq:K0-d2}. The next task is to show that the functional
$F_2$ given by \eqref{def:G2} belongs to $C^{1-\alpha}(\mathbb{T}).$ Recall that \eqref{eq:par-th-R} holds true, then to get the suitable regularity for $F_2$ it is enough to show that
\begin{equation}\label{eq:G2-Hold}
 \theta\in\mathbb{T}\mapsto  F_{2,1}(r(\theta))\triangleq  \int_0^{2\pi}\int_0^{R(\eta)}\nabla_xK_1^\alpha(R(\theta) e^{i\theta},\rho e^{i\eta})
  \rho \,\dd\rho \dd\eta \in C^{1-\alpha}(\TT).
\end{equation}
Since $K_1^\alpha$ is smooth inside the domain $\mathbb{D}^2$ as in Lemma \ref{lem:ker-exp} and the patch is far away the boundary $\mathbb{T}=\partial \mathbb{D}$
then  it is plain that $F_{2,1}(r)\in L^\infty(\TT).$
To establish the H\"{o}lder regularity, we consider two points  $x_1= R(\theta_1) e^{i\theta_1}$ and $x_2 = R(\theta_2) e^{i\theta_2}$,
and write by the mean value theorem and the boundedness of $\nabla_x^2 K_1^\alpha$ in any compact set of $\mathbb{D}^2,$
\begin{align}\label{es:nab-K1}
  |F_{2,1}(r(\theta_1))-F_{2,1}(r(\theta_2))|
  \leqslant & \int_0^{2\pi}\int_0^{ R(\eta)}
  |\nabla_x K_1^\alpha(x_1,\rho e^{i\eta})-\nabla_x K_1^\alpha( x_2,\rho e^{i\eta})| \rho \dd\rho \dd\eta \nonumber \\
  \leqslant & |x_1 - x_2| \int_0^1\iint_{B(0,\sqrt{b^2 + 2\epsilon_0})}
  |\nabla_x^2 K_1^\alpha(\kappa x_1 + (1-\kappa) x_2,y)| \dd y \dd\kappa \nonumber \\
  \leqslant & C|R(\theta_1) e^{i\theta_1}-R(\theta_2) e^{i\theta_2}|,
\end{align}
where $C>0$ is independent of $x_1,x_2$. Therefore, applying \eqref{eq:R-lip} gives that $F_{2,1}(r)$ is Lipschitz and consequently it belongs to the space $C^{1-\alpha}$ achieving the proof of the claim \eqref{eq:G2-Hold}.
To sum, we have proven that the functions
$\theta\in\mathbb{T}\mapsto \partial_{\theta}\Psi(r(\theta)), F(\Omega,r(\theta)) $ belong to $C^{1-\alpha}(\mathbb{T})$. It remains to check the symmetry for these functions in order to be in the space $Y_m$. According to the first line of \eqref{eq-V-s}, it suffices to show  the symmetry for the function $\theta\in\mathbb{T}\mapsto \partial_{\theta}\Psi(r(\theta)).$

Now we want to  show that $\partial_\theta \Psi(r(\theta))$ and $F(\Omega,r(\theta))$ has the series expansion as in $Y_m$.
In light of \eqref{F-psi}, statement (i) of Lemma \ref{lem:HD-sym} and the fact that $R(-\eta)=R(\eta)$,
$\forall \eta\in \mathbb{R}$, one obtains
\begin{align*}
  \Psi(r(-\theta))&=\int_0^{2\pi} \int_{0}^{R(\eta)}K^{\alpha}(R(-\theta)e^{-i\theta},R(\eta)e^{i\eta})\rho \,\dd\rho \dd\eta\\
  &=\int_0^{2\pi} \int_0^{R(\eta)}K^{\alpha}(R(\theta)e^{-i\theta},R(\eta)e^{-i\eta})\rho\, \dd\rho  \dd\eta \\
  &=\int_0^{2\pi} \int_0^{R(\eta)}K^{\alpha}(R(\theta)e^{i\theta},R(\eta)e^{i\eta})\rho \, \dd\rho \dd\eta \\
  &=\Psi(r(\theta)),
\end{align*}
which implies that
\begin{align*}
  \Psi(r(\theta))=\sum_{n=0}^\infty c_n\cos n \theta,\quad\textrm{with}\;\;  c_n\in \mathbb{R}.
\end{align*}
Thanks to the statement (ii) of Lemma \ref{lem:HD-sym}, the fact that $R(\eta)=R(\eta+\frac{2\pi}{m})$ for every $\eta\in\RR$ and the change of variable $\eta\mapsto \eta+\tfrac{2\pi}{m}$
we deduce by elementary operations
\begin{align*}
  \Psi\Big(r\big(\theta+\tfrac{2\pi}{m}\big)\Big)
  & = \int_0^{2\pi} \int_0^{R(\eta)} K^\alpha\Big(R\big(\theta+\tfrac{2\pi}{m}\big)
  e^{i(\theta+\frac{2\pi}{m})},R(\eta)e^{i\eta}\Big)\rho  \, \dd\rho \dd\eta \\
  &=\int_0^{2\pi} \int_0^{R(\eta)} K^\alpha\big(R(\theta)e^{i(\theta+\frac{2\pi}{m})},R(\eta)e^{i(\eta+\frac{2\pi}{m})}\big)
  \rho \,\dd\rho \dd\eta \\
  &= \Psi(r(\theta)).
\end{align*}
Thus
\begin{align*}
  \Psi(r(\theta))=\Psi\Big(r\big(\theta+\tfrac{2\pi}{m}\big)\Big) & = \sum_{k=0}^{+\infty}c_n
  \cos\Big( n\theta+ \tfrac{2n\pi}{m}\Big) \\
  &=\sum_{k=0}^{+\infty} c_n \Big(\cos (n\theta)\,\cos \big(\tfrac{2n\pi}{m}\big)
  - \sin(n\theta)\, \sin\big(\tfrac{2n\pi}{m} \big) \Big).
\end{align*}
By the uniqueness of Fourier coefficients, we infer that
$\cos \big(\frac{2n\pi}{m}\big) =1$ and $\sin \big(\frac{2\pi n}{m} \big) = 0$.
Hence
\begin{equation*}
  \Psi(r(\theta))=\sum_{n=0}^{+\infty} c_{nm} \cos (nm \theta),\quad \textrm{with}\;\; c_{nm}\in \mathbb{R}.
\end{equation*}
Consequently $\partial_{\theta}\Psi(r)\in Y_m$ and therefore $F: \RR\times B_{\epsilon_0}\rightarrow Y_m$ is well-defined.

{\bf{(ii)}} It amounts to showing that the partial derivatives $\partial_{\Omega}F$ and $\partial_rF$
exist in the Frechet  sense and they are both continuous. For $\partial_\Omega F$, it is obvious to see that
\begin{align*}
  \partial_{\Omega}F(\Omega,r)(\theta)= r'(\theta),
\end{align*}
and thus $\partial_\Omega F $ is a linear bounded operator from $X_m$ to $Y_m$ and is independent of $\Omega$.
Then we only need to check that $\partial_r F(\Omega,r)\in \mathcal{L}(X_m,Y_m)$ and it is continuous.
In light of \eqref{F-Lin-r}-\eqref{Veq-F-4}, we have the decomposition
\begin{equation}\label{eq:pa-rF-dec}
\begin{split}
  \partial_r F(\Omega,r) h(\theta) = \big[\Omega+V_1(r)(\theta)\big] h'(\theta)
  +  V_2(r)(\theta)h(\theta) + V_3(r,h)(\theta)+ V_4(r,h)(\theta).
\end{split}
\end{equation}	
For the first term on the right-hand side of \eqref{eq:pa-rF-dec}, we may argue as for the estimate of  $\partial_{\theta}\Psi(r(\theta))$ developed in the preceding point \textbf{(i)}, leading to
\begin{equation*}
 \theta\mapsto  \int_0^{2\pi} \int_0^{R(\eta)}(\nabla_x K^\alpha)\left(R(\theta)e^{i\theta},\rho e^{i\eta}\right)\cdot
  e^{i\theta} \rho \,\dd\rho \dd \eta \in C^{1-\alpha}(\TT),
\end{equation*}
and thus combined with the fact that $R^{-1}(\theta)\in C^{2-\alpha}$,
we deduce from the law products that
\begin{align}\label{es:I1}
 \nonumber  \lVert (\Omega + V_1(r))h'\rVert_{C^{1-\alpha}}&\lesssim \lVert h'\rVert_{C^{1-\alpha}}
  +\lVert V_1(r)\rVert_{C^{1-\alpha}}\lVert h'\rVert_{C^{1-\alpha}}\\
  &   \lesssim \lVert h\rVert_{C^{2-\alpha}}.
\end{align}
For the second term $V_2(r)h$, described in \eqref{Veq-F-2},  one easily obtains
\begin{align*}
  \lVert V_2(r)h \rVert_{C^{1-\alpha}} \leqslant C \lVert V_2(r)\rVert_{C^{1-\alpha}} \lVert h\rVert_{C^{1-\alpha}}.
\end{align*}
To check that $V_2(r)$ belongs to $C^{1-\alpha}(\mathbb{T})$,
it is enough to show that
\begin{align*}
  \theta\mapsto \int_0^{2\pi}\int_0^{R(\eta)}\left(\nabla_x^2 K^\alpha_1\big(R(\theta)e^{i\theta},\rho e^{i\eta}\big)\cdot
  e^{i\theta}\right)\cdot \partial_\theta\big(R(\theta)e^{i\theta}\big)\rho\, \dd\rho \dd\eta \in C^{1-\alpha}(\mathbb{T}),
\end{align*}
since the treatment of the remaining  term in \eqref{Veq-F-2} is quite similar to $V_1(r)$.
This latter term is easily estimated from the product law since $K_1^\alpha$ is highly smooth inside the domain $\mathbb{D}^2.$
For the third term $V_3(r,h)$ given by \eqref{Veq-F-3}, one can easily estimate as before  the first term connected to $K_1^\alpha$: \begin{align*}
  \left\lVert\theta\mapsto \int_{0}^{2\pi}\nabla_xK^{\alpha}_1 \big(R(\theta)e^{i\theta},R(\eta) e^{i\eta}\big) \cdot \partial_\theta\big(R(\theta)e^{i\theta}\big)h(\eta)\,\dd \eta\right\rVert_{C^{1-\alpha}}\leqslant C\lVert h\rVert_{L^{\infty}}.
\end{align*}
As to  the remaining term in $V_3(r,h)$,
noting that
\begin{align*}
  L_A(h) = \mathrm{Im}\Big(\partial_\eta\big( \tfrac{h(\eta)}{R(\eta)} e^{i\eta} \big)\,
  \overline{\partial_\theta(R(\theta) e^{i\theta})}\Big),
\end{align*}
one finds by using \eqref{eq:par-th-R}, \eqref{eq:K-d}
and Lemma \ref{lem:reg-1},
\begin{align*}
  & \left\lVert\theta\mapsto \int_0^{2\pi}K_0^{\alpha}\big(R(\theta)e^{i\theta} - R(\eta)e^{i\eta}\big)L_A(h)\dd\eta\right\rVert_{C^{1-\alpha}} \\
  & \lesssim \left\lVert \int_0^{2\pi} K_0^\alpha\big(R(\theta)e^{i\theta} - R(\eta)e^{i\eta}\big)
 \partial_\eta\left(\tfrac{h(\eta)e^{i\eta}}{R(\eta)} \right) \dd \eta \right\rVert_{C^{1-\alpha}}
  \|\partial_\theta\big(R(\theta)e^{-i\theta} \big)\|_{C^{1-\alpha}} \\
  & \lesssim \lVert h\rVert_{C^1}.
\end{align*}
Combining with the above two estimates leads to
\begin{align}\label{eq:I-3-bound}
  \lVert V_3(r,h)\rVert_{C^{1-\alpha}}\leqslant C\lVert h\rVert_{C^1}.
\end{align}
Now we focus on the last term $V_4(r,h)$ given by \eqref{Veq-F-4}.
In view of \eqref{eq:ted-identity}, we write it as
\begin{align*}
  V_4(r,h) = & \, \int_0^{2\pi} \widetilde{H}(\theta,\eta) \big(i\partial_\eta(R(\eta)e^{i\eta})\big) \dd\eta
  \cdot \big(\partial_\theta \big(R(\theta)e^{i\theta}\big)\big) ,
\end{align*}
with
\begin{align}\label{eq:tildH}
  \widetilde{H}(\theta,\eta) \triangleq \nabla_x K_0^\alpha\big(R(\theta)e^{i\theta} - R(\eta) e^{i\eta}\big)\cdot
  \left(e^{i\theta} \frac{h(\theta)}{R(\theta)}-e^{i\eta} \frac{h(\eta)}{R(\eta)}\right).
\end{align}
From the mean value theorem we infer
\begin{align*}
  \left| e^{i\theta} \frac{h(\theta)}{R(\theta)} - e^{i\eta} \frac{h(\eta)}{R(\eta)}\right|
  \leqslant C|\theta-\eta|\lVert h\rVert_{C^1},
\end{align*}
and since $h$ and $R$ are $2\pi$-periodic functions then we  we can argue as \eqref{eq:R-lip} in order
to get
\begin{align}\label{eq:h-R-diff}
  \left| e^{i\theta} \frac{h(\theta)}{R(\theta)}-e^{i\eta} \frac{h(\eta)}{R(\eta)}\right|
  \leqslant C\left|\sin \tfrac{\theta-\eta}{2}\right| \lVert h\rVert_{C^1}.
\end{align}
Combining this  with  \eqref{def:K0alp} and \eqref{eq:K-d}-\eqref{eq:K0-d2} allows to get
\begin{align*}
  |\widetilde{H}(\theta,\eta)| \leqslant  C\frac{\lVert h\rVert_{C^1}}{\big|\sin \tfrac{\theta-\eta}{2}\big|^\alpha},\quad
  |\partial_\theta \widetilde{H}(\theta,\eta)| \leqslant C
  \frac{\lVert h\rVert_{C^1}}{\big|\sin \tfrac{\theta-\eta}{2}\big|^{1+\alpha}},\quad \forall \eta\neq \theta \in \TT.
\end{align*}
Applying Lemma \ref{lem:reg-1} and \eqref{eq:par-th-R}, we derive that
\begin{equation}\label{eq:I-4-bound}
  \lVert V_4(r,h)\rVert_{C^{1-\alpha}}\leqslant C\lVert h\rVert_{C^1}.
\end{equation}
In conclusion, gathering the above estimates, we prove that
\begin{equation*}
  \lVert \partial_rF(\Omega,r)h\rVert_{C^{1-\alpha}}\leqslant C\lVert h\rVert_{C^{2-\alpha}},
\end{equation*}
which implies that $\partial_r F(\Omega,r)\in \mathcal{L}(X_m,Y_m)$.

The next step is to prove that for given $\Omega\in\mathbb{R}$, $\partial_rF(\Omega, r)$
is a continuous mapping taking values in the space of bounded linear operators from $X_m$ to $Y_m$.
In other words, we will show that, for every $r_1$, $r_2\in B_{\epsilon_0}\subset X_m$,
\begin{equation*}
  \sup_{\lVert h\rVert_{C^{2-\alpha}}\leq 1}\lVert \partial_rF(\Omega,r_1)h-\partial_rF(\Omega,r_2)h\rVert_{C^{1-\alpha}}
  \to 0,  \quad \text{as }\;\lVert r_1-r_2\rVert_{C^{2-\alpha}} \to 0.
\end{equation*}
Thanks to \eqref{eq:pa-rF-dec} and the algebra structure of H\"older spaces, this can be guaranteed by the following continuous result that
as $\lVert r_1-r_2\rVert_{C^{2-\alpha}}\to 0$,
\begin{equation}\label{eq:Ij-cont}
  \sum_{j=1}^2\|V_j(r_1)-V_j(r_2)\|_{C^{1-\alpha}} +
  \sup_{\lVert h\rVert_{C^{2-\alpha}}\leq 1} \sum_{j=3}^{4} \lVert V_j(r_1,h) - V_j(r_2,h)\rVert_{C^{1-\alpha}}\to 0.
\end{equation}
To prove the continuity  result \eqref{eq:Ij-cont} regarding $V_1$, given by \eqref{Veq-F-1}, we proceed first  in a similar way
as in the derivation \eqref{eq:P-F-psi2}, by writing
\begin{align}\label{eq:V1-dec}
  V_1(r)(\theta) 
  =&R^{-1}(\theta)\int_0^{2\pi}
  \int_0^{R(\eta)} \big(-\nabla_y K^\alpha_0 + \nabla_x K^\alpha_1\big)\big(R(\theta)e^{i\theta} ,\rho e^{i\eta}\big)
  \cdot  e^{i\theta}\rho \, \dd\rho \dd\eta \nonumber \\
  =& \frac{1}{R(\theta)}\int_0^{2\pi} \int_0^{R(\eta)}
  \nabla_x K^\alpha_1 (R(\theta)e^{i\theta},\rho e^{i\eta})\cdot e^{i\theta}\rho \,\dd\rho \dd\eta \nonumber \\
  & + \frac{1}{R(\theta)} \int_0^{2\pi} K^\alpha_0\big(R(\theta)e^{i\theta} - R(\eta)e^{i\eta}\big) \,
  \big(i \partial_\eta(R(\eta)e^{i\eta})\big) \cdot e^{i\theta} \, \dd\eta \nonumber \\
  \triangleq &\, \frac{V_{11}(r)(\theta) + V_{12}(r)(\theta)}{R(\theta)},
\end{align}
with $R(\theta) \triangleq \sqrt{b^2 + 2 r(\theta)}$.
The continuity of the $r\mapsto R$ is immediate. Then it remains to explore the continuity of the mappings $r\mapsto V_{11}(r),  V_{12}(r)$. With the notation  $R_j(\theta) \triangleq \sqrt{b^2 + 2 r_j(\theta)}$
one gets
\begin{align*}
  & V_{11}(r_1)-V_{11}(r_2)  = \int_0^{2\pi}\int_{R_2(\eta)}^{R_1(\eta)}
  \nabla_x K^\alpha_1 \big(R_1(\theta) e^{i\theta},\rho e^{i\eta}\big) \cdot e^{i\theta} \rho \,\dd\rho \dd\eta  \\
  & \quad +
  \int_0^{2\pi} \int_0^{R_2(\eta)} \Big(\nabla_x K^\alpha_1\big(R_1(\theta)e^{i\theta},\rho e^{i\eta}\big) -
  \nabla_x K^\alpha_1 \big(R_2(\theta) e^{i\theta},\rho e^{i\eta}\big)\Big)\cdot e^{i\theta} \rho\, \dd\rho \dd\eta .
\end{align*}
In view of the fact that $|R_1(\theta)-R_2(\theta)| \leqslant \frac{|r_1(\theta) -r_2(\theta)|}{\sqrt{b^2-2\epsilon_0}}$
and as $K_1^\alpha$ is smooth in $\mathbb{D}^2$, we can use the mean value theorem in order to get
\begin{align*}
  \|V_{11}(r_1) - V_{11}(r_2)\|_{L^\infty} \leqslant C \|r_1 -r_2\|_{L^\infty}.
\end{align*}
Similarly, by $|R_1'(\theta) -R_2'(\theta)| \leqslant C \frac{\|r_1 - r_2\|_{C^1}}{\sqrt{b^2 -2\epsilon_0}} $,
we obtain
\begin{align*}
  \|V_{11}(r_1) - V_{11}(r_2)\|_{C^1} \leqslant C \|r_1 - r_2\|_{C^1}.
\end{align*}
Hence we find by Sobolev embedding,
\begin{align*}
  \|V_{11}(r_1) - V_{11}(r_2)\|_{C^{1-\alpha}} \leqslant C \|r_1 - r_2\|_{C^{2-\alpha}}.
\end{align*}
For $V_{12}$ defined by \eqref{eq:V1-dec}, we write
\begin{align*}
  V_{12}(r_1)-V_{12}(r_2)
  &=  \int_0^{2\pi} H(\theta,\eta)\,\partial_\eta\big(iR_1(\eta)e^{i\eta}\big)\cdot e^{i\theta} \dd \eta \\
  &\quad + \int_0^{2\pi} K^\alpha_0
  \big(R_2(\theta)e^{i\theta} - R_2(\eta)e^{i\eta}\big) \,
  \partial_\eta\big(iR_1(\eta)e^{i\eta}-iR_2(\eta)e^{i\eta}\big)\cdot e^{i\theta} \, \dd\eta \\
  & \triangleq \, J_1 + J_2 ,
\end{align*}
where
\begin{align}\label{eq:kernelH}
  H(\theta,\eta) \triangleq K_0^\alpha\big(R_1(\theta)e^{i\theta} - R_1(\eta)e^{i\eta}\big)
  - K_0^\alpha\big(R_2(\theta)e^{i\theta} - R_2(\eta)e^{i\eta}\big).
\end{align}
To estimate $J_2$ we proceed as in   the estimate \mbox{of  \eqref{eq:K0-Hold},}  leading to
\begin{align*}
  \left\lVert J_2\right\rVert_{C^{1-\alpha}}
  &\leqslant C \|\partial_\eta \big(R_1(\eta)e^{i\eta} - R_2(\eta)e^{i\eta}\big)\|_{L^\infty}\\
  &\leqslant C\lVert r_1-r_2\rVert_{C^{1}}.
\end{align*}
As to  the first term $J_1$,
We will apply Lemma \ref{lem:reg-1} to control the above right-hand term.
Note that
\begin{align*}
   \big |\big(R_1(\theta) - R_2(\theta)\big)e^{i\theta} - \big(R_1(\eta)-R_2(\eta)\big) e^{i\eta}\big| &
  \leqslant \big \lVert \partial_\theta\big((R_1(\theta) - R_2(\theta))e^{i\theta}\big) \big\rVert_{L^{\infty}} |\theta-\eta|\\
  &\leqslant  C\lVert r_1-r_2\rVert_{C^1}|\theta-\eta|,
\end{align*}
which implies that (owing to \eqref{es:Rthe} and the $2\pi$-periodicity of $R_i$)
\begin{align}\label{eq:R-diff}
   \nonumber \big| |R_1(\theta)e^{i\theta} - R_1(\eta)e^{i\eta}| - |R_2(\theta)e^{i\theta} - R_2(\eta) e^{i\eta}|\big|  \leqslant& \big |(R_1(\theta) - R_2(\theta))e^{i\theta}-(R_1(\eta)-R_2(\eta))e^{i\eta}\big|\\\leqslant &C\lVert r_1-r_2\rVert_{C^1}\left|\sin \tfrac{\theta-\eta}{2}\right|.
\end{align}
Using \eqref{def:K0alp}, \eqref{eq:R-lip}, \eqref{eq:R-diff} and the following inequality
\begin{align*}
  |A^\alpha - B^\alpha|\leqslant C_\alpha \big( A^{\alpha -1}+B^{\alpha -1}\big) |A-B|, \,\quad \forall A,B>0,
\end{align*}
applied with $A=|R_1(\theta)-R_1(\eta)|^{-1}$ and $B=|R_2(\theta)-R_2(\eta)|^{-1}$
we obtain
\begin{align*}
  |H(\theta,\eta)| \leqslant C\frac{ \lVert r_1-r_2\rVert_{C^1}}{\big|\sin \tfrac{\theta-\eta}{2}\big|^\alpha}.
\end{align*}
On the other hand, straightforward computations yield the splitting
\begin{align*}		
  \partial_{\theta}H(\theta,\eta) & = \nabla_x K^\alpha_0\big(R_1(\theta)e^{i\theta} - R_1(\eta)e^{i\eta} \big)\cdot
  \partial_\theta\big(R_1(\theta)e^{i\theta}\big)\\
  &\quad -  \nabla_x K^\alpha_0\big(R_2(\theta)e^{i\theta} - R_2(\eta)e^{i\eta} \big)\cdot
  \partial_\theta\big(R_2(\theta)e^{i\theta}\big)  \\
  & = H_1(\theta,\eta)+ H_2(\theta,\eta)+ H_3(\theta,\eta)
\end{align*}
where
\begin{align*}
  H_1(\theta,\eta)&\triangleq\, -\alpha c_\alpha\frac{(R_1(\theta)-R_2(\theta))e^{i\theta}-(R_1(\eta)-R_2(\eta))e^{i\eta}}
  {|R_1(\theta)e^{i\theta}-R_1(\eta)e^{i\eta}|^{\alpha+2}}\cdot \partial_\theta\big(R_1(\theta)e^{i\theta}\big), \\
  H_2 (\theta,\eta)& \triangleq \, -\alpha c_\alpha\frac{R_2(\theta)e^{i\theta}-R_2(\eta)e^{i\eta}}
  {|R_2(\theta)e^{i\theta}-R_2(\eta)e^{i\eta}|^{\alpha+2}}\cdot
  \partial_\theta\big(R_1(\theta)e^{i\theta}-R_2(\theta)e^{i\theta}\big),
\end{align*}
\begin{align*}
  H_3(\theta,\eta)& \triangleq \,\alpha c_\alpha \,\partial_\theta
  \big(R_1(\theta)e^{i\theta}\big)\cdot (R_2(\theta)e^{i\theta}-R_2(\eta)e^{i\eta}) \\
  &\quad \times \frac{|R_1(\theta)e^{i\theta}-R_1(\eta)e^{i\eta}|^{\alpha+2}-|R_2(\theta)e^{i\theta}-R_2(\eta)e^{i\eta}|^{\alpha+2}}
  {|R_1(\theta)e^{i\theta}-R_1(\eta)e^{i\eta}|^{\alpha+2}|R_2(\theta)e^{i\theta}-R_2(\eta)e^{i\eta}|^{\alpha+2}}.
\end{align*}
For $H_1$ and $H_2$, using \eqref{eq:R-lip} and \eqref{eq:R-diff}, we readily infer that
\begin{align}\label{eq:K2122}
  |H_1(\theta,\eta)|+|H_2(\theta,\eta)|\leqslant C
  \frac{\lVert r_1-r_2\rVert_{C^1}}{\big|\sin \frac{\theta-\eta}{2}\big|^{1+\alpha}},\quad \forall \theta\neq \eta\in\TT.
\end{align}
Concerning the last term $H_3$, by virtue of the following inequality 
\begin{align}\label{ineq:k-alpha}
  |A^{k+\alpha}-B^{k+\alpha}|\leqslant C_{k,\alpha}|A-B|(A^{k+\alpha-1}+B^{k+\alpha-1}),\quad A,B\geqslant 0, k\in \mathbb{N},\, 0<\alpha<1,
\end{align}
and using \eqref{eq:R-lip} and \eqref{eq:R-diff}, we find
\begin{align}\label{eq:K23}
  |H_3(\theta,\eta)|\leqslant C\frac{\lVert r_1-r_2\rVert_{C^1}}{\big|\sin \frac{\theta-\eta}{2}\big|^{1+\alpha}},
  \quad \forall \theta\neq \eta\in\TT.
\end{align}
Putting together \eqref{eq:K2122} and \eqref{eq:K23} leads to the wanted estimate of $\partial_\theta H(\theta,\eta)$.
Hence, the kernel $H$ satisfies the required assumptions in Lemma \ref{lem:reg-1}, and consequently,
\begin{align*}
  \|J_1\|_{C^{1-\alpha}} \leqslant C \|r_1 -r_2 \|_{C^1} \|\partial_\eta \big(R_1(\eta)e^{i\eta}\big)\|_{L^\infty}
  \leqslant C \|r_1 -r_2 \|_{C^1} .
\end{align*}
Hence, we conclude that $r\in X_m\mapsto (h\mapsto V_1(r)h)$ as a mapping from $X_m$ to $\mathcal{L}(X_m, Y_m)$ is continuous.\\
Concerning the continuity of $r\mapsto V_2(r)$ introduced in \eqref{Veq-F-2} it can be checked in a similar way to $V_1$ discussed before. Therefore we will skip it. \\
Concerning the continuity of  $r\mapsto V_3(r,h)$ given by \eqref{Veq-F-3}, we start with   the first term
\begin{align*}
  J_4 & \triangleq \int_0^{2\pi}\nabla_x K^\alpha_1\left(R_1(\theta)e^{i\theta},R_1(\eta) e^{i\eta}\right)\cdot
  \partial_\theta\big(R_1(\theta)e^{i\theta}\big)h(\eta)\dd \eta \\
   & \quad -\int_0^{2\pi}\nabla_x K^\alpha_1 \left(R_2(\theta)e^{i\theta},R_2(\eta) e^{i\eta}\right)\cdot
  \partial_\theta\big(R_2(\theta)e^{i\theta}\big)h(\eta)\dd \eta,
\end{align*}
with $R_j(\theta) = \sqrt{b^2 + 2 r_j(\theta)}$, $j=1,2$.
It is easy to see
\begin{align*}
  J_4 & = \int_0^{2\pi}\nabla_x K^\alpha_1\left(R_1(\theta)e^{i\theta},R_1(\eta) e^{i\eta}\right)\cdot
  \Big(\partial_\theta\big(R_1(\theta)e^{i\theta}\big) - \partial_\theta\big(R_2(\theta)e^{i\theta} \big)
  \Big)h(\eta)\dd \eta \\
  & \quad + \int_0^{2\pi} \Big(\nabla_x K^\alpha_1 \big(R_1(\theta)e^{i\theta}, R_1(\eta)e^{i\eta}\big)
  - \nabla_x K^\alpha_1\big(R_2(\theta)e^{i\theta}, R_1(\eta)e^{i\eta}\big) \Big)\cdot
  \partial_\theta\big(R_2(\theta)e^{i\theta}\big)h(\eta)\dd \eta \\
  & \quad + \int_0^{2\pi} \Big(\nabla_x K^\alpha_1 \big(R_2(\theta)e^{i\theta}, R_1(\eta)e^{i\eta}\big)
  - \nabla_x K^\alpha_1\big(R_2(\theta)e^{i\theta}, R_2(\eta)e^{i\eta}\big) \Big)\cdot
  \partial_\theta\big(R_2(\theta)e^{i\theta}\big)h(\eta)\dd \eta \\
  & \triangleq J_{4,1} + J_{4,2} + J_{4,3}.
\end{align*}
For $J_{41}$, we use the following estimate
\begin{align}\label{eq:pa-R1-R2}
   \nonumber\|\partial_\theta\big(R_1(\theta)e^{i\theta}\big) - \partial_\theta\big(R_2(\theta)e^{i\theta}\big) \|_{C^{1-\alpha}}
  & \leq \|R_1-R_2\|_{C^{1-\alpha}} + C \|R_1^{-1}\|_{C^{1-\alpha}} \|r_1' - r_2'\|_{C^{1-\alpha}}
  \\
  &\quad + C \|R_1^{-1} - R_2^{-1}\|_{C^{1-\alpha}} \|r_2'\|_{C^{1-\alpha}} \nonumber \\
  & \leq C \|r_1 - r_2 \|_{C^{2-\alpha}},
\end{align}
and argue as \eqref{eq:K0-Hold} in order to get
\begin{align*}
  \|J_{4,1}\|_{C^{1-\alpha}} \leq C \|h\|_{L^\infty} \|r_1 - r_2\|_{C^{2-\alpha}} .
\end{align*}
For $J_{4,2}$, we apply the mean value theorem and the smoothness of $K_1^\alpha$ inside $\mathbb{D}^2$
\begin{align*}
  J_{4,2} = \int_0^1\int_0^{2\pi} \partial_\theta\big(R_2 e^{i\theta}\big) \cdot
  \Big(\nabla_x^2 K^\alpha_1 \big(\kappa R_1 e^{i\theta} + (1-\kappa)R_2 )e^{i\theta},R_1 e^{i\eta}\big)\Big)\cdot
  \big(R_1e^{i\theta} - R_2 e^{i\theta}\big) h(\eta) \,\dd \eta \dd \kappa,
\end{align*}
allowing to get
\begin{align*}
  \|J_{4,2}\|_{C^{1-\alpha}} &\leqslant C \|h\|_{L^\infty} \|R_1(\theta)e^{i\theta} - R_2(\theta)e^{i\theta}\|_{C^{1-\alpha}}\\
  &\leqslant C  \|h\|_{L^\infty} \|r_1-r_2\|_{C^1}.
\end{align*}
The term $J_{4,3}$ can be analogously treated as above to obtain that
\begin{align*}
  \|J_{4,3}\|_{C^{1-\alpha}} &\leqslant C \|h\|_{L^\infty} \|R_1(\eta)e^{i\eta} - R_2(\eta)e^{i\eta}\|_{L^\infty}
  \\
  &\leqslant C \|h\|_{L^\infty} \|r_1 - r_2\|_{L^\infty}.
\end{align*}
Hence  from the preceding estimates we obtain the continuity for $J_4.$
For the second term in $V_3(r,h)$ given by \eqref{Veq-F-3}, we consider
\begin{align*}
  J_5 \triangleq \int_0^{2\pi} K^\alpha_0(R_1(\theta)e^{i\theta},R_1(\eta)e^{i\eta})L_{A_1}(h)\dd\eta
  - \int_0^{2\pi} K^\alpha_0 (R_2(\theta)e^{i\theta},R_2(\eta)e^{i\eta})L_{A_2}(h)\dd\eta,
\end{align*}
where for $j=1,2$,
\begin{align*}
  L_{A_j}(h) \triangleq \partial_\eta\partial_\theta\Big(\frac{h(\eta)R_j(\theta)\sin(\eta-\theta)}{R_j(\eta)}\Big).
\end{align*}
Notice that
\begin{align*}
  J_5 & = \int_0^{2\pi} H(\theta,\eta)\, L_{A_1}(h) \dd \eta +
  \int_0^{2\pi} K^\alpha_0\big(R_2(\theta)e^{i\theta}, R_2(\eta)e^{i\eta} \big) \big(L_{A_1}(h) - L_{A_2}(h)\big) \dd \eta \\
  & \triangleq J_{5,1} + J_{5,2},
\end{align*}
where $H$ is the kernel function defined by \eqref{eq:kernelH}.
For $J_{5,1}$, it can be estimated following the same lines as  $J_2$, leading to
\begin{align*}
  \|J_{5,1}\|_{C^{1-\alpha}} & \leqslant \Big\|\int_0^{2\pi} H(\theta,\eta)\partial_\eta\left( \frac{h(\eta) e^{i\eta}}{R_1(\eta)}\right)
  \dd \eta\Big\|_{C^{1-\alpha}} \|\partial_\theta\big(R_1(\theta)e^{i\theta}\big)\|_{C^{1-\alpha}} \\
  & \leqslant C \|h\|_{C^1} \|r_1 -r_2\|_{C^1}.
\end{align*}
For the term $J_{5,2}$, by using \eqref{eq:ted-identity}, \eqref{eq:pa-R1-R2} and Lemma \ref{lem:reg-1}, we similarly get
\begin{align*}
   \|J_{5,2}\|_{C^{1-\alpha}}
  & \lesssim \Big\|\int_0^{2\pi} K^\alpha_0\big(R_2(\theta)e^{i\theta} - R_2(\eta)e^{i\eta} \big)
 \partial_\eta\left( \frac{h(\eta) e^{i\eta}}{R_1(\eta)}\right) \dd \eta\Big\|_{C^{1-\alpha}}
  \|\partial_\theta\big(R_1e^{i\theta} - R_2e^{i\theta}\big) \|_{C^{1-\alpha}} \\
  &\quad  + C \Big\|\int_0^{2\pi} K^\alpha_0\big(R_2(\theta)e^{i\theta} - R_2(\eta)e^{i\eta} \big)
  \partial_\eta\left[\left( \frac{h(\eta) e^{i\eta}}{R_1(\eta)}\right)-\left( \frac{h(\eta) e^{i\eta}}{R_2(\eta)}\right) \right] \dd \eta\Big\|_{C^{1-\alpha}} \\
  & \leqslant C \|h\|_{C^1} \|r_1 -r_2\|_{C^{2-\alpha}}.
\end{align*}
Hence, based the the preceding results, we claim  that $V_3(r,h)$ as a mapping from $X_m$ to $\mathcal{L}(X_m, Y_m)$ is continuous.
\\
Now we shall  investigate the continuity estimate of the last term   $ V_4(r,h)$. In view of \eqref{eq:ted-identity} and \eqref{Veq-F-4},
we need to consider
\begin{align*}
  J_6 \triangleq & \, \int_0^{2\pi} \widetilde{H}_1(\theta,\eta) \big(i\partial_\eta(R_1(\eta)e^{i\eta})\big) \dd\eta
  \cdot \Big(\partial_\theta \big(R_1(\theta) e^{i\theta}\big)\Big) \\
  & - \int_0^{2\pi} \widetilde{H}_2(\theta,\eta) \big(i\partial_\eta(R_2(\eta)e^{i\eta})\big) \dd\eta
  \cdot \Big(\partial_\theta \big(R_2(\theta) e^{i\theta}\big)\Big) ,
\end{align*}
with $R_j(\theta) = \sqrt{b^2 + 2 r_j(\theta)}$, $j=1,2$ and
\begin{align*}
  \widetilde{H}_j(\theta,\eta) \triangleq \nabla_x K_0^\alpha\big(R_j(\theta)e^{i\theta} - R_j(\eta) e^{i\eta}\big)\cdot
  \left(e^{i\theta} \frac{h(\theta)}{R_j(\theta)}-e^{i\eta} \frac{h(\eta)}{R_j(\eta)}\right).
\end{align*}
Observe that
\begin{align}\label{eq:J6-dec}
  J_6 & = \int_0^{2\pi} \big(\widetilde{H}_1(\theta,\eta) - \widetilde{H}_2(\theta,\eta) \big)
  \big(i\partial_\eta (R_1(\eta)e^{i\eta}) \big) \dd\eta \cdot \Big( \partial_\theta \big(R_1(\theta)e^{i\theta}\big) \Big) \nonumber \\
  & \quad + \int_0^{2\pi} \widetilde{H}_2(\theta,\eta) \big(i\partial_\eta (R_1(\eta)e^{i\eta})
  - i \partial_\eta (R_2(\eta) e^{i\eta}) \big) \dd\eta \cdot \Big( \partial_\theta \big(R_1(\theta)e^{i\theta}\big) \Big) \nonumber \\
  & \quad + \int_0^{2\pi}  \widetilde{H}_2(\theta,\eta)
  \big(i\partial_\eta (R_1(\eta)e^{i\eta}) \big) \dd\eta \cdot \Big( \partial_\theta \big(R_1(\theta)e^{i\theta}\big)
  - \partial_\theta \big( R_2(\theta)e^{i\theta}\big) \Big) \nonumber \\
  & \triangleq \, J_{6,1} + J_{6,2} + J_{6,3} .
\end{align}
Concerning the terms  $J_{6,2}$ and $J_{6,3}$,  we may proceed in a similar way to  $J_{5,1}$ and $J_{5,2}$, and one gets
\begin{align*}
  \|J_{6,2}\|_{C^{1-\alpha}} + \|J_{6,3}\|_{C^{1-\alpha}} \leqslant C \|h\|_{C^1} \|r_1 - r_2\|_{C^{2-\alpha}}.
\end{align*}
For $J_{6,1}$, we first notice by the product law
\begin{align*}
  \|J_{6,1}\|_{C^{1-\alpha}} \leqslant C \Big\|\int_0^{2\pi}
  \big(\widetilde{H}_1(\theta,\eta) -\widetilde{H}_2(\theta,\eta) \big)
  \big(i \partial_\eta (R_1(\eta)e^{i\eta})\big) \dd \eta \Big\|_{C^{1-\alpha}},
\end{align*}
and we shall use Lemma \ref{lem:reg-1} in order to tackle  the above right-hand term.
Direct computations yield  the following decomposition
\begin{align*}
  &\widetilde{H}_1(\theta,\eta) - \widetilde{H}_2(\theta,\eta) \\
  &= \left(\nabla_x K^\alpha_0 \big(R_1(\theta)e^{i\theta},R_1(\eta) e^{i\eta}\big)
  - \nabla_x K^\alpha_0 \big(R_2(\theta)e^{i\theta}, R_2(\eta) e^{i\eta}\big)\right)
  \cdot \left(e^{i\theta} \frac{h(\theta)}{R_1(\theta)}- e^{i\eta} \frac{h(\eta)}{R_1(\eta)}\right)\\
  &\quad + \nabla_x K^\alpha_0 \big(R_2(\theta)e^{i\theta},R_2(\eta) e^{i\eta}\big)
  \cdot \left(e^{i\theta} \frac{h(\theta)}{R_1(\theta)}- e^{i\eta} \frac{h(\eta)}{R_1(\eta)}
  -e^{i\theta} \frac{h(\theta)}{R_2(\theta)}+ e^{i\eta} \frac{h(\eta)}{R_2(\eta)}\right) \\
  & \triangleq S_1(\theta,\eta) + S_2(\theta,\eta).
\end{align*}
Noting that
\begin{align*}
  \nabla_x K_0^\alpha(x_1-y_1)-\nabla_x & K_0^\alpha(x_2-y_2)=-\alpha c_\alpha \frac{x_1-y_1}{|x_1-y_1|^{\alpha+2}}+\alpha c_\alpha\frac{x_2-y_2}{|x_2-y_2|^{\alpha+2}}\\
  =&\alpha c_\alpha
  (x_1-y_1)\left(\frac{|x_1-y_1|^{\alpha+2}-|x_2-y_2|^{\alpha+2}}{|x_1-y_1|^{\alpha+2}|x_2-y_2|^{\alpha+2}}\right)
  +\alpha c_\alpha\frac{(x_2-y_2)-(x_1-y_1)}{|x_2-y_2|^{\alpha+2}},
\end{align*}
and using \eqref{eq:R-lip}, \eqref{eq:h-R-diff}, \eqref{eq:R-diff} and \eqref{ineq:k-alpha},
we have that
\begin{align*}
  |S_1(\theta,\eta)| \leqslant C \frac{\lVert h\rVert_{C^1}\lVert r_1-r_2\rVert_{C^1}}{\big|\sin \frac{\theta-\eta}{2}\big|^{\alpha}},
  \quad \forall\theta\neq\eta\in \TT.
\end{align*}
Using mean value theorem
and the $2\pi$-periodicity of $R_j$ and $h$, we find
\begin{align}\label{eq:R-h-r-diff}
  \left| \frac{e^{i\theta}h(\theta)}{R_1(\theta)}-  \frac{e^{i\eta}h(\eta)}{R_1(\eta)}- \frac{e^{i\theta}h(\theta)}{R_2(\theta)}
  + \frac{e^{i\eta}h(\eta)}{R_2(\eta)}\right|\leqslant  C\lVert h\rVert_{C^1}\lVert r_1-r_2\rVert_{C^1}\left|\sin \tfrac{\theta-\eta}{2}\right|.
\end{align}
According to \eqref{eq:K0-d2} and \eqref{eq:R-h-r-diff}, we obtain
\begin{align*}
  |S_2 (\theta,\eta)|\leqslant C \frac{\lVert r_1-r_2\rVert_{C^1}\lVert h\rVert_{C^1}}
  {\big|\sin \frac{\theta-\eta}{2}\big|^\alpha},
  \quad \forall\theta\neq\eta\in \TT.
\end{align*}
Consequently,
\begin{align}\label{ineq:J-4}
  |\widetilde{H}_1(\theta,\eta) -\widetilde{H}_2(\theta,\eta) |\leqslant C \frac{\lVert r_1-r_2\rVert_{C^1}\lVert h\rVert_{C^1}}{\big|\sin \frac{\theta-\eta}{2}\big|^{\alpha}},
  \quad \forall\theta\neq\eta\in \TT.
\end{align}
Now we intend to estimate $\partial_{\theta}S_1 (\theta,\eta)$ and $\partial_\theta S_2 (r,\theta)$.
We see that
\begin{align*}
  \partial_\theta S_1(\theta,\eta) & =  \bigg(\nabla_x^2 K^\alpha_0 \big(R_1(\theta)e^{i\theta} - R_1(\eta) e^{i\eta}\big)
  \cdot \partial_\theta \big(R_1(\theta) e^{i\theta} \big) \\
  & \quad\quad - \nabla_x^2 K^\alpha_0 \big(R_2(\theta)e^{i\theta} - R_2(\eta) e^{i\eta}\big) \cdot \partial_\theta
  \big(R_2(\theta) e^{i\theta} \big)\bigg)
  \cdot \left(e^{i\theta} \frac{h(\theta)}{R_1(\theta)}- e^{i\eta} \frac{h(\eta)}{R_1(\eta)}\right) \\
  & \quad + \left(\nabla_x K^\alpha_0 \big(R_1(\theta)e^{i\theta}-R_1(\eta) e^{i\eta}\big)
  - \nabla_x K^\alpha_0 \big(R_2(\theta)e^{i\theta}- R_2(\eta) e^{i\eta}\big)\right)
  \cdot \partial_\theta\left(e^{i\theta} \frac{h(\theta)}{R_1(\theta)} \right).
\end{align*}
Note that for every $x=(x^1,x^2)$ and $y=(y^1,y^2)$ satisfying $x\neq y$,
\begin{align*}
  \nabla^2_xK_0^{\alpha}(x-y)= \frac{-\alpha c_\alpha}{|x-y|^{\alpha+2}}
  \begin{pmatrix}
	1 & 0\\
	0 & 1
  \end{pmatrix}
  +\frac{\alpha(\alpha+2)c_\alpha}{|x-y|^{\alpha+4}}
  \begin{pmatrix}
	(x^1-y^1)^2 & (x^1-y^1)(x^2-y^2)\\
	(x^1-y^1)(x^2-y^2) & (x^2-y^2)^2
  \end{pmatrix},
\end{align*}
and
\begin{align*}
  |\nabla^3_x K^\alpha_0(x-y)| \leqslant \frac{C }{|x-y|^{\alpha +3}},
\end{align*}
then using \eqref{eq:R-lip} (and its suitable variation), we know that
\begin{equation}\label{eq:K-2-R-bound}
  |\nabla_x^2 K_0^{\alpha}(R(\theta)e^{i\theta} - R(\eta)e^{i\eta})| \leqslant C \big|\sin \tfrac{\theta-\eta}{2}\big|^{-(\alpha+2)},
\end{equation}
and for every $\kappa \in [0,1]$,
\begin{align*}
  \Big|\nabla^3_x K^\alpha_0\Big(\kappa \big( R_1(\theta) e^{i\theta} - R_1(\eta)e^{i\eta}\big)
  + (1-\kappa) \big(R_2(\theta)e^{i\theta} - R_2(\eta)e^{i\eta}\big) \Big)\Big|
  \leqslant C \big|\sin \tfrac{\theta -\eta}{2} \big|^{-(\alpha +3)}.
\end{align*}
Thanks to the above estimates, together with \eqref{eq:R-lip}, \eqref{eq:h-R-diff} and \eqref{eq:R-diff}, we conclude that
\begin{align*}
  \big| \nabla^2_x K_0^\alpha\big(R_1(\theta)e^{i\theta} - R_1(\eta)e^{i\eta}\big)
  - \nabla^2_x K^\alpha_0\big(R_2(\theta)e^{i\theta} - R_2(\eta)e^{i\eta}\big) \big|
  \leqslant C \frac{\lVert r_1-r_2\rVert_{C^1}}{\big|\sin \frac{\theta-\eta}{2}\big|^{\alpha+2}},
\end{align*}
and
\begin{align*}
  |\partial_\theta S_1(\theta,\eta)| \leqslant C
  \frac{\lVert r_1-r_2\rVert_{C^1}\lVert h\rVert_{C^1}}{\big|\sin \frac{\theta-\eta}{2}\big|^{1+\alpha}},
  \quad \forall \theta \neq \eta\in \TT.
\end{align*}
On the other hand, by  direct computations we get
\begin{align*}
  \partial_\theta S_2(\theta,\eta)
  &= \partial_\theta\big(R_2(\theta)e^{i\theta} \big) \cdot
  \nabla_x^2 K^\alpha_0 \big(R_2(\theta)e^{i\theta}-R_2(\eta) e^{i\eta}\big)
  \cdot \left( \tfrac{e^{i\theta}h(\theta)}{R_1(\theta)}- \tfrac{e^{i\eta}h(\eta)}{R_1(\eta)}
  - \tfrac{e^{i\theta}h(\theta)}{R_2(\theta)} + \tfrac{ e^{i\eta}h(\eta)}{R_2(\eta)}\right) \\
  & \quad + \nabla_x K^\alpha_0 \big(R_2(\theta)e^{i\theta}-R_2(\eta) e^{i\eta}\big)
  \cdot \partial_\theta\left( \tfrac{e^{i\theta}h(\theta)}{R_1(\theta)}
  - \tfrac{e^{i\theta}h(\theta)}{R_2(\theta)} \right),
\end{align*}
and
\begin{align}\label{ineq:R-h-d}
  \left|\partial_\theta\left( \tfrac{e^{i\theta}h(\theta)}{R_1(\theta)}- \tfrac{e^{i\theta}h(\theta)}{R_2(\theta)}\right)\right|
  \leqslant C \lVert h\rVert_{C^1}\lVert r_1-r_2\rVert_{C^1}.
\end{align}
Together  with \eqref{eq:R-h-r-diff} and \eqref{eq:K-2-R-bound}, we infer
\begin{align*}
  |\partial_{\theta} S_2(\theta,\eta)|\leqslant C_{\alpha}\frac{\lVert r_1-r_2\rVert_{C^1}\lVert h\rVert_{C^1}}{\big|\sin \frac{\theta-\eta}{2}\big|^{1+\alpha}}.
\end{align*}
It follows that
\begin{align}\label{ineq:J-4-d}
  |\partial_\theta\widetilde{H}_1(\theta,\eta) - \partial_\theta \widetilde{H}_2(\theta,\eta) |
  \leqslant    C_{\alpha}\frac{\lVert r_1-r_2\rVert_{C^1}\lVert h\rVert_{C^1}}{\big|\sin \frac{\theta-\eta}{2}\big|^{1+\alpha}}.
\end{align}
At this stage we can use \eqref{ineq:J-4}, \eqref{ineq:J-4-d} and Lemma \ref{lem:reg-1} to show that
\begin{align*}
  \lVert J_{6,1}\rVert_{C^{1-\alpha}}\leqslant C\lVert h\rVert_{C^1}\lVert r_1-r_2\rVert_{C^1}.
\end{align*}
Hence from \eqref{eq:J6-dec} and the above estimates, we have
\begin{align*}
  \|J_6\|_{C^{1-\alpha}} \leqslant C \|h\|_{C^1} \|r_1 -r_2\|_{C^{2-\alpha}},
\end{align*}
which ensures  that $r\mapsto (h\mapsto V_4(r,h))$ is a continuous mapping from $X_m$ to $\mathcal{L}(X_m, Y_m)$.
\\
In conclusion, we have established that  $r\mapsto \partial_r F(\Omega,r)$
is a continuous mapping from the small ball $B_{\epsilon_0}$ of $X_m$ to $\mathcal{L}(X_m,Y_m)$.
\\
{\bf{(iii)}} We shall compute $\partial_{\Omega}\partial_{r}F(\Omega,r)$ and prove the continuity of this function.
Let $r(\theta)\in B_{\epsilon_0}$ and $h\in X_m$, then in view of \eqref{F-Lin-r} one has
\begin{align*}
  \partial_{\Omega}\partial_rF(\Omega,r)h(\theta)= h'(\theta),
\end{align*}
which is independent of $r$ and $\Omega$. Hence, the continuity of $\partial_{\Omega}\partial_{r}F(\Omega,r)$ is obvious.
This concludes the proof of the Proposition \ref{prop:regularity-F1}.
	
\end{proof}

\section{Spectral study}\label{sec:spect-Stu}
In this section we focus on the spectral study of the linearized operator of $F(\Omega,r)$ around zero, which is denoted by $\partial_r F(\Omega,0)$.
It turns out that only some discrete  values of $\Omega$ are allowed to generate a non-trivial kernel. In addition, we will see that all the spectral properties required by
Theorem \ref{thm:C-R} are satisfied at least  for large symmetry $m$. The main result of this section reads as follows.

\begin{proposition}\label{prop:other-condition}
Let $(\alpha,b,m)$  satisfy one of the cases \eqref{case1}-\eqref{case2}-\eqref{case3}. Then the  following statements hold true.
\begin{enumerate}
\item
The kernel of $\partial_{r}F(\Omega,0)$ in $X_m$ is non-trivial if and only if
$\Omega=\Omega^{\alpha}_{\ell m,b},$ for some $\ell\in \mathbb{N}^{+}$ with
\begin{equation}\label{eq:Omg-al-m}
\begin{split}
  \Omega_{m,b}^\alpha & = -V_1(0)- \alpha_{m,b}  \\
  & \triangleq  2\sum_{ k\geqslant 1}x_{0,k}^{\alpha-2}\frac{J_1^2\big(x_{0,k} b \big)}{J_1^2(x_{0,k} )}
  -  2\sum_{k\geqslant 1}x_{m,k}^{\alpha-2}\frac{J_m^2(x_{m,k}b)}{J_{m+1}^2(x_{m,k})},
\end{split}
\end{equation}
and in this case, it is a one-dimensional vector space in $X_m$ generated by  $\theta\mapsto \cos (\ell m\theta)$.
\item The range of $\partial_{r}F(\Omega^\alpha_{\ell m,b},0)$ is closed in $Y_m$ and is of co-dimension one. It is given by
\begin{align*}
  R\left(\partial_rF(\Omega_{\ell m,b}^{\alpha},0)\right)=\bigg\{
  r\in C^{1-\alpha}(\mathbb{T}): r(\theta)=\sum_{n\geqslant 1\atop n\ne \ell}a_n \sin (nm\theta),\, a_n \in \mathbb{R}
  \bigg\}.
\end{align*}
\item Transversality assumption:
\begin{align*}
  \partial_{\Omega}\partial_rF(\Omega^\alpha_{\ell m,b},0)\big(\cos (\ell m\theta)\big)
  \notin R(\partial_{r}F(\Omega^{\alpha}_{\ell m,b},0)).
\end{align*}
\end{enumerate}
\end{proposition}
As one can easily observe, the proof of Theorem \ref{main-theorem} is a direct consequence of Theorem \ref{thm:C-R}, Proposition \ref{prop:other-condition} and Proposition \ref{prop:regularity-F1}.\\

The proof of Proposition \ref{prop:other-condition} will be done  in several steps through the subsections below.
\subsection{Analysis of the linear frequencies}
Before proceeding forward, we collect some properties on the asymptotic behavior of the sequence $\{\Omega_{m,b}^\alpha\}_{m\geq 1}$
with respect to $m$ and $\alpha$.
\begin{lemma}\label{lem:Omg-m}
We have the following results.
\begin{enumerate}
\item  Let $(\alpha,b,m)$  satisfy one of cases \eqref{case1}-\eqref{case2}-\eqref{case3}
with $m^*=m^*(\alpha,b) \in \NN^+$ $\big($a rough bound is $m^* \leq \frac{1}{\log b} \big(\log \frac{1-\alpha}{1-\frac{\alpha}{2}- (e\log b)^{-1}}\big)\big)$ and $\alpha^*=\alpha^*(b)>0$ a small number.
Then the map $m\mapsto \Omega_{m,b}^\alpha$ is strictly increasing. In addition, for any $m\geq 1$ and $b\in (0,1)$,
\begin{align*}
  \lim_{\alpha\rightarrow 0} \Omega_{m,b}^\alpha = \frac{m-1+ b^{2m}}{2m} ,
\end{align*}
and
\begin{align}\label{eq:Omeg-alp->1}
  \lim_{\alpha\rightarrow 1} \Omega_{m,b}^\alpha = \frac{2}{\pi b}\sum_{k=1}^{m-1}\frac{1}{2k+1}+\frac{2}{\pi}
  \int_{0}^{\infty}\left(\frac{I^2_1(b\rho)K_0(\rho)}{I_0(\rho)}+\frac{I^2_m(b\rho)K_m(\rho)}{I_m(\rho)}\right)\dd\rho.
\end{align}
\item For $\alpha,b\in (0,1)$ fixed and $m\in\NN$ large enough,
\begin{align}\label{eq:alpha_m^1_bound}
  \alpha_{m,b} = \frac{2^{\alpha-1}\Gamma(1-\alpha)}{b^{\alpha}\Gamma^2(1-\frac{\alpha}{2})}m^{\alpha-1}+O\left(\tfrac{1}{m^{3-\alpha}}\right).
 \end{align}
 \end{enumerate}

\end{lemma}

\begin{proof}[Proof of Lemma \ref{lem:Omg-m}]

\textbf{(i)}
By using $\eqref{eq:Omg-al-m}$ combined with Sneddon's formula \eqref{eq:sned-form} and \eqref{def:J2}, we have
\begin{align}\label{V-1-transfer}
   -V_1(0) = \frac{\Gamma(1-\alpha)\Gamma(1+\frac{\alpha}{2})}{b^\alpha 2^{1-\alpha}\Gamma^2(1-\frac{\alpha}{2})
  \Gamma(2-\frac{\alpha}{2})} + \frac{2}{\pi}\sin (\tfrac{\alpha\pi}{2})
  \int_0^{\infty}\rho^{\alpha-1}\frac{I^2_1(b\rho)K_{0}(\rho)}{I_0(\rho)}\dd\rho,
\end{align}
and
\begin{equation}\label{eq:alp-m-exp}
\begin{split}
  \alpha_{m,b} & = \frac{2}{\pi} \sin\big(\tfrac{\alpha\pi}{2}\big)
  \int_0^\infty \rho^{\alpha-1} I_m(b\rho) K_m(b\rho) \dd \rho
  -\frac{2}{\pi}\sin (\tfrac{\alpha\pi}{2})\int_0^{\infty}\rho^{\alpha-1}\frac{I^2_m(b\rho)K_{m}(\rho)}{I_m(\rho)}\dd\rho \\
  & = \frac{\Gamma(1-\alpha)\Gamma(m+\frac{\alpha}{2})}{b^\alpha 2^{1-\alpha}\Gamma^2(1-\frac{\alpha}{2})
  \Gamma(m+1-\frac{\alpha}{2})}-\frac{2}{\pi}\sin (\tfrac{\alpha\pi}{2})\int_0^{\infty}\rho^{\alpha-1}\frac{I^2_m(b\rho)K_{m}(\rho)}{I_m(\rho)}\dd\rho \\
  & \triangleq \alpha_{m,b}^{(1)} - \alpha_{m,b}^{(2)},
\end{split}
\end{equation}
where $I_m$ and $K_m$ are the modified Bessel functions introduced in Section \ref{Sect-Bess}.\\
In order to show that $\Omega_{m,b}^\alpha$ is strictly increasing in $m$,
we shall analyze the monotonicity of the sequence $\{\alpha_{m,b}\}_{m\geq 1}$.
Consider the Wallis quotient defined by
\begin{align*}
  \mathtt{W}_\alpha(m) \triangleq \frac{\Gamma(m+\frac{\alpha}{2})}{\Gamma(m+1-\frac{\alpha}{2})},
\end{align*}
then we easily see that
\begin{align}\label{eq:alp-m1}
  \alpha_{m,b}^{(1)} = \frac{2^{\alpha-1}\Gamma(1-\alpha)}{b^{\alpha}\Gamma^2(1-\frac{\alpha}{2})}\mathtt{W}_\alpha(m).
 \end{align}
Straightforward computation based on the identity $\Gamma(1+x)=x\Gamma(x)$ allows us to get
\begin{align}\label{difference}
  \mathtt{W}_\alpha(m+1)-\mathtt{W}_\alpha(m)=-\frac{1-\alpha}{1+m-\frac\alpha2}\mathtt{W}_\alpha(m).
\end{align}
In particular this implies that $\{\mathtt{W}_\alpha(m)\}_{m\geqslant 1}$ and $\{\alpha^{(1)}_{m,b}\}_{m\geqslant 1}$
are strictly decreasing.
Now let us move to the analysis  of $\alpha_{m,b}^{(2)}$.
Recalling from \eqref{I-m-q} that
\begin{equation}\label{eq:Im}
  I_m(z)=\sum_{n=0}^\infty \frac{\left(\frac{1}{2}z\right)^{m+2n}}{n! (m+n)!},
\end{equation}
and (e.g. see 6.22 (5) of \cite{Wat66})
\begin{equation*}
  K_m(x)=\int_0^{+\infty}e^{-x\cosh t} \cosh(m t) \dd t>0,\quad \forall x\in \RR,
\end{equation*}
we see that
\begin{equation}\label{eq:Im_b}
  I_m(bx)\leqslant b^mI_m(x),\qquad \forall\, x \geqslant 0, 0<b<1,
\end{equation}
and
\begin{align*}
  I_m(x) K_m(x)\geqslant 0,\quad \forall\, x\geqslant 0.
\end{align*}
A refined version of \eqref{eq:Im_b} is that for all $x> 0$ and $0<b<1$,
\begin{align}\label{eq:Im_b-2}
  I_m(bx) = \big( b^m -  r_{m,b}(x)\big) I_m(x),\quad \textrm{with}\;\; 0\leqslant r_{m,b}(x) \leqslant b^m\min\big\{1,\tfrac{x^2 }{4m}  \big\},
\end{align}
which can be easily seen from the following formula
\begin{align*}
  I_m(bx) - b^m I_m(x) & = - \sum_{n=1}^\infty \big(b^m - b^{m+2n}\big) \frac{(\tfrac{1}{2}x)^{m+2n}}{n! (m+n)!} \\
  & = - \frac{x^2}{4} \sum_{n=0}^\infty \big(b^m- b^{m+2n+2}\big) \frac{(\tfrac{1}{2}x)^{m+2n}}{(n+1)! (m+n+1)!} .
\end{align*}
Hence, we infer that
\begin{align}\label{eq:alp-m2}
  0\leqslant \alpha_{m,b}^{(2)}
  & = \frac{2}{\pi}\sin (\tfrac{\alpha\pi}{2}) b^m \int_0^\infty \rho^{\alpha -1} I_m(b\rho) K_m(\rho) \dd \rho
  - \frac{2}{\pi}\sin (\tfrac{\alpha\pi}{2}) \int_0^\infty \rho^{\alpha-1} r_{m,b}(\rho) I_m(b\rho) K_m(\rho) \dd \rho
  \nonumber \\
  & = \frac{2}{\pi}\sin (\tfrac{\alpha\pi}{2}) b^{2m} \int_0^\infty \rho^{\alpha -1} I_m(\rho) K_m(\rho) \dd \rho \nonumber \\
  &\quad - \frac{2}{\pi}\sin (\tfrac{\alpha\pi}{2}) \int_0^\infty \rho^{\alpha-1} r_{m,b}(\rho) \big(I_m(b\rho) + b^m I_m(\rho) \big)
  K_m(\rho) \dd \rho  \nonumber \\
  & \triangleq \alpha_{m,b}^{(21)} - \alpha_{m,b}^{(22)}.
\end{align}
In view of the formula \eqref{eq:J1-2} and the fact that (see e.g. 8.334 of \cite{GR15})
\begin{align*}
  \Gamma(\tfrac{\alpha}{2})\Gamma(1-\tfrac{\alpha}{2}) = \frac{\pi}{\sin(\pi\tfrac{\alpha}{2})},
\end{align*}
we get
\begin{align}\label{eq:alpha-m-21}
  \alpha_{m,b}^{(2)} \leqslant \alpha_{m,b}^{(21)} & = \frac{2}{\pi} \sin(\tfrac{\alpha \pi}{2}) b^{2m}
  \frac{\Gamma(\tfrac{\alpha}{2})\Gamma(1-\alpha)}{2^{2-\alpha} \Gamma(1-\tfrac{\alpha}{2})}
  \frac{ \Gamma(m+\tfrac{\alpha}{2}) }{\Gamma(m +1 -\tfrac{\alpha}{2}) } \\
  & = \frac{1}{\pi} \sin(\tfrac{\alpha \pi}{2}) b^{2m} \frac{\Gamma(\tfrac{\alpha}{2}) \Gamma(1-\alpha)}{2^{1-\alpha}\Gamma(1-\tfrac{\alpha}{2})}
  \mathtt{W}_\alpha(m) \nonumber \\
  & = \frac{ 2^{\alpha-1} \Gamma(1-\alpha)}{\Gamma^2(1-\tfrac\alpha2)}b^{2m} \mathtt{W}_\alpha(m). \nonumber
\end{align}
For the remainder term $\alpha_{m,b}^{(22)}$,  we can show from the second inequality of  \eqref{eq:Im_b-2} that
\begin{align*}
\forall  \delta\in (0,1),\quad r_{m,b}(\rho) \leqslant b^m (\tfrac{1}{4m})^{\delta/2}\rho^\delta,
\end{align*}
combined with \eqref{eq:J1-2} it yields that for every $\delta \in (0,1-\alpha)$,
\begin{align}\label{eq:alpha-m-22}
  0\leqslant \alpha_{m,b}^{(22)} &
  \leqslant \frac{4}{\pi} \frac{b^{2m}}{(4m)^{\delta/2}} \sin(\tfrac{\alpha \pi}{2})
  \int_0^\infty \rho^{\alpha-1+\delta} I_m(\rho) K_m(\rho) \dd \rho \nonumber \\
  & \leqslant  \frac{4}{\pi} \frac{b^{2m}}{(4m)^{\delta/2}} \sin(\tfrac{\alpha \pi}{2})
  \frac{\Gamma(\tfrac{\alpha +\delta}{2}) \Gamma(m +\tfrac{\alpha+\delta }{2}) \Gamma(1-\alpha-\delta)}
  {2^{2-\alpha-\delta} \Gamma(m+1 -\tfrac{\alpha+\delta}{2}) \Gamma(1-\tfrac{\alpha+\delta}{2})},
\end{align}
which will be useful in the sequel.
Notice that $\{\alpha_{m,b}^{(2)}\}$ is positive and $\{\mathtt{W}_\alpha(m)\}$
is positive and decreasing, then
\begin{align*}
  \alpha_{m+1,b}^{(2)}- \alpha_{m,b}^{(2)} \geqslant -\alpha_{m,b}^{(2)}
  \geqslant - \frac{b^{2m} 2^{\alpha-1} \Gamma(1-\alpha)}{\Gamma^2(1-\tfrac\alpha2)} \mathtt{W}_\alpha(m).
\end{align*}
Thus in combination with \eqref{eq:alp-m-exp}, \eqref{eq:alp-m1} and \eqref{difference} we obtain
\begin{align*}
  \alpha_{m+1,b}-\alpha_{m,b}&\leqslant -\frac{2^{\alpha-1}\Gamma(1-\alpha)}{b^{\alpha}\Gamma^2(1-\frac{\alpha}{2})}
  \frac{1-\alpha}{1+m-\frac\alpha2} \mathtt{W}_\alpha(m)+\frac{b^{2m} 2^{\alpha-1} \Gamma(1-\alpha)}{\Gamma^2(1-\tfrac\alpha2)}  \mathtt{W}_\alpha(m)\\
  &\leqslant\frac{2^{\alpha-1}\Gamma(1-\alpha)}{\Gamma^2(1-\frac{\alpha}{2})}
  \frac{b^{-\alpha}}{1+m-\frac\alpha2} \mathtt{W}_\alpha(m)\Big(-(1-\alpha)  +(1+m-\tfrac{\alpha}{2})b^{2m+\alpha}\Big).
\end{align*}
Now, we intend to check the monotonicity of $\{\alpha_{m,b}\}_{m\geq 1}$ for fixed $\alpha,b$ and for $m$ large enough.\\
Since for fixed $\alpha, b\in (0,1)$, $\lim\limits_{m\rightarrow \infty}(m+1)b^{2m} =0$,
there exists a constant $m^*=m(\alpha,b)\in\NN$ such that $\alpha-1 + (1+m -\tfrac{\alpha}{2}) b^{2m+\alpha}<0$ for every $m\geq m^*$,
thus we have
\begin{align*}
  \forall\, m\geqslant m^*,\quad \alpha_{m+1,b}-\alpha_{m,b}<0.
\end{align*}
On the other hand by using the inequality  $m b^{2m} \leqslant (-\frac{1}{e \log b})b^m$, $\forall m\geqslant 1$, we infer
\begin{align*}
  -(1-\alpha) + (1+m -\tfrac{\alpha}{2})b^{2m+\alpha} < -(1-\alpha) + \big(1-\tfrac{\alpha}{2} -\tfrac{1}{e\log b}\big) b^m,
\end{align*}
then we may choose  $m^* \leqslant \frac{1}{\log b} \big(\log \frac{1-\alpha}{1-\frac{\alpha}{2} - (e\log b)^{-1}}\big)$.
This gives the condition \eqref{case2}.\\
Next, we shall check the condition \eqref{case1} dealing with the monotonicity of
$\{\alpha_{m,b}\}_{m\geqslant1}$ for fixed $\alpha\in (0,1)$ but small $b$. For this aim we introduce the function
\begin{align*}
  \varphi(m) \triangleq (1+m-\tfrac{\alpha}{2})b^{2m+\alpha}, \quad m\in[1,\infty).
\end{align*}
Differentiating in $m$ gives
\begin{align*}
  \varphi^\prime(m)=b^{2m+\alpha}\Big(1+2 (\log b) (1+m-\tfrac{\alpha}{2})\Big).
\end{align*}
We can observe that if $\varphi^\prime(1)\leqslant 0$ then $\varphi^\prime(m)\leqslant 0$ for any $m\geqslant1$.
Let $b^*\in (0,1)$ be such that
\begin{align}\label{Cond-10}
  1+2 (\log b^*) (2-\tfrac{\alpha}{2})\leqslant 0.
\end{align}
Then for any $b\in[0,b^*]$ and $m\geqslant1$,
\begin{align*}
  \varphi^\prime(m)\leqslant 0.
\end{align*}
Consequently, for any $b\in[0,b^*]$
\begin{align*}
  \forall\, m\geqslant 1,\quad \alpha-1+\varphi(m)&\leqslant
  \alpha-1+\varphi(1)=\alpha-1+(2-\tfrac{\alpha}{2})b^{2+\alpha}\\
  &\leqslant \alpha-1+ (2-\tfrac{\alpha}{2})b^2.
\end{align*}
Fix $b^*\triangleq \sqrt{\tfrac{1-\alpha}{2-\frac{\alpha}{2}}}$,
then the condition \eqref{Cond-10} becomes
\begin{align*}
  1+\log(\tfrac{1-\alpha}{2-\frac{\alpha}{2}})(2-\tfrac{\alpha}{2})\leqslant 0.
\end{align*}
This inequality holds true for any $\alpha\in (0,1)$. Finally we get the following
\begin{align*}
  \forall \alpha\in (0,1), b\in (0,b^*), m\geqslant 1,\quad \alpha_{m+1,b}-\alpha_{m,b}<0.
\end{align*}
The last point to discuss is \eqref{case3} related to the monotonicity of $\{\alpha_{m,b}\}_{m\geqslant 1}$ for fixed $b\in ]0,1[$ and small $\alpha$.
In view of \eqref{eq:alp-m1}, \eqref{eq:alp-m2} and \eqref{eq:alpha-m-21}, we have
\begin{align}\label{decompo-V1}
  \alpha_{m,b} = \frac{2^{\alpha-1}\Gamma(1-\alpha)}{\Gamma^2(1-\frac{\alpha}{2})}
  \Big(\frac{1}{b^\alpha} -b^{2m} \Big) \mathtt{W}_\alpha(m) + \alpha_{m,b}^{(22)}.
\end{align}
By choosing $\delta=\tfrac{1}{2}$ in \eqref{eq:alpha-m-22}, we infer that (under the additional constraint $\alpha\in (0,\tfrac{1}{2})$)
\begin{align}\label{Estalpha22}
  |\alpha_{m,b}^{(22)}| \leqslant \frac{4}{\pi} \frac{b^{2m}}{(4m)^{1/4}}
  \sin(\tfrac{\alpha \pi}{2})
  \frac{\Gamma(\tfrac{\alpha}{2}+\tfrac{1}{4}) \Gamma(m +\tfrac{\alpha}{2}+\tfrac{1}{4}) \Gamma(\tfrac{1}{2}-\alpha)}
  {2^{3/2-\alpha} \Gamma(m+\tfrac{3}{4} -\tfrac{\alpha}{2}) \Gamma(\tfrac{3}{4}-\tfrac{\alpha}{2})}\cdot
\end{align}
Applying  Gautschi's inequality
\begin{align*}
  \forall x>0, \forall s\in (0,1),\quad x^{1-s}<\frac{\Gamma(x+1)}{\Gamma(x+s)}<(x+1)^{1-s},
\end{align*}
we deduce that for any $m\geqslant 1$, $a,b\in (0,1)$,
\begin{align*}
  \frac{(1+m)^{a-1}}{m^{b-1}}\leqslant \frac{\Gamma(m +a)}{ \Gamma(m+b) }&\leqslant \frac{m^{a-1}}{(m+1)^{b-1}},
\end{align*}
leading to
\begin{align}\label{frac-gamm-imp}
  \tfrac12m^{a-b}\leqslant \frac{\Gamma(m +a)}{ \Gamma(m+b) }&\leqslant 2m^{a-b}.
\end{align}
It follows from the right-hand side inequality of \eqref{frac-gamm-imp} that
\begin{align}\label{inequ-gamm11}
  \forall m\geqslant1,\forall \alpha\in (0,1),\quad \frac{\Gamma(m +\tfrac{\alpha}{2}+\tfrac{1}{4})}
  { \Gamma(m+\tfrac{3}{4} -\tfrac{\alpha}{2}) }\leqslant 2 m^{\alpha-\frac12}.
\end{align}
Inserting \eqref{inequ-gamm11} into \eqref{Estalpha22} yields for any $m\geqslant1$, $\alpha\in (0,\tfrac{1}{3})$,
\begin{align}\label{Estalpha23}
  \nonumber |\alpha_{m,b}^{(22)}| &\leqslant  \tfrac{4\sqrt2}{\pi} \frac{b^{2m}}{m^{\frac34-\alpha}}
  \sin(\tfrac{\alpha \pi} {2})
  \frac{\Gamma(\tfrac{\alpha}{2}+\tfrac{1}{4}) \Gamma(\tfrac{1}{2}-\alpha)}
  {2^{3/2-\alpha} \Gamma(\tfrac{3}{4}-\tfrac{\alpha}{2})}\\
  &\leqslant \tfrac{C_0 b^{2m}}{m^{3/4-\alpha}} \alpha,
\end{align}
with $C_0$ an absolute constant. On the other hand, by
applying the right-hand side inequality in  \eqref{frac-gamm-imp} we infer
\begin{align}\label{Est-alpham}
  \forall m\geqslant1, \forall \alpha\in (0,1),\quad \mathtt{W}_\alpha(m) \leqslant 2 m^{\alpha-1}.
\end{align}
Combining this inequality with  \eqref{difference} we deduce that
\begin{align}\label{difference-N10}
  \nonumber \forall m\geqslant 1,\forall \alpha\in(0,1),\quad\mathtt{W}_\alpha(m+1)-\mathtt{W}_\alpha(m)&\leqslant -\frac{1-\alpha}{2(1+m-\frac\alpha2)m^{1-\alpha}}\\
  &\leqslant -\frac{1}{8m^{2-\alpha}}\cdot
\end{align}
Putting together \eqref{decompo-V1} with \eqref{Estalpha23}, \eqref{Est-alpham} and \eqref{difference-N10},
allows us to get for any $m\geqslant1, \alpha\in (0,\tfrac{1}{3})$,
\begin{equation*}
  \alpha_{m+1,b}-\alpha_{m,b}  \leqslant  -C_1  m^{\alpha-2}+C_2 b^{2m},
\end{equation*}
for some absolute constants $C_1,C_2>0$. Therefore, for fixed $b\in (0,1)$ we can find $\overline{m}=\overline{m}(b)\in\NN$ such that
\begin{equation}\label{eq:alpham-conv}
  \forall\,\alpha\in (0,\tfrac{1}{3}),\,\forall m\geqslant \overline{m},\quad \alpha_{m+1,b}-\alpha_{m,b}  <  0.
\end{equation}
By virtue of \eqref{decompo-V1} we may write 
\begin{align*}
  \alpha_{m,b} =
  \mathtt{V}_m(\alpha)+ \alpha_{m,b}^{(22)},\qquad \mathtt{V}_m(\alpha)
  \triangleq  \frac{2^{\alpha-1}\Gamma(1-\alpha)}{\Gamma^2(1-\frac{\alpha}{2})}
  \Big(\frac{1}{b^\alpha} -b^{2m} \Big) \mathtt{W}_\alpha(m),
\end{align*}
which can be decomposed as follows
\begin{align}\label{decompo-V01}
 \nonumber  \alpha_{m,b} &= \mathtt{V}_m(0)+
\big( \mathtt{V}_m(\alpha) - \mathtt{V}_m(0)\big)+ \alpha_{m,b}^{(22)}\\
&\triangleq \mathtt{V}_m(0)+\varrho_m(\alpha).
\end{align}
Using the mean value theorem together with \eqref{Estalpha23} yield
\begin{align}\label{alpha-bound}
  \forall \alpha\in[0,\tfrac13],\,\forall m\in[1,\overline{m}+1],\quad |\varrho_m(\alpha)|\leqslant C_3 \alpha,
\end{align}
for some constant $C_3=C_3(b)$ depending only in $b$.
It is obvious that
\begin{align*}
  \mathtt{V}_m(0)=\frac{1-b^{2m}}{2m}.
\end{align*}
According to \cite[Proposition 15]{DHHM}, the sequence $\{\mathtt{V}_m(0)\}_{m\geq 1}$ is strictly decreasing in $m$. Then we can find $\varepsilon=\varepsilon(b)>0$ such that
\begin{align*}
  \forall m\in [1,\overline{m}],\quad \mathtt{V}_{m+1}(0) - \mathtt{V}_m(0)<-\varepsilon.
\end{align*}
Combining this inequality with  \eqref{decompo-V01} and \eqref{alpha-bound} implies
\begin{align*}
  \forall m\in [1,\overline{m}],\quad \alpha_{m+1,b}-\alpha_{m,b}<-\varepsilon+2C_3\alpha.
\end{align*}
By taking $\alpha^*\triangleq \frac{\varepsilon}{2C_3}$, which depends only on $b$,  we get
\begin{align*}
  \forall \alpha\in  (0,\alpha^*),\,\forall m\in [1,\overline{m}],\quad \alpha_{m+1,b}-\alpha_{m,b}<0.
\end{align*}
It follows from \eqref{eq:alpham-conv} that
\begin{equation*}
  \forall \alpha\in (0,\alpha^*),\,\forall m\geqslant 1,\quad \alpha_{m+1,b}-\alpha_{m,b}  <  0.
\end{equation*}
This concludes the proof of the monotonicity.\\[1mm]
Let us move to the computation of $\lim\limits_{\alpha\to0} \Omega_{m,b}^\alpha$. Putting together \eqref{decompo-V1}, \eqref{Estalpha23} yields
\begin{align*}
  \lim_{\alpha\to0}\alpha_{m,b} = \tfrac12
  \Big(1 -b^{2m} \Big) \mathtt{W}_0(m)=\frac{1-b^{2m}}{2m}\cdot
\end{align*}
For $-V_1(0)$ given by \eqref{eq:Omg-al-m}, by virtue of \eqref{V-1-transfer}, one can expect that
\begin{align*}
  \lim_{\alpha\rightarrow 0} \big(-V_1(0)\big)
  = \frac{1}{2} + \frac{2}{\pi}\lim_{\alpha\rightarrow 0} \bigg(\sin (\tfrac{\alpha\pi}{2})
  \int_0^{\infty}\rho^{\alpha-1}\frac{I^2_1(b\rho)K_{0}(\rho)}{I_0(\rho)}\dd\rho\bigg).
\end{align*}
In light of \eqref{eq:Im}, we observe that
\begin{align*}
  0\leqslant \frac{I_1(bx)}{I_0(x)} \leqslant \tfrac{1}{2} b x,\quad \forall x\geqslant 0.
\end{align*}
Then we use Remark \ref{rem:sned-form} and \eqref{eq:Gam-prop} to get
\begin{equation}\label{eq:V10-rem}
\begin{split}
  \Big| \int_0^{\infty}\rho^{\alpha-1}\frac{I^2_1(b\rho)K_{0}(\rho)}{I_0(\rho)}\dd\rho \Big|
  & \leqslant \frac{b}{2} \int_0^\infty \rho^\alpha I_1(b\rho) K_0(\rho) \dd \rho  \\
  & = \frac{b^2 \Gamma^2(1+\tfrac{\alpha}{2})}{2^{2-\alpha} }
  F\big(1+ \tfrac{\alpha}{2},1+\tfrac{\alpha}{2};2;b^2\big).
\end{split}
\end{equation}
Note that the hypergeometric function $F(1+\tfrac{\alpha}{2},1+\tfrac{\alpha}{2};2; b^2)$ is a convergent series for every $\alpha\in (0,1)$
and $b\in (0,1)$, and it converges to $F(1,1;2;b^2)= -\frac{\log(1-b^2)}{b^2} $
as $\alpha\rightarrow 0$ (e.g. see 9.12 of \cite{GR15}). Thus we find
\begin{align*}
  \lim_{\alpha\rightarrow 0} \big(-V_1(0)\big)= \tfrac{1}{2}.
\end{align*}
Hence, we have
$\lim\limits_{\alpha\rightarrow 0}\Omega^\alpha_{m,b}  = \lim\limits_{\alpha\rightarrow 0} \big(-V_1(0) -\alpha_{m,b} \big) = \frac{m-1+b^{2m}}{2m}$ as desired.

Finally, we consider $\lim\limits_{\alpha\rightarrow 1} \Omega_{m,b}^\alpha$.
In view of \eqref{V-1-transfer} and \eqref{eq:alp-m-exp}, and using the monotone convergence theorem, we have
\begin{align*}
  \lim_{\alpha\to 1}\Omega^{\alpha}_{m,b}=&\frac{1}{\pi b}
  \lim_{\alpha\to 1}\Gamma(1-\alpha)\left(\frac{\Gamma(1+\frac{\alpha}{2})}{\Gamma(2-\frac{\alpha}{2})}
  -\frac{\Gamma(m+\frac{\alpha}{2})}{\Gamma(m+1-\frac{\alpha}{2})}\right) \\
  &+\frac{2}{\pi}
  \int_0^\infty \left(\frac{I^2_1(b\rho)K_0(\rho)}{I_0(\rho)}+\frac{I^2_m(b\rho)K_m(\rho)}{I_m(\rho)}\right)\dd\rho.
\end{align*}
As proved by Lemma 3-(1) in \cite{HH15}, we obtain
\begin{equation*}
  \frac{1}{\pi b}\lim_{\alpha\to 1}\Gamma(1-\alpha)
  \left(\frac{\Gamma(1+\frac{\alpha}{2})}{\Gamma(2-\frac{\alpha}{2})}-\frac{\Gamma(m+\frac{\alpha}{2})}{\Gamma(m+1-\frac{\alpha}{2})}\right)
  =\frac{2}{\pi b}\sum_{k=1}^{m-1}\frac{1}{2k+1}.
\end{equation*}
In addition, using Remark \ref{rem:sned-form}, \eqref{eq:Im_b} and \eqref{eq:V10-rem}, we infer that
\begin{align*}
  \int_0^\infty \frac{I^2_1(b\rho)K_0(\rho)}{I_0(\rho)} \dd\rho
  & \leqslant \int_0^\infty \frac{b}{2}\rho I_1(b\rho)K_0(\rho) \dd\rho \\
  & = \frac{b^2\pi }{8}
  F\big(\tfrac{3}{2},\tfrac{3}{2};2;b^2\big) <+\infty,
\end{align*}
and
\begin{align*}
  \int_0^\infty \frac{I^2_m(b\rho)K_m(\rho)}{I_m(\rho)} \dd\rho
  & \leqslant \int_0^\infty b^m I_m(b\rho) K_m(\rho) \dd\rho \\
  & = b^{2m}\frac{\Gamma(\frac{1}{2})\Gamma(m+\frac{1}{2})}{\Gamma(m+1)} F(m+\tfrac{1}{2},\tfrac{1}{2};m+1;b^2) \\
  & \leqslant \tfrac{\sqrt{\pi}}{2} F(1,\tfrac{1}{2};1;b^2) b^{2m} < +\infty.
\end{align*}
Hence, it follows that
\begin{equation*}
  \lim_{\alpha\to 1}\Omega^\alpha_{m,b}=
  \frac{2}{\pi b}\sum_{k=1}^{m-1}\frac{1}{2k+1}+\frac{2}{\pi}
  \int_0^\infty \left(\frac{I^2_1(b\rho)K_0(\rho)}{I_0(\rho)}+\frac{I^2_m(b\rho)K_m(\rho)}{I_m(\rho)}\right)\dd\rho.
\end{equation*}
This ends the proof the first point \textbf{(i)}.
\vskip1mm

\textbf{(ii)}
By virtue of Lemma A.1 in \cite{HHM20} we get
\begin{align*}
\mathtt{W}_m(\alpha)=\tfrac{1}{m^{1-\alpha}}+O\left(\tfrac{1}{m^{3-\alpha}}\right).
\end{align*}
Then we deduce from \eqref{eq:alp-m-exp}, \eqref{eq:alp-m1} and \eqref{eq:alpha-m-21} that for any $\alpha, b\in(0,1),$
\begin{align*}
  \alpha_{m,b} = \frac{2^{\alpha-1}\Gamma(1-\alpha)}{b^{\alpha}\Gamma^2(1-\frac{\alpha}{2})}m^{\alpha-1}+O\left(\tfrac{1}{m^{3-\alpha}}\right)+O(b^{2m}).
\end{align*}
This concludes the proof of  \eqref{eq:alpha_m^1_bound}.
\end{proof}

\subsection{Proof of Proposition \ref{prop:other-condition}}


\textbf{(i)} Consider 
\begin{align}\label{eq:h-exp}
  \theta\in\RR\mapsto h(\theta)=\sum\limits_{n=1}^{\infty}a_n\cos (nm\theta)\in X_m,\quad \textrm{with} \quad a_n\in \RR,
\end{align} 
and let us check that
\begin{align}\label{eq:F-h-general}
  \partial_{r}F(\Omega,0) h(\theta)
  = - \sum_{n=1}^{\infty}a_n \big(\Omega-\Omega^{\alpha}_{nm,b}\big)nm \sin(nm\theta).
\end{align}
Then the result of \textbf{(i)} about the kernel structure follows immediately from this description. To proceed with, we apply first  \eqref{F-Lin-r} with $r=0$ in \eqref{F-Lin-r}, leading to
\begin{align}\label{F-Lin-0}
  \partial_r F(0) h(\theta)& =\big[\Omega+V_1(0)(\theta)\big]
  h'(\theta) + V_2(0)(\theta)h(\theta)+ V_3(0,h)(\theta)+ V_4(0,h)(\theta).
\end{align}	
Below we shall analyse the right-hand side terms of \eqref{F-Lin-0} one by one.
From \eqref{Veq-F-1}, we have
\begin{align}\label{Veq-FL}
  V_1(0)(\theta)& = b^{-1}\int_0^{2\pi} \int_0^b  \nabla_x K^\alpha \big(be^{i\theta},\rho e^{i\eta}\big)
  \cdot e^{i\theta} \rho \,\dd\rho \dd\eta.
\end{align}
Noting that (owing to \eqref{eq:K-der})
\begin{align*}
  \nabla_x K^\alpha \big(be^{i\theta},\rho_2 e^{i\eta}\big)\cdot e^{i\theta}
  = \partial_{\rho_1} G \big(b,\theta,\rho_2,\eta\big),
\end{align*}
which yields in view of \eqref{eq:G-exp} and Fubini's theorem
\begin{align*}
  V_1(0)(\theta)& =b^{-1}\int_0^{2\pi}\int_0^b(\partial_{\rho_1} G)\big(b,\theta,\rho,\eta\big)\rho \,\dd\rho \dd\eta \\
  &=2\pi b^{-1}\sum_{k\geqslant 1}x_{0,k}^{\alpha-1}A_{0,k}^2J_0^\prime\big(b x_{0,k} \big)
  \int_0^b J_0\big(x_{0,k}\rho \big) \rho \dd\rho \\
  &=-2\pi b^{-1}\sum_{k\geqslant 1} x_{0,k}^{\alpha-1}A_{0,k}^2J_1\big(b x_{0,k} \big)
  \int_0^b J_0\big(x_{0,k}\rho \big) \rho \dd\rho.
\end{align*}
In addition, thanks to the identity $\displaystyle{\int_0^a t J_0(t)\dd t=a J_1(a)}$ (see e.g. 6.561 of \cite{GR15}), it follows that
\begin{align*}
  \int_0^b J_0\big(x_{0,k}\rho \big) \rho\, \dd\rho & = \frac{b}{x_{0,k}}J_1\big(x_{0,k}b\big).
\end{align*}
Hence we find that $V_{1}(0)(\theta)$ is independent of $\theta$ and by virtue of \eqref{eq:Ank},
\begin{align}\label{Veq-NF-1}
  V_1(0)&=-2\sum_{ k\geqslant 1}x_{0,k}^{\alpha-2}\frac{J_1^2\big(x_{0,k} b \big)}{J_1^2(x_{0,k} )}\cdot
\end{align}
For $V_2(r)(\theta)$ given by \eqref{Veq-F-2}, we use
$\partial_\theta(R(\theta)e^{i\theta})|_{r=0} = ibe^{i\theta}$, together with \eqref{eq:F(0)}-\eqref{eq:K-itheta}
and \eqref{eq:G-exp} to find
\begin{align*}
  V_2(0)(\theta)= &\, b^{-1} \int_0^{2\pi}\int_0^b \left(\nabla^2_xK_1^\alpha(be^{i\theta},\rho e^{i\eta})
  \cdot e^{i\theta}\right)\cdot (ibe^{i\theta})\rho \,\dd\rho \dd\eta \\
  &+b^{-1}\int_0^{2\pi} \int_{0}^{b}\nabla_xK^\alpha\left(be^{i\theta},\rho e^{i\eta}\right)
  \cdot (e^{i\theta}i) \rho \,\dd\rho \dd\eta \\
  =&\int_0^{2\pi}\int_0^{b}\left(\nabla^2_x K_1^\alpha(be^{i\theta},\rho e^{i\eta})\cdot e^{i\theta}\right)
  \cdot (ie^{i\theta})\rho \,\dd\rho \dd\eta
  + b^{-2}\int_0^{2\pi} \int_0^b \partial_{\theta}G(b,\theta,\rho,\eta) \rho \, \dd\rho \dd\eta \\
  =&\int_0^{2\pi} \int_0^b\left(\nabla^2_xK_1^\alpha(be^{i\theta},\rho e^{i\eta})\cdot e^{i\theta}\right)
  \cdot (ie^{i\theta})\rho \,\dd\rho \dd\eta.
\end{align*}
By virtue of the splitting  \eqref{eq:Kalp} we may write
\begin{align*}
  \nabla_x^2K_1^\alpha = \nabla_x(\nabla_x K^\alpha - \nabla_xK_0^\alpha) = \nabla_x^2 K^\alpha+\nabla_x\nabla_yK_0^\alpha.
\end{align*}
Applying Gauss-Green theorem  and using straightforward computations we deduce that
\begin{align*}
  & \int_0^{2\pi}\int_0^b \left(\nabla_y\nabla_xK_0^\alpha (be^{i\theta},\rho e^{i\eta})\cdot (ie^{i\theta})\right)
  \cdot e^{i\theta}\rho \dd \rho \dd\eta \\
  & = \iint_{b \DD} \left(\nabla_y\nabla_xK_0^\alpha (be^{i\theta},y)\cdot (ie^{i\theta})\right)
  \cdot e^{i\theta} \dd y  \\
  & =\int_0^{2\pi} \left(\nabla_x K_0^\alpha (be^{i\theta},b e^{i\eta})(be^{i\eta})\cdot (ie^{i\theta})\right)
  \cdot e^{i\theta} \dd\eta \\
  & =b\int_0^{2\pi}\nabla_xK_0^\alpha(be^{i\theta}-be^{i\eta})\cdot e^{i\theta} \sin (\eta-\theta)\dd\eta \\
  & =(-\alpha)c_\alpha b^{-\alpha} \int_0^{2\pi}\frac{1-\cos (\eta-\theta)}{|1-e^{i(\eta-\theta)}|^{\alpha+2}}
  \sin (\eta-\theta) \dd\eta	= 0.
\end{align*}
Next, we intend to compute the following integral
\begin{align*}
  \int_0^{2\pi} \int_0^b\left(\nabla^2_xK^\alpha\big(be^{i\theta},\rho e^{i\eta}\big)\cdot e^{i\theta}\right)
  \cdot (ie^{i\theta})\rho\, \dd\rho \dd\eta.
\end{align*}
Direct computations based on  \eqref{eq:K-der} give that
\begin{align}
  \nabla_y K^\alpha\big(be^{i\theta},\rho e^{i\eta}\big)\cdot (ie^{i\theta})
  & =\sin (\eta-\theta) \partial_{\rho_2}G + \rho^{-1}\cos (\eta-\theta) \partial_{\eta}G, \nonumber \\
  \nabla_x K^\alpha\big(be^{i\theta}, \rho e^{i\eta}\big)\cdot e^{i\eta}
  & = \cos (\eta-\theta) \partial_{\rho_1}G + \sin (\eta-\theta)b^{-1} \partial_{\theta}G, \nonumber \\ 
  \nabla_y K^\alpha\big(be^{i\theta},\rho e^{i\eta}\big)\cdot e^{i\eta}&=\partial_{\rho_2}G(b,\theta,\rho,\eta), \nonumber
\end{align}
and
\begin{equation*}
  \Big(\nabla^2_x K^\alpha (be^{i\theta},\rho e^{i\eta})\cdot
  e^{i\theta}\Big)\cdot(ie^{i\theta})=b^{-1}\partial_{\rho_1}\partial_{\theta}G(b,\theta,\rho,\eta)
  - b^{-2}\partial_{\theta}G(b,\theta,\rho,\eta).
\end{equation*}
Thus, we infer from  \eqref{eq:G-exp} and Fubini's theorem
\begin{align}\label{Veq-NF-2}
  & V_2(0)(\theta) = \int_0^{2\pi}\int_0^{b}\left(\nabla^2_xK^\alpha (be^{i\theta},\rho e^{i\eta})
  \cdot e^{i\theta}\right)\cdot (ie^{i\theta})\rho \dd\rho \dd\eta \nonumber \\	
  & = \int_0^{2\pi}\int_0^{b}\Big(\rho b^{-1}\partial_{\rho_1}
  \partial_{\theta}G(b,\theta,\rho,\eta)-\rho b^{-2}
  \partial_{\theta}G(b,\theta,\rho,\eta) \Big)  \dd\rho \dd\eta \nonumber \\	
  & = \int_0^{2\pi} \int_0^b \bigg(\sum_{n\in\NN\atop k\geqslant 1} x_{n,k}^{\alpha-2} A_{n,k}^2 \rho J_n(x_{n,k}\rho)
  \Big(b^{-1}x_{n,k} J_n'(x_{n,k} b) - b^{-2}J_n(x_{n,k}b) \Big)
  n\,\sin n(\eta-\theta)\bigg) \dd \rho \dd\eta \nonumber \\
  & =0.
\end{align}
The next task is to evaluate the contributions induced by  the nonlocal operators $V_3$ and $V_4$. The computation will be done in a formal way and can be justified by standard approximation arguments that we will omit here.
For $V_3(r,h)$ given by \eqref{Veq-F-3}, it is easy to check
\begin{align*}
  L_A(h)|_{r=0} 
  = -h'(\eta)\cos (\eta-\theta) + h(\eta)\sin (\eta-\theta)
  = - \partial_\eta\big(h(\eta) \cos(\eta-\theta) \big),
\end{align*}
leading to
\begin{align*}
  V_3(0,h)(\theta)= \int_0^{2\pi}\nabla_xK_1^\alpha (be^{i\theta},be^{i\eta})\cdot (ibe^{i\theta})h(\eta) \dd\eta
  -\int_0^{2\pi}K_0^\alpha(be^{i\theta} - be^{i\eta})L_A(h)|_{r=0}(\eta)\dd\eta.
\end{align*}
Due to that
\begin{align*}
  K_0^\alpha(be^{i\theta} - be^{i\eta})=K_0^\alpha(b - be^{i(\eta-\theta)}),\quad
\end{align*}
and
$|1-e^{i\eta}|= \big|2\sin \frac{\eta}{2}\big|$,
we use the integration by parts to infer that
\begin{align*}
  & \int_0^{2\pi}K_0^\alpha(be^{i\theta} - be^{i\eta})L_A(h)|_{r=0}(\eta)\dd\eta \nonumber\\
  &= - \sum_{n\geqslant1} a_n \int_0^{2\pi} K_0^\alpha(b - be^{i(\eta-\theta)})
  \partial_\eta\big(\cos (nm\eta) \cos (\eta-\theta)
  \big) \dd\eta \nonumber \\
  & = - \sum_{n\geqslant1} a_n \int_0^{2\pi} K_0^\alpha(b - be^{i\eta})
  \partial_\eta\big(\cos (nm \eta+ n m\theta) \cos \eta
  \big) \dd\eta \nonumber \\
  & = \sum_{n\geqslant1} a_n \int_0^{2\pi} \partial_\eta \Big(\frac{c_{\alpha}}{b^\alpha (2\sin \frac{\eta}{2})^\alpha} \Big)
  \cos (nm \eta+ n m\theta) \cos \eta \,\dd\eta \nonumber \\
  & = - \sum_{n\geqslant 1} a_n \frac{\alpha c_\alpha}{b^\alpha}\int_0^{2\pi}
  \frac{\cos(nm\theta) \cos(nm\eta) - \sin (nm\theta)\sin (nm\eta)}{|2\sin \frac{\eta}{2}|^{\alpha+2}} \cos \eta \sin \eta\,\dd\eta \\
  & = \sum_{n\geqslant 1} a_n \frac{\alpha c_\alpha}{b^\alpha} \sin (nm\theta)\int_0^{2\pi}
  \frac{ \sin (nm\eta) \cos \eta \sin \eta}{(2\sin \frac{\eta}{2})^{\alpha+2}} \dd\eta.
\end{align*}
Note that $\nabla_x K^\alpha_0(x-y) = -\alpha  c_\alpha \frac{x-y}{|x-y|^{\alpha+2}}$,
and for $e_{nm}(\theta)\triangleq e^{inm\theta}$, we have
\begin{align*}
  & \int_0^{2\pi}\nabla_xK_0^\alpha(be^{i\theta}-be^{i\eta})\cdot (ibe^{i\theta}) h(\eta) \dd\eta \nonumber \\
  = & \sum_{n\geqslant 1} a_n\,\mathrm{Re} \left( \int_0^{2\pi}\nabla_xK_0^\alpha (be^{i\theta} - be^{i\eta})\cdot (ibe^{i\theta})
  e_{nm}(\eta)\dd\eta \right) \nonumber \\
  = & \sum_{n\geqslant 1} a_n\,\mathrm{Re}\left( \frac{\alpha c_\alpha}{b^\alpha} \int_0^{2\pi} \frac{(e^{i\eta}-e^{i\theta} )\cdot
  (i e^{i\theta})}{|1-e^{i(\eta-\theta)}|^{\alpha+2}} e_{nm}(\eta) \dd \eta \right) \nonumber \\
  = & \sum_{n\geqslant 1} a_n\, \mathrm{Re}\left(\frac{\alpha c_\alpha}{b^\alpha} \int_0^{2\pi}
  \frac{\sin (\eta-\theta)}{|1-e^{i(\eta-\theta)}|^{\alpha+2}} e_{nm}(\eta) \dd\eta\right).
\end{align*}
By a change of variables and applying the orthogonality of trigonometric functions, we further deduce that
\begin{align*}
  & \quad \int_0^{2\pi}\nabla_xK_0^\alpha(be^{i\theta}-be^{i\eta})\cdot (ibe^{i\theta}) h(\eta) \dd\eta \nonumber \\
  &  =\sum_{n\geqslant 1} a_n\, \mathrm{Re}\left(\frac{\alpha c_\alpha}{b^\alpha} e_{nm}(\theta)
  \int_0^{2\pi} \frac{\sin \eta}{|1-e^{i\eta}|^{\alpha+2}}  e_{nm}(\eta)\dd\eta \right) \nonumber \\
  &  = \sum_{n\geqslant 1} a_n\, \mathrm{Re} \left(i \frac{\alpha c_\alpha}{b^\alpha} e_{nm}(\theta) \int_0^{2\pi}
  \frac{\sin \eta}{|1-e^{i\eta}|^{\alpha+2}}\sin (nm\eta) \dd\eta \right) \nonumber \\
  &  = -\sum_{n\geqslant 1} a_n\, \frac{\alpha c_\alpha}{b^\alpha} \sin (nm\theta)
  \int_0^{2\pi} \frac{\sin \eta\,\sin(nm\eta)}{|1-e^{i\eta}|^{\alpha+2}} \dd\eta.
\end{align*}
Using \eqref{eq:Ank}, \eqref{eq:G-exp} and \eqref{eq:K-itheta}, one can see that
\begin{align*}
  & \int_0^{2\pi}\nabla_xK^\alpha(be^{i\theta}, be^{i\eta})\cdot \big(ibe^{i\theta}\big) h(\eta) \dd\eta \nonumber\\
  = & \sum_{n\geqslant 1} a_n\,\mathrm{Re}\left(\int_0^{2\pi}\nabla_xK^\alpha(be^{i\theta},be^{i\eta})\cdot
  \big(ibe^{i\theta}\big) e_{nm}(\eta)\dd\eta \right) \nonumber \\
  = & \sum_{n\geqslant 1} a_n\, \mathrm{Re}\left( \int_0^{2\pi}\partial_{\theta}G(b,\theta,b,\eta)e_{nm}(\eta) \dd\eta \right) \nonumber \\
  =& \sum_{n\geqslant 1} a_n\, \mathrm{Re} \bigg( \sum_{\ell\in \mathbb{N}\atop k\geqslant  1} x_{\ell,k}^{\alpha-2} A_{\ell,k}^2 J_\ell^2(x_{\ell,k}b)
  \int_0^{2\pi} \ell \sin \ell(\eta-\theta) e^{inm\eta} \dd\eta \bigg).
\end{align*}
Consequently, we find
\begin{align*}
  & \int_0^{2\pi}\nabla_xK^\alpha(be^{i\theta}, be^{i\eta})\cdot \big(ibe^{i\theta}\big) h(\eta) \dd\eta \nonumber\\
  = & \sum_{n\geqslant 1} a_n\, \mathrm{Re} \bigg( e_{nm}(\theta)\sum_{\ell\in \mathbb{N}\atop k\geqslant 1} x_{\ell,k}^{\alpha-2} A_{\ell,k}^2 J_\ell^2(x_{\ell,k}b)
  \int_0^{2\pi} \ell  \sin (\ell\eta) e^{inm\eta} \dd\eta \bigg) \nonumber \\
  =&  - \sum_{n\geqslant 1} a_n\,n m \,\alpha_{nm,b}\, \sin (nm\theta) ,
\end{align*}
with
\begin{align}\label{def:alp-m}
  \alpha_{m,b} \triangleq 2\sum_{k\geqslant 1}x_{m,k}^{\alpha-2}\frac{J_m^2(x_{m,k}b)}{J_{m+1}^2(x_{m,k})}.
\end{align}
Hence, gathering the preceding  identities yields
\begin{align}\label{Veq-NF-3}
  V_3(0,h)(\theta) = - \sum_{n\geqslant 1} a_n\, \Big( nm\,\alpha_{nm,b} -
  \frac{\alpha c_\alpha}{b^\alpha} \int_0^{2\pi}\frac{\sin(nm\eta) (1-\cos \eta)
  \sin \eta}{|1-e^{i\eta}|^{\alpha+2}}\dd\eta\Big) \sin (nm\theta).
\end{align}
For $V_4(r,h)$ given by \eqref{Veq-F-4}, when $r=0$, one has
\begin{align*}
  \nabla_xK_0^\alpha (be^{i\theta}- be^{i\eta})\cdot e^{i\theta}
  = - \frac{c_{\alpha} \alpha}{b^{\alpha+1}} \frac{1-\cos (\eta-\theta)}{|1-e^{i(\eta-\theta)}|^{\alpha+2}}
  = -\nabla_x K_0^\alpha(be^{i\theta} - be^{i\eta})\cdot e^{i\eta},
\end{align*}
leading to
\begin{align*}
  V_4(0,h) = &-b\int_0^{2\pi}\left(\nabla_xK_0^\alpha(be^{i\theta},be^{i\eta})\cdot
  \big(e^{i\theta} h(\theta) - e^{i\eta} h(\eta)\big) \right)\sin (\eta-\theta)\dd\eta \\
  =& \frac{c_\alpha\alpha}{b^\alpha} \int_0^{2\pi}\left(\frac{1-\cos (\eta-\theta)}{|1-e^{i(\eta-\theta)}|^{\alpha+2}}
  \big(h(\theta) + h(\eta)\big)\right)\sin (\eta-\theta)\dd\eta.
\end{align*}
Recalling the expression of $h$ in \eqref{eq:h-exp}, we infer that
\begin{align}\label{Veq-NF-4}
  V_4(0,h) & = \sum_{n\geqslant 1} a_n \,\mathrm{Re} \bigg( \frac{c_\alpha\alpha}{b^\alpha}
  \int_0^{2\pi}\left(\frac{1-\cos (\eta-\theta)}{|1-e^{i(\eta-\theta)}|^{\alpha+2}}
  (e_{nm}(\theta) + e_{nm}(\eta))\right)\sin (\eta-\theta) \dd\eta \bigg) \nonumber \\
  & = \sum_{n\geqslant 1} a_n \,\mathrm{Re}\left( \frac{c_\alpha \alpha}{b^\alpha} e_{nm}(\theta) \int_0^{2\pi}
  \left( \frac{1-\cos \eta}{|1-e^{i\eta}|^{\alpha+2}} \big(1+e_{nm}(\eta)\big)\right)\sin \eta\, \dd\eta \right) \nonumber \\
  & = \sum_{n\geqslant 1} a_n \,\mathrm{Re} \left(i \frac{c_\alpha\alpha}{b^\alpha} e_{nm}(\theta)
  \int_0^{2\pi} \frac{1-\cos \eta}{|1-e^{i\eta}|^{\alpha+2}}
  \sin (nm\eta)  \sin \eta \,\dd\eta\right) \nonumber \\
  & = - \sum_{n\geqslant 1} a_n \,\frac{\alpha c_\alpha}{b^\alpha} \sin (nm\theta)
  \int_0^{2\pi} \frac{1-\cos \eta}{|1-e^{i\eta}|^{\alpha+2}}
  \sin (nm\eta) \sin \eta \,\dd\eta .
\end{align}
Putting together \eqref{Veq-NF-3} and \eqref{Veq-NF-4} allows to get
\begin{align}\label{eq:V3+V4}
  V_3(0,h) + V_4(0,h) = - \sum_{n\geqslant 1} a_n\, \alpha_{nm,b} \,nm \sin(nm\theta).
\end{align}
Consequently, collecting equalities \eqref{F-Lin-0}, \eqref{Veq-NF-1}, \eqref{Veq-NF-2}, \eqref{eq:V3+V4}, we  obtain
\begin{equation}\label{eq:F-lin}
  \partial_r F(\Omega,0)h(\theta) = - \sum_{n=1}^\infty a_n \big(\Omega+V_1(0)+\alpha_{nm,b}\big) nm \sin(nm\theta).
\end{equation}
In light of Lemma \ref{lem:Omg-m}-(i), the map $m\mapsto \Omega_{m,b}^\alpha = -V_1(0)-\alpha_{m,b}$ is strictly increasing in the considered cases.
Hence, the kernel of $\partial_r F(\Omega,0)$ is nontrivial if and only if there exists $\ell\in \NN^+$ such that
\begin{align*}
  \Omega=-V_1(0)-\alpha_{\ell m,b} = \Omega_{\ell m,b}^{\alpha}.
\end{align*}
Moreover, the kernel of $\partial_r F(\Omega_{\ell m,b}^\alpha,0)$  is one-dimensional vector space generated by the function $\theta\mapsto \cos(\ell m \theta)$,
as desired.

\textbf{(ii)}
Now we intend to show that for any $m,\ell \geqslant  1$ the range $R(\partial_rF(\Omega_{\ell m,b}^{\alpha},0))$
coincides with the subspace
\begin{align}\label{range-diff}
  Z_{\ell m}\triangleq \Big\{f\in C^{1-\alpha}(\mathbb{T}):\, f(\theta)=\sum_{n\geqslant 1\atop n\ne \ell}b_n\sin (nm\theta),\, b_n\in \RR,
  \, \theta\in\mathbb{T}\Big\}.
\end{align}
Note that this sub-space is closed and of co-dimension one in the ambient space $Y_m$.
In addition, one can easily deduce from \eqref{eq:F-h-general} the trivial inclusion
$R(\partial_rF(\Omega_{\ell m,b}^\alpha,0))\subset Z_{\ell m}$,
and therefore it remains to show the converse.
For this purpose, let $f\in Z_{\ell m}$, we shall try to find a pre-image $h\in X_m$ satisfying
$\partial_r F(\Omega^\alpha_{\ell m,b},0)(h) = f$.
From the relation \eqref{eq:F-lin}, it reduces to
\begin{align*}
  a_n \big(\Omega_{\ell m,b}^\alpha - \Omega_{nm,b}^\alpha \big)n m = b_n,\quad \forall n\geqslant 1, n\neq \ell.
\end{align*}
This uniquely determines the sequence $\{a_n\}_{n\geqslant 1,n\neq \ell}$ with
\begin{align*}
  a_n = \frac{b_n}{nm(\Omega_{\ell m,b}^\alpha - \Omega_{n m,b}^\alpha)},\quad \forall n\geqslant 1, n\neq \ell.
\end{align*}
However, the coefficient $a_\ell$ is free and we can take it to be zero.
Then, for $\theta\mapsto f(\theta) = \sum\limits_{n=1\atop n\ne \ell}^\infty b_n \sin (nm\theta) \in Y_m$, in order to show $h\in X_m$, %
we only need to prove that
\begin{equation*}
  \theta\mapsto \sum_{n=1\atop n\ne \ell}^{\infty}\frac{b_n}{nm(\alpha_{\ell m,b} -\alpha_{nm,b})} \cos(nm\theta) \in C^{2-\alpha}(\TT),
\end{equation*}
or  equivalently
\begin{align*}
  \theta\mapsto H(\theta)\triangleq \sum_{n\geqslant \ell +1}^{\infty}\frac{b_n}{n(\alpha_{\ell m,b} -\alpha_{nm,b})} \cos(nm\theta) \in C^{2-\alpha}(\TT).
\end{align*}
We shall skip the proof of this point because it  is quite similar to  that of Proposition 8-(2) in \cite{HH15}. We use essentially the same arguments together with the asymptotic structure \eqref{eq:alpha_m^1_bound}.

\textbf{(iii)}
According to the continuity	property of the second derivative $\partial_{\Omega}\partial_{r}F$,
the transversality assumption reduces to
\begin{align*}
  \partial_{\Omega}\partial_{r}F(\Omega,0)(h)\left|_{\Omega=\Omega_{\ell m,b}^{\alpha},h=\cos(\ell m\theta)}\right.
  \notin R(\partial_rF(\Omega_{\ell m,b}^{\alpha},0)).
\end{align*}
This is indeed obvious by virtue of \eqref{range-diff}, due to that
\begin{align*}
  \partial_{\Omega}\partial_rF(\Omega_{\ell m,b}^\alpha,0)\cos(\ell m\theta) = -\ell m\sin(\ell m\theta)
  \notin R(\partial_rF(\Omega_{\ell m,b}^{\alpha},0))=Z_{\ell m}.
\end{align*}
This achieves the proof of Proposition \ref{prop:other-condition}.


\section{Appendix}\label{sec:appendix}
We intend to recall  some tools used in the paper and discuss the proofs of some results established before. The first result concerns the classical Crandall-Rabinowitz theorem on the bifurcation from simple eigenvalues which can be stated as follows, see \cite{C-R71}.
\begin{theorem}[Crandall-Rabinowitz theorem]\label{thm:C-R}
Let $X$ and $Y$ be two Banach spaces, $V$ a neighborhood of $0$ in $X$ and let
$F : \RR \times V \to Y$ be with the following  properties:
\begin{enumerate}
\item $F (\lambda, 0) = 0$ for any $\lambda\in \RR$.
\item The partial derivatives $F_\lambda$, $F_x$ and $F_{\lambda x}$ exist and are continuous.
\item $N(\mathcal{L}_0)$ and $Y/R(\mathcal{L}_0)$ are one-dimensional.
\item {\it Transversality assumption}: $F_{tx}(0, 0)x_0 \not\in R(\mathcal{L}_0)$, where
\begin{align*}
  N(\mathcal{L}_0) = span\{x_0\}, \quad \mathcal{L}_0\triangleq \partial_x F(0,0).
\end{align*}
\end{enumerate}
If $Z$ is any complement of $N(\mathcal{L}_0)$ in $X$, then there is a neighborhood $U$ of $(0,0)$ in $\RR \times X$,
an interval $(-a,a)$, and continuous functions $\varphi: (-a,a) \to \RR$, $\psi: (-a,a) \to Z$ such that $\varphi(0) = 0$, $\psi(0) = 0$ and
\begin{align*}
  F^{-1}(0)\cap U=\Big\{\big(\varphi(\xi), \xi x_0+\xi\psi(\xi)\big)\,:\,\vert \xi\vert<a\Big\}
  \cup\Big\{(\lambda,0)\,:\, (\lambda,0)\in U\Big\}.
\end{align*}
\end{theorem}

The next result deals with the asymptotic growth of the normalized eigenfunctions to the spectral Laplacian in bounded smooth domains.
The proof can be deduced from the standard elliptic estimates (see for instance Section 6.3 of \cite{Evans10})
and we here omit the details.
\begin{lemma}\label{lem:phi}
Let $\mathbf{D}$ be a smooth domain and $\{\phi_j,j\geqslant 1\}$ be the orthonormal basis of $L^2(\mathbf{D})$ of eigenfunctions of $-\Delta_{\mathbf{D}}$ satisfying the constraints  \eqref{def:phi-j}. Then we have
\begin{equation*}
  \lVert \phi_{j}\rVert_{H^{2n}(\mathbf{D})}\leqslant C_n(1+\lambda_j)^n,
  \quad\forall n\in \mathbb{N},
\end{equation*}
where $C_n>0$ is a constant independent of $j$ but may depend on $n$.
\end{lemma}


The current purpose is to prove Lemma $\ref{lem:estimate_H}$ following  the ideas developed in estimating (36) of
Constantin and Ignatova \cite{CI16a}.
\begin{proof}[Proof of Lemma $\ref{lem:estimate_H}$]
We take two points $(\bar{x},y) \in \mathbf{D} \times \mathbf{D}$,
and we consider $x\in B(\bar{x},\tfrac{\delta}{2})$, where $\delta>0$
is defined as
\begin{align*}
  \delta : =
  \begin{cases}
    \frac{d(\bar{x})}{8},\quad & \textrm{if}\;\; |\bar{x}-y|> \frac{d(\bar{x})}{4}, \\
    \frac{d(\bar{x})}{2}, \quad & \textrm{if}\;\; |\bar{x}-y|\leqslant  \frac{d(\bar{x})}{4}.
  \end{cases}
\end{align*}
 Fix $(\bar{x},y)$ and take the function $(t,z)\mapsto h(t,z) = H_\mathbf{D}(t,z,y)$, with $H_{\mathbf{D}}$ the heat kernel defined \mbox{by \eqref{Heat-Str1}.}
Now, we apply Green's formula on the domain $U = B(\bar{x},\delta)\times (0,t)$ to obtain (denote $\mathbf{n}$ as the outer normal vector of $\partial B(\bar{x},\delta)$)
\begin{align*}
  0 & = \int_U \Big[\left((\partial_s-\Delta_z)h(s ,z)\right)G_{t-s}(x-z) + h(s,z)(\partial_s + \Delta_z)G_{t-s}(x-z) \Big] \dd z \dd s \\
  &= h(t,x)-G_t(x-y) + \int_0^t\int_{\partial B(\bar{x},\delta)}\left[\frac{\partial G_{t-s}(x-z)}{\partial \mathbf{n}}h(s,z)-
  \frac{\partial h(s,z)}{\partial \mathbf{n}}G_{t-s}(x-z)\right] \dd \sigma(z) \dd s,
\end{align*}
which leads to
\begin{align*}
  H_\mathbf{D}(t,x,y) = G_t(x-y) - \int_0^t\int_{\partial B(\bar{x},\delta)}
  \left[\frac{\partial G_{t-s}(x-z)}{\partial \mathbf{n}}h(s,z) - \frac{\partial h(s,z)}{\partial \mathbf{n}}G_{t-s}(x-z)\right]\dd \sigma(z) \dd s.
\end{align*}
By differentiating $n$-times in $x$ in the above formula, and using the upper bounds \eqref{eq:H_D-es0}-\eqref{eq:H_D-es1},
we have that for every $0<t\leqslant T_0$,
\begin{align*}
  & \quad |\nabla^n_x H_\mathbf{D}(t,x,y)- \nabla^n_x G_t(x-y)| \\
  &\leqslant  C\int_0^t\int_{\partial B(\bar{x},\delta)} (t-s)^{-\frac{d+n+1}{2}} p_{n+1}\Big(\tfrac{|x-z|}{\sqrt{t-s}}\Big)
  e^{-\frac{|x-z|^2}{4(t-s)}}s^{-\frac{d}{2}}e^{-\frac{|y-z|^2}{C s}}\dd z\dd s \\
  &\quad + C \int_0^{\min\{t, d(y)^2\}}\int_{\partial B(\bar{x},\delta)} (t-s)^{-\frac{d+n}{2}}p_n\Big(\tfrac{|x-z|}{\sqrt{t-s}}\Big)
  e^{-\frac{|x-y|^2}{4(t-s)}}s^{-\frac{d+1}{2}}p_1\Big(\tfrac{|y-z|}{\sqrt{s}}\Big) e^{-\frac{|y-z|^2}{C s}}\dd z\dd s \\
  & \quad + C \int_{\min\{t, d(y)^2\}}^t\int_{\partial B(\bar{x},\delta)}
  (t-s)^{-\frac{d+n}{2}}p_n\Big(\tfrac{|x-z|}{\sqrt{t-s}}\Big)
  e^{-\frac{|x-y|^2}{4(t-s)}}s^{-\frac{d}{2}}\tfrac{1}{d(y)}\tfrac{\phi_1(y)}{|y-z|}e^{-\frac{|y-z|^2}{C s}}\dd z \dd s,
\end{align*}
where $p_k(\xi)$ are  polynomials of degree $k$.
These integrals are all nonsingular.
Note that since in the first integral of the right hand side one has the restriction $|x-z|\ge \frac{\delta}{2}$, then by elementary arguments we deduce that
\begin{align*}
  (t-s)^{-\frac{k}{2}}p_l\big( \tfrac{|x-z|}{\sqrt{t-s}}\big)&\leqslant C |x-z|^{-k} e^{\frac{|x-z|^2}{8(t-s)}}\\
  & \leqslant C \delta^{-k} e^{\frac{|x-z|^2}{8(t-s)}}.
\end{align*}
Similarly, by noting that
\begin{align*}
  |y-z|\geqslant
  \begin{cases}
    |\bar{x}-y|-\delta \geqslant \delta,\quad & \textrm{if}\;\; |\bar{x}-y|> \frac{d(\bar{x})}{4}, \\
    |\bar{x}-z| - |\bar{x}-y| \geqslant \frac{\delta}{2}, \quad & \textrm{if}\;\; |\bar{x}-y|\leqslant \frac{d(\bar{x})}{4},
  \end{cases}
\end{align*}
we obtain
\begin{align*}
s^{-\frac{k}{2}} p_l\big(\tfrac{|y-z|}{\sqrt{s}}\big)&\leqslant C |y-z|^{-k} e^{\frac{|y-z|^2}{2Cs}}\\
&\leqslant C \delta^{-k} e^{\frac{|y-z|^2}{2Cs}}.
\end{align*}
On the other hand, it is known that the first eigenfunction  satisfies
\begin{align*}
  0\leqslant  \phi_1(y) \leqslant C_0 d(y),
\end{align*}
with $C_0>0$ a constant depending on $\mathbf{D}$. Combined with the inequality
\begin{align*}
  \tfrac{|x-y|^2}{t} \leqslant  2\big(\tfrac{|x-z|^2}{t-s} + \tfrac{|y-z|^2}{s}\big),
\end{align*}
it follows that
\begin{align*}
  \tfrac{\phi_1(y)}{|y-z|d(y)}\leqslant C\delta^{-1}.
\end{align*}
Putting together the above estimates and
\begin{align*}
  e^{-\frac{|x-z|^2}{8(t-s)}-\frac{|y-z|^2}{2Cs}}\leqslant e^{-\frac{|x-y|^2}{\widetilde{C}t}},\quad
  \widetilde{C}\triangleq  16 + 4C,
\end{align*}
yields for every $0< t\leqslant \min\{T_0,d(x)^2\}$,
\begin{align*}
  |\nabla_x^n H_\mathbf{D}(x,y,t)-\nabla_x^n G_{t}(x-y)|\leqslant Ce^{-\frac{|x-y|^2}{\widetilde{C}t}}t\delta^{-d-n-2}
  \leqslant Ct^{-\frac{d+n}{2}}e^{-\frac{|x-y|^2}{\widetilde{C} t}},
\end{align*}
where in the last inequality we have used the fact that $\delta \approx d(\bar{x}) \approx d(x)$. Thus the desired estimate follows from the estimate of the heat kernel in the full plane $G_{t},$ which is easy to check.
\end{proof}

The next goal is to establish the proof of Lemma  \ref{lem:es-kern}.
\begin{proof}[Proof of Lemma $\ref{lem:es-kern}$]
  We borrow the idea from the treatment  of (171) in \cite{CI16a}.
For $x$ and $y$ fixed, there exits an open domain $\mathbf{D}_0 \subset\overline{\mathbf{D}_0}\subset\mathbf{D}$ such that $x,y\in \mathbf{D}_0$.
Denote \mbox{by $d_0 \triangleq \min\{\sqrt{T_0},d({\mathbf{D}_0},\partial \mathbf{D})\}>0$.}
Let $\chi\in C^\infty$ be a cutoff function such that $\chi\equiv 1$ on
$\big\{z| d(z,{\mathbf{D}_0}) \leqslant \frac{1}{3} d_0\big\}$ and
$\chi \equiv 0$ on $\big\{z|d(z,{\mathbf{D}_0})  \geqslant \frac{1}{2} d_0\big\}$.
Observe that
$ \widetilde{h}(t,z) = \chi(z)G_t(z-y)$
solves the following equations
\begin{align*}
  (\partial_t-\Delta)\widetilde{h}(t,z) & = -\big[(\Delta\chi(z))G_t(z-y) +2(\nabla\chi(z))\cdot\nabla G_t(z-y)\big]
  \triangleq F(t,z,y), \\
  \widetilde{h}(0,z) & = \chi(z) \delta(z-y), \\
  \widetilde{h}(t,z)|_{\partial \mathbf{D}} & =0,
\end{align*}
with $\delta(\cdot)$ the Dirac mass  centered at $0$.
Thus Duhamel's formula gives
\begin{align*}
  \widetilde{h}(t,z) = e^{t\Delta} \widetilde{h}(0,z) + \int_0^t e^{(t-s)\Delta}F(s,z,y)\dd s,
\end{align*}
which, in combination with $(e^{t\Delta}f)(z) = \int_\mathbf{D} H_\mathbf{D}(t,z,w)f(w)\dd w$ yields
\begin{align*}
  \chi(z)G_t(z-y) = \chi(y)H_\mathbf{D}(t,z,y) + \int_0^t\int_\mathbf{D} H_\mathbf{D}(t-s,z,w) F(s,w,y) \dd w\dd s
\end{align*}
for all $z\in \mathbf{D}$. Noting that $\chi(x)=\chi(y) =1$, and setting $z=x$, we obtain
\begin{equation}\label{eq:H-G-eq-2}
  H_\mathbf{D}(t,x,y) - G_t(x-y) =  \int_0^t\int_\mathbf{D} H_\mathbf{D}(t-s,x,w)
  \big[\Delta\chi(w)G_s(w-y) +2\nabla\chi(w)\cdot\nabla G_s(w-y)\big]\dd w\dd s.
\end{equation}
The integral on the right-hand side is not singular and indeed $C^\infty$-smooth.
In fact, notice that the fact
$\mathrm{supp}\, \nabla\chi \subset \{w| \frac{1}{3}d_0 \leqslant d(w,{\mathbf{D}_0}) <\frac{1}{2}d_0\}$
ensures $\frac{1}{3}d_0\leqslant |x-w|,|y-w|\leqslant C$, and by using the Gaussian upper bounds combined with Lemma \ref{lem:estimate_H},
we have that for every $0\leqslant t\leqslant d_0^2$,
\begin{align*}
  & \quad \left|\nabla_x^k \nabla_y^l \Big( \int_0^t \int_\mathbf{D} H_\mathbf{D}(t-s,x,w) F(s,w,y)\,\dd w \dd s\Big)\right| \\
  & \leqslant  C \int_0^t \int_{\frac{d_0}{3}\leqslant d(w,\mathbf{D}_0)<\frac{d_0}{2}} (t-s)^{-\frac{k+d}{2}}
  e^{-\frac{|x-w|^2}{C(t-s)}}  \Big(|\nabla^l_y G_s(w-y)| + |\nabla^{l+1}_y G_s(w-y)|\Big)\dd w \dd s \leqslant C,
\end{align*}
as desired.
\end{proof}

\vskip5mm


\bibliographystyle{plain}

	\end{document}